\documentclass[11pt]{article}
\usepackage{geometry}
\usepackage{graphicx}
\usepackage{amsfonts,amssymb,amsbsy,amsmath,amssymb,float,subfloat,epstopdf,subfig,amsthm,enumerate}
\usepackage{hyperref,tikz}
\usepackage{mathrsfs}
\newcommand{\intL}{\int\limits }
\newtheorem{thm}{Theorem}

\newtheorem{prop}[thm]{Proposition}
\newtheorem{algorithms}[thm]{Algorithm}
\newtheorem{cor}[thm]{Corollary}
\newtheorem{lem}[thm]{Lemma}
\usepackage{xcolor}
\usepackage[dvipsnames]{xcolor}

\title{Initial condition recovery in nonlinear damped viscous photoacoustic tomography using a convolutional neural network-guided gradient-free optimization framework}
\author{Madhu Gupta\thanks{Department of Mathematics, Indian Institute of Technology, Gandhinagar, Gujarat, India. \url{madhu.gupta@iitgn.ac.in}}
\and Anwesa Dey\thanks{Department of Mathematics, University of Utah, Salt Lake City, Utah, UT, USA. \url{anwesary@math.utah.edu}}
\and Prapti Tala\thanks{Department of Mathematics, Indian Institute of Technology, Gandhinagar, Gujarat, India. \url{prapti.tala@iitgn.ac.in}}
\and Souvik Roy\thanks{Department of Mathematics, The University of Texas at Arlington, Arlington, TX, USA. \url{souvik.roy@uta.edu}}}
\date{}

\begin{document}
\maketitle

\begin{abstract}
Photoacoustic tomography (PAT) is a hybrid imaging modality that combines high optical contrast with high ultrasonic resolution for biomedical imaging applications. In this work, we investigate the inverse problem of recovering the initial pressure distribution from boundary measurements in the presence of nonlinear acoustic propagation and viscous attenuation effects. To model these phenomena more accurately, we consider a nonlinear damped viscoelastic wave equation incorporating spatially varying sound speed, temporal attenuation, and nonlinear propagation mechanisms. 
We first establish the well-posedness of the corresponding forward problem using a Galerkin approximation combined with energy estimates and a fixed-point argument. For the inverse problem, we derive existence, uniqueness, and local uniqueness results under suitable assumptions through a harmonic extension reduction, spectral Laplace transform techniques, and observability estimates. 
To numerically reconstruct the initial pressure field, we develop a hybrid reconstruction framework that combines a convolutional neural network (CNN) with a gradient-free optimization strategy based on the sequential quadratic Hamiltonian (SQH) method derived from Pontryagin’s maximum principle. The CNN is used to generate an informative initial guess, while the SQH framework enforces the governing PDE dynamics during the reconstruction process. Numerical experiments demonstrate that the proposed hybrid strategy significantly improves reconstruction quality, contrast, and robustness compared to standalone time-reversal and CNN-based approaches.
\end{abstract}

\vspace{2mm} 

\section{Introduction}\label{sec:intro}

Optical tomography (OT) is a noninvasive imaging technique that utilizes near--infrared (NIR) light to probe biological tissues and recover their internal optical characteristics. In particular, it aims to estimate spatial distributions of absorption and scattering coefficients, which are key indicators in medical diagnostics. Since cancerous tissues often exhibit reduced scattering compared to healthy ones, precise reconstruction of these optical parameters is critical for early detection and accurate characterization of malignancies. However, the associated inverse problem is highly ill-posed, making reconstructions sensitive to noise and typically resulting in low spatial resolution.

To mitigate these challenges, hybrid and multi-physics imaging strategies have been introduced, combining optical measurements with additional physical modalities. Examples include fluorescence molecular tomography, confocal diffuse tomography, various forms of photoacoustic tomography, and ultrasound-modulated optical tomography. These approaches enhance stability and resolution by leveraging complementary physical information \cite{gupta2020sparse, gupta2021sparsity, adesokan2018nonlinear, roy2018optimization}.  Among them, photoacoustic tomography (PAT)~\cite{Gao12,scherzer15,wang09} has gained particular prominence due to its ability to achieve high-resolution imaging at clinically relevant depths, making it especially valuable in biomedical applications.

In PAT, short pulses of NIR light are delivered to the boundary of a biological specimen. As light propagates within the tissue, part of its energy is absorbed, leading to localized temperature increases and subsequent thermoelastic expansion, while the remaining photons scatter throughout the medium. When the illumination ceases, the tissue relaxes, generating transient pressure waves through the photoacoustic effect~\cite{bell80}. These acoustic signals propagate through the tissue and are captured by detectors placed at or near the boundary. Notably, acoustic wave propagation occurs on time scales significantly slower than optical transport, enabling effective separation of the two processes. The central inverse problem in PAT is to reconstruct the initial pressure distribution, proportional to the absorbed optical energy, from these boundary measurements (see, e.g., \cite{ammaribjk10,anastasiozmr07,kuchment14book,moonzangerlip19}).

A large body of mathematical work models acoustic propagation in PAT using the standard wave equation with spatially varying sound speed; see, for instance, \cite{moonjop16,hwangjmp24,knoxm20,liuu15,moonsiims23} and the references therein. While such models account for heterogeneity and finite propagation speed, they often neglect attenuation effects inherent to biological media. In practice, acoustic waves experience amplitude decay due to absorption, scattering, and viscous dissipation as they travel through tissue. Furthermore, there is also the prevalence of nonlinear acoustic mechanisms that involve the alteration of sound waves as they travel through a medium, where high intensities or specific material characteristics cause deviations from typical linear behavior \cite{Enflo2002}. Unlike linear waves, nonlinear waves experience asymmetry in their compressions and rarefactions, leading to phenomena such as waveform distortion, generation of harmonics, and even the formation of shock waves. These nonlinear effects are especially pronounced in high-intensity applications, where increased wave amplitudes result in rapid wavefront steepening and the potential for shock wave development. Ignoring these mechanisms can lead to reconstruction artifacts and degradation in image quality. To incorporate attenuation, several modeling frameworks have been proposed in the PAT literature~\cite{Ammari2012,elbauss17,moonip24,kowars12,nachmansw90}, including frequency-dependent power-law models, thermo-viscous formulations, and time-convolution-based approaches. These models introduce additional complexity into the inverse problem and can exacerbate ill-posedness. Furthermore, many of these formulations rely on the simplifying assumption of constant sound speed, which may not adequately capture realistic tissue heterogeneity.

In this setting, nonlinear damped viscous wave (NDVW) equations play a fundamental role in modeling wave propagation and in addressing inverse problems aimed at recovering structural information from scattered wave data. These equations arise in a wide range of applications, including civil and environmental engineering, geophysics, ocean acoustics, and biomedical ultrasonics (see \cite{duck2002nonlinear}). By incorporating viscous damping, frictional dissipation, and nonlinear effects, NDVW models provide a more realistic description of wave behavior in viscoelastic media, thereby enabling more accurate forward simulations and parameter estimation. 
In this work, we focus on a specific NDVW model, given by the Westervelt's wave equation, modeling acoustic wave propagation in PAT.
Specifically, we consider the pressure field \(u(x,t)\) as the solution to the Cauchy problem
\begin{equation}
  u_{tt}(x,t) - c^2(x) \Delta u(x,t) - d(x,t) \mathcal{M}* \Delta u_t = \eta(x,t)\big( u^2 \big)_{tt},~ (x,t) \in \mathbb{R}^n \times (0,T),
  \label{eq:main-pde}
\end{equation}
where $n\leq 3$ denotes the spatial dimension, with the initial data
\begin{equation}\label{eq:waveic}
u(x,0)=u_0(x), \qquad \partial_t u(x,0)=0.
\end{equation}
Here, \(c(\cdot)\) denotes a spatially varying sound speed, $d(x,t) > 0$ is a smooth strictly positive damping coefficient that models cumulative attenuation phenomena such as absorption and viscous losses, and \(\eta(x,t) > 0\) is a smooth, strictly positive nonlinearity coefficient. The non-local attenuation is captured by the term $d(x,t) \mathcal{M}* \Delta u_t$, where $*$ is the Laplace convolution operator:
\[
(\mathcal{M}*f)(t) = \int_0^t \mathcal{M}(t-s)f(s)ds
\]
and $\mathcal{M}$ is the memory kernel.
We assume that  the region of interest \(\Omega\subset\mathbb{R}^n\) is a bounded domain with smooth boundary \(\partial\Omega\) and that the sound speed, the damping and nonlinearity coefficients, and the initial data satisfy the assumptions~(\ref{eq:assumptions}) stated in Section~\ref{sec:analysis}.
The measured data are given by the boundary trace
\[
g(y,t) = u(y,t)\big|_{(y,t)\in \partial\Omega\times[0,\infty)}.
\]
In this work, our focus is to reconstruct the initial pressure distribution \(u_0\), and hence the absorbed optical energy density, inside \(\Omega\) from the knowledge of \(g\).

We further note that when the damping coefficient in~\eqref{eq:main-pde} depends solely on the spatial variable, i.e., $d = d(x)$ and when $\eta=0$, the associated inverse problems have been investigated in several works~\cite{dohknn24,haltmeiern19,homan13,palacios16}. In the setting of multiwave imaging, Homan~\cite{homan13} and Palacios~\cite{palacios16} studied damped wave models with spatially varying attenuation and established reconstruction procedures based on Neumann series expansions under appropriate conditions. For media with variable sound speed $c(\cdot)$, the inverse problem can be effectively addressed using time-reversal techniques or equivalent Neumann series formulations~\cite{hristovakn08,nguyenhkd22,oksanenu14,stefanovu09}. This framework was subsequently generalized in~\cite{haltmeiern19} to incorporate spatially dependent damping. In contrast, when the damping exhibits explicit time dependence along with the nonlinearity, the fundamental time-reversal symmetry of the governing equations breaks down, rendering standard time-reversal-based reconstruction methods inapplicable in a direct manner.

Our first contribution is the mathematical analysis of the nonlinear damped viscoelastic wave (NDVW) equation arising in photoacoustic tomography with attenuation and nonlinear propagation effects. We establish the well-posedness of the corresponding forward problem by proving the existence and uniqueness of solutions in appropriate energy spaces. The analysis is carried out through a Galerkin approximation procedure combined with uniform energy estimates, compactness arguments, and a Banach fixed-point framework for the nonlinear problem. Our second contribution is the development of the theoretical framework for the inverse problem of recovering the initial pressure distribution $u_0(x)$ from boundary observations. In this direction, we establish existence, uniqueness, stability, and local uniqueness results for the reconstruction under suitable assumptions on the coefficients and measurement operators. The analysis combines harmonic extension techniques, spectral Laplace transform arguments, and observability estimates for the associated linearized operator. These results provide a rigorous mathematical foundation for the nonlinear inverse problem and clarify the role of attenuation, memory effects, and nonlinear propagation in the reconstruction process.

While the theoretical analysis establishes the solvability and stability properties of the inverse problem, the practical reconstruction of $u_0$ remains computationally challenging due to the nonlinear, nonlocal, and ill-posed nature of the governing PDE model. This motivates the development of efficient computational reconstruction algorithms capable of incorporating both the underlying PDE structure and data-driven feature extraction. For this purpose, we employ the sequential quadratic Hamiltonian (SQH) algorithm, originally introduced in~\cite{Breit18, Breit19, Breit19a} for nonsmooth PDE-constrained optimal control problems, and adapt it to the present inverse problem. The SQH framework is based on the Pontryagin maximum principle (PMP), a fundamental result in optimal control theory. Its appeal lies in its computational efficiency, robustness, and well-posedness as an iterative scheme. A key feature of SQH is its reliance on pointwise optimization steps, which simplifies implementation and enhances scalability. This approach has recently been applied successfully to tomographic inverse problems in optical imaging~\cite{dey2024,Roy2022,Roy2023,Roy2025}. 

Despite these advantages, the performance of SQH remains sensitive to the choice of initial guess. Standard strategies include initializing with a trivial guess (e.g., zero) or using a preliminary reconstruction obtained via an auxiliary method. It is well known that iterative schemes benefit significantly when the initial iterate captures essential structural features of the true solution. One natural candidate is the time-reversal method, which provides a physically consistent approximation of $u_0$. However, in attenuating media, time-reversal reconstructions suffer from substantial amplitude loss due to damping effects, leading to poor contrast.

Convolutional neural networks (CNNs) have demonstrated remarkable success in solving ill-posed, PDE-constrained inverse problems. Their data-driven nature enables them to learn complex, nonlinear mappings directly from training data, effectively capturing features that are difficult to model analytically. By combining convolutional layers with nonlinear activation functions, CNNs construct hierarchical representations capable of approximating intricate solution structures. Moreover, unlike classical variational approaches that require explicit regularization (e.g., $L^2$, $L^1$, or Tikhonov penalties), CNNs incorporate regularization implicitly through their architecture and training process. Nevertheless, a key limitation of CNN-based methods is their lack of explicit enforcement of physical laws, as well as their dependence on large training datasets. Although physics-informed neural networks (PINNs) aim to address this issue, they often face challenges in balancing data fidelity with strict adherence to governing equations. As a result, purely data-driven reconstructions may exhibit nonphysical artifacts.

To use the complementary strengths of both approaches, we propose a novel hybrid initialization strategy for the SQH algorithm for solving the NDVW inverse problem. Specifically, the initial guess is constructed as a combination of the time-reversal reconstruction and the output of a convolutional neural network (CNN). This hybrid strategy preserves the physically meaningful structural information captured by the time-reversal method while simultaneously benefiting from the enhanced contrast and feature extraction capabilities of the CNN.  In addition, the proposed CNN architecture is specifically designed for the nonlinear PAT reconstruction problem by directly learning the mapping from boundary wave measurements to the initial pressure distribution in the presence of attenuation and nonlinear propagation effects. Unlike conventional reconstruction approaches that rely solely on physics-based inversion or purely data-driven learning, the proposed architecture is trained to capture complex nonlinear relationships and hidden spatial features arising from the NDVW dynamics. The network further serves as a physics-guided data-driven regularizer by generating reconstructions that provide informative structural priors for the subsequent SQH optimization procedure.  As demonstrated in Section~\ref{sec:numericalresult}, this hybrid initialization significantly improves the quality of the final reconstruction, leading to enhanced contrast, sharper interfaces, improved structural resolution, and greater robustness compared to standalone time-reversal or CNN-based reconstructions.

The remainder of the paper is organized as follows. In Section \ref{sec:analysis}, we introduce the nonlinear damped viscoelastic wave (NDVW) equation and establish the well-posedness of the corresponding forward problem. The analysis is carried out using a Galerkin approximation scheme combined with energy estimates for a linearized problem and a Banach fixed-point argument. In Section \ref{sec: Inverse problem theory}, we develop the theoretical framework for the inverse problem of reconstructing the initial pressure distribution $f(x)$ from the boundary measurements $g(x,t)$. In particular, we establish existence, uniqueness, and local uniqueness results through a harmonic extension reduction, a spectral Laplace transform approach, and observability estimates for the associated linearized operator. Section \ref{sec:control} presents the numerical reconstruction framework. The reconstruction problem is formulated as an optimization problem consisting of a least-squares data fidelity term together with an $L^2$--$L^1$ regularization term to promote sparsity in the initial pressure distribution. We establish the existence of an optimal solution and derive the corresponding optimality system using Pontryagin's maximum principle. The resulting sequential quadratic Hamiltonian (SQH) algorithm is initialized using a convolutional neural network (CNN)-based reconstruction to obtain an efficient hybrid reconstruction framework. Numerical experiments demonstrating the effectiveness and robustness of the proposed approach are presented in Section \ref{sec:numericalresult}, where comparisons with the classical time-reversal method are also provided. We end with a section of conclusions.

\section{Analysis of the nonlinear damped wave model}
\label{sec:analysis}
In this section, we determine the well-posedness of the NDVW equation \eqref{eq:main-pde} inside a bounded Lipschitz observation domain $\Omega \subset \mathbb{R}^n$, with $1 \leq n \leq 3$. Consider
\begin{equation}\label{eq: Main PDE}
\begin{aligned}
 &u_{tt} - c^2(x) \Delta u - d(x,t) \mathcal{M}* \Delta u_t = \eta(x,t)\big( u^2 \big)_{tt},~ (x,t) \in \Omega \times (0,T), \\
& u(x,0)=u_0(x), \quad u_t(x,0)=0,~ x \in \Omega.
\end{aligned}
\end{equation}
We make the following assumptions:
\begin{equation}\label{eq:assumptions}
\begin{aligned}
&c \in W^{1,\infty} (\Omega), c^2(x) \geq c_0 > 0,\\
&d \in W^{1,\infty} (0,T; L^{\infty}(\Omega)), ~\text{with}~ d(x,t) \geq d_0 > 0,\\
&\eta \in W^{1,\infty}(0,T; W^{1,\infty}(\Omega)),\\
&\mathcal{M} \in L^1(0,T) \geq 0 \qquad \mathcal M'(t)\le0
\end{aligned}
\end{equation}
For $T>0$, define a Banach space
\[
X_T := \{u \in L^{\infty}(0,T;H^1(\Omega)) \cap  W^{1,\infty}(0,T; L^2(\Omega)) \cap L^2(0,T;H^2(\Omega))\}
\]

endowed with the norm
\[
\|u\|_{X_T}:= \text{sup}_{t\in[0,T]} \left( \|u(t)\|_{H^1(\Omega)} + \|u_t(t)\|_{L^2(\Omega)}\right) + \|u\|_{L^2(0,T;H^2(\Omega))} 
\]
The right-hand side of \eqref{eq: Main PDE} can be reformulated as $\eta(x,t) (u^2)_{tt} = 2\eta(u^2_t + u u_{tt}). $
Thus, \eqref{eq: Main PDE} can be written as 
\[
(1-2\eta u)u_{tt} - c^2(x)\Delta u - d(x,t) \mathcal{M} * \Delta u_t =  2 \eta u_t^2 
\]
with the same data and boundary conditions. 

We have the following result for the existence and uniqueness of solutions of \eqref{eq: Main PDE}.

\begin{thm}\label{th:NDVW}
 Under the Assumptions \eqref{eq:assumptions}, there exists a unique solution to the NDVW equation \eqref{eq: Main PDE} in $X_T$.  
\end{thm}
To prove this theorem, we will first show the existence and uniqueness of solutions of a corresponding linearized equation: Choose a fixed $v \in X_T$ satisfying 
\begin{equation}\label{eq: nondeg}
    0< \alpha\leq  1-2\eta(x,t) v(x,t) \leq \beta <\infty 
\end{equation}
for all $v$ in a ball $B_R$ with radius $R$ such that $B_R \subset X_T$. Consider the following linear problem 
\begin{equation}\label{eq: linear pde}
\begin{aligned}
        (1-2\eta v)u_{tt} - c^2(x)\Delta u - d(x,t) \mathcal{M} * \Delta u_t &=  2 \eta v_t^2 \quad \text{in}~ (0,T) \times \Omega \\
        u(x,0)& = u_0(x) \\
        u_t(x,0) &= 0. \\
    \end{aligned}
\end{equation}
This can be ensured by the smallness assumption of $\|v\|_{L^{\infty}(0,T; L^{\infty}(\Omega))}$ and boundedness of $\eta$.
We remark that the ball $B_R \subset X_T$ does not directly control $\|v\|_{L^{\infty}(0,T;L^{\infty}(\Omega))}$. It controls only $\|v\|_{L^{\infty}(0,T;H^1_0(\Omega))}$,  $\|v_t\|_{L^{\infty}(0,T;L^2(\Omega))}$ and $\|\nabla v_t\|_{L^2(0,T;L^2(\Omega))}$.
Using the Sobolev embedding $H^1_0 (\Omega) \hookrightarrow L^6(\Omega) $ for $n \leq 3$, we get 
\[
\sup_{t \in [0,T]} \|v(t)\|_{L^6(\Omega)} \leq CR
\]
Fix $v \in B_R$, then for every $t \in [0,T]$, we have
\[
v(t) = v(0) + \int_0^t v_s(s) ds,
\]
which gives
\[
\|v(t)\|_{L^6} \leq \|v(0)\|_{L^6} + \int_0^t \|v_s(s)\|_{L^6} ds.
\]
Thus, we have
\begin{equation*}
\begin{aligned}
\int_0^t \|v_s(s)\|_{L^6} ds 
&\leq \left(\int_0^t 1^2 ds \right)^{\frac{1}{2}} \left(\int_0^t \|v_s(s)\|^2_{L^6}\right)^{\frac{1}{2}}\\
&\leq T^{\frac{1}{2}} \|v_s\|_{L^2(0,T;L^6)}\\
&\leq C T^{\frac{1}{2}} \|\nabla v_t\|_{L^2(0,T;L^2)}\\
& \leq CT^{1/2} R.
\end{aligned}
\end{equation*}
From the estimate
\begin{equation*}
\begin{aligned}
\|v(t)\|_{L^6(\Omega)}  &\le
\|v(0)\|_{L^6(\Omega)} + C T^{1/2} R,\\
&\leq C \|v(0)\|_{H^1}+ C T^{1/2} R\\
& \leq C(1+T^{1/2})R,
\end{aligned}
\end{equation*}
we obtain a uniform bound on $v$ in $L^\infty(0,T;L^6(\Omega))$.
Using the Sobolev-Gagliardo-Nirenberg inequality, we obtain 
\[
\|v(t)\|_{L^{\infty}(\Omega)} \leq C \|v(t)\|^{1/2}_{H^2(\Omega)} \|v(t)\|^{1/2}_{L^6(\Omega)}
\]
which gives $v \in L^{\infty}(0,T,L^{\infty})$ as $v \in L^{\infty}(0,T;H^2)$.\\
Further, to ensure the non-degeneracy condition
$0<\alpha\leq 1-2\eta(x,t) v(x,t) \le \beta < \infty$,
Assume $\eta \in L^{\infty}(\Omega \times (0,T))$. Then
we have
\[
2 \eta(x,t) v(x,t)\leq 2 \|\eta\|_{\infty} M
\]
which implies 
\[
1-2 \eta(x,t) v(x,t)\geq 1-2 \|\eta\|_{\infty} M. \\
\]
Choose $\alpha$ such that $1-2 \|\eta\|_{\infty} M > \alpha >0 $.
Thus, 
\[
M \leq \frac{1-\alpha}{2 \|\eta\|_{\infty}}.
\]
Since $M=C(R,T)$ depends on $R$ and on $T$, so choosing $R$ and possibly $T$ small enough gives us
\[
C(R,T) \leq \frac{1-\alpha}{2 \|\eta\|_{\infty}}.
\]

\begin{thm}\label{th:linearpde}
For each $v\in X_T$, there exists a unique solution of the linearized equation \eqref{eq: linear pde} in $X_T$, under the condition \eqref{eq: nondeg}.
\end{thm}
\begin{proof}
Now, we have the following weak form for a.e. $t$ and all $\varphi \in H^1_0(\Omega)$
\begin{equation}\label{eq:Weak form}
    \begin{aligned}
        & \int_{\Omega}(1-2\eta v)\, u_{tt} \, \varphi \,dx + \int_{\Omega} c^2(x)\nabla u \cdot   \nabla\varphi dx  + \int_{\Omega} (\nabla c^2 \cdot \nabla u ) \,\varphi \,dx +\int_{\Omega} d(x,t) \nabla (\mathcal{M} *  u_{t}) \cdot \nabla  \varphi \, dx\\
        & + \int_{\Omega} \nabla d(x,t) \cdot \nabla (\mathcal{M} *  u_{t}) \,\varphi \, dx =  \int_{\Omega} 2 \, \eta \, v_t^2 \, \varphi \, dx
    \end{aligned}
\end{equation}
Let $\{\phi_k\}_{k=1}^{\infty}$ be eigenfunctions of Dirichlet Laplacian, i.e.,
\begin{equation} \label{eq:Galerkin}
    \begin{aligned}
    -\Delta \phi_k &= \lambda_k \phi_k \quad \text{in}~ \Omega,    \\
    \phi_k  & =0 \,\,\,\qquad \text{on}~ \partial \Omega,
    \end{aligned}
\end{equation}
with $0< \lambda_1\leq \lambda_2 \leq \hdots$. 
For $N \in \mathbb{N}$, define the finite dimensional subspace 
\[
V_N := \text{span} \{\phi_1,\phi_2,\dots, \phi_N\} \subset V,
\]
and let $u_N$ be an approximate solution of \eqref{eq: linear pde} of the form
\[
u_N(x,t) = \sum_{i=1}^N q_k^N(t) \phi_k(x)
\]
where the scaler functions $q_k^N$ are determined by requiring that the weak form of \eqref{eq: linear pde} holds for all test functions $\varphi \in V_N$.  This provides an \textit{N}-dimensional system of second-order ODE for $q^N_k$, with initial conditions obtained by projecting $u_0$ onto $V_N$. Thus, we obtain a finite-dimensional Volterra integro-differential system, which has a unique local solution in time by standard ODE theory\cite{conti2014well,landes1980galerkin}.  

We now define the energy $E_N(t)$ as follows:
\begin{equation}
    \begin{aligned}
        E_N(t) & := \frac{1}{2} \int_{\Omega} (1-2\eta v) |u_{N,t}|^2 dx +\frac{1}{2} \int_{\Omega} c^2(x) |\nabla u_{N}|^2 dx 
    \end{aligned}
\end{equation}
\begin{equation}
    \begin{aligned}
        \mathcal{E}_N(t) & := \frac{1}{2} \int_{\Omega} (1-2\eta v) |u_{N,t}|^2 dx +\frac{1}{2} \int_{\Omega} c^2(x) |\nabla u_{N}|^2 dx + \frac{1}{2} \int_0^t \int_{\Omega} d(x,s) \mathcal{M}(t-s)|\nabla (u_{N,t})(x,s)|^2 dx\, ds
    \end{aligned}
\end{equation}
Using \eqref{eq: nondeg} and positivity of $d$, this is equivalent to 
\[
 \mathcal{E}_N(t) \approx \|u_{N,t}\|^2_{L^2} + \|\nabla u_N(t)\|_{L^2}^2 + \int_0^t \|\nabla(\mathcal{M} * u_{N,t})(s) \|^2_{L^2}ds 
\]
Observe, $E_N(t) \leq \mathcal{E_N}(t)$ as $d > 0$. Now consider $\varphi = u_{N,t}(t)$ in \eqref{eq:Weak form}, we obtain
\begin{equation}\label{eq: Galerkin_Solution}
    \begin{aligned}
       & \int_{\Omega}(1-2\eta v)\, u_{N,tt} u_{N,t} \,dx + \int_{\Omega} c^2(x)\nabla u_N \cdot \nabla u_{N,t}  dx  + \int_{\Omega} (\nabla c^2 \cdot \nabla u_N ) \,u_{N,t}(t) \,dx \\
        &+\int_{\Omega} d(x,t) \nabla (\mathcal{M} *  u_{N,t}) \cdot \nabla  u_{N,t} \, dx + \int_{\Omega} \nabla d(x,t) \cdot \nabla (\mathcal{M} *  u_{N,t}) \,u_{N,t}\, dx =  \int_{\Omega} 2 \, \eta \, v_t^2 \, u_{N,t} \, dx
    \end{aligned}
\end{equation}
Note that
\begin{equation}
    \begin{aligned}
       \int_{\Omega} (1-2 \eta v)u_{N,tt} u_{N,t} dx & = \frac{1}{2} \frac{d}{dt} \int_{\Omega} (1-2 \eta v) |u_{N,t}|^2 dx - \frac{1}{2} \int_{\Omega} (1-2\eta v)_t |u_{N,t}|^2 dx \\
       \int_{\Omega} c^2(x) \nabla u_N \cdot \nabla u_{N,t} dx & = \frac{1}{2} \frac{d}{dt} \int_{\Omega} c^2|\nabla u_N|^2 dx 
       \end{aligned}
\end{equation}
Substituting these values in equation \eqref{eq: Galerkin_Solution}, we obtain:
\begin{equation}
    \begin{aligned}
      & \frac{d}{dt} E_N(t) -\frac{1}{2} \int_{\Omega} (1-2 \eta v)_t |u_{N,t}|^2 +\int_{\Omega} (\nabla c^2 \cdot \nabla u_N ) \,u_{N,t}(t) \,dx  +\int_{\Omega} \nabla d(x,t) \cdot \nabla (\mathcal{M} *  u_{N,t}) \,u_{N,t}\, dx\\
      & +\int_{\Omega} d(x,t) \nabla (\mathcal{M} *  u_{N,t}) \cdot \nabla  u_{N,t} \, dx =\int_{\Omega} 2 \, \eta \, v_t^2 \, u_{N,t} \, dx\\
    \end{aligned}
\end{equation}
Hence,
\begin{equation}
    \begin{aligned}
       \frac{d}{dt} E_N(t) &= \frac{1}{2} \int_{\Omega} (1-2 \eta v)_t |u_{N,t}|^2 -\int_{\Omega} (\nabla c^2 \cdot \nabla u_N ) \,u_{N,t}(t) \,dx-\int_{\Omega} \nabla d(x,t) \cdot \nabla (\mathcal{M} *  u_{N,t}) \,u_{N,t}\, dx\\  
       &  -\int_{\Omega}  d(x,t)  \nabla (\mathcal{M} *  u_{N,t}) \cdot \nabla u_{N,t}\, dx+\int_{\Omega} 2 \, \eta \, v_t^2 \, u_{N,t} \, dx\\
    \end{aligned}
\end{equation}
Next, we estimate all non-energy terms. Let $(1-2\eta v)_t \in L^{\infty}(0,T; L^{\infty}(\Omega))$
\begin{equation}
    \begin{aligned}
        \left |\frac{1}{2} \int_{\Omega} (1-2 \eta v)_t |u_{N,t}|^2 \right | & \leq \frac{1}{2} \|(1-2 \eta v)_t\|_{L^{\infty}} \|u_{N,t}\|^2_{L^2} \\
        & \leq C_1 \|u_{N,t}\|^2_{L^2} \\
        & \leq C_1 {E}_N(t)
    \end{aligned}
\end{equation}
Assume $\nabla c^2 \in L^{\infty} (\Omega)$. Then \\
\begin{equation}
    \begin{aligned}
        \left |-\int_{\Omega} (\nabla c^2 \cdot \nabla u_N ) \,u_{N,t}(t) \,dx \right | & \leq \|\nabla c^2\|_{L^{\infty}} \|\nabla u_N\|_{L^2} \|u_{N,t}\|_{L^2}  \\
        & \leq C_2 (\|\nabla u_N\|^2_{L^2} + \|u_{N,t}\|^2_{L^2})\\
        & \leq C_2 {E}_N(t)
        \end{aligned}
\end{equation}
For the memory term, we obtain the following estimate
\begin{equation}
\begin{aligned}
\left|\int_{\Omega} d \nabla (\mathcal{M} * u_{N,t}) \cdot \nabla u_{N,t} \right| 
& \leq \|d\|_{L^{\infty}} \int_0^t \mathcal{M}(t-s) \|\nabla u_{N,t}(s)\|_{L^2(\Omega)} \|\nabla u_{N,t}\|_{L^2(\Omega)} ds \\
& \leq  \|d\|_{L^{\infty}} \|\nabla u_{N,t}\|^2_{L^2(\Omega)} + \frac{1}{2}\|d\|_{L^{\infty}}  \int_0^t \mathcal{M}(t-s) \| \nabla u_{N,t}\|^2_{L^2} 
\end{aligned}
\end{equation}
Assume $\nabla d \in L^{\infty}(\Omega \times (0,T))$, then we have 
\begin{equation}
\begin{aligned}
\left|\int_{\Omega} \nabla d(x,t) \cdot \nabla (\mathcal{M} *  u_{N,t}) \,u_{N,t}\, dx \right| & \leq \|\nabla d\|_{L^{\infty}} \|u_{N,t}(t)\|_{L^2} \int_0^t \mathcal{M}(t-s)\|\nabla u_{N,t}(s)\|_{L^2}ds \\
& \leq \epsilon \|u_{N,t}\|^2_{L^2} + C_{\epsilon} \int_0^t \mathcal{M}(t-s)\|\nabla u_{N,t}(s)\|_{L^2}^2 ds \\
& \leq \epsilon E_N(t) + C_{\epsilon} \int_0^t \mathcal{M}(t-s)\|\nabla u_{N,t}(s)\|_{L^2}^2 ds, \\
\end{aligned}
\end{equation}
where, in the last step, we used the definition of $E_{N,t}$. 
Let $ \eta \in L^{\infty} $ and $v_t \in L^4(\Omega)$, and then using Cauchy-Schwarz and Young's inequality, we get
\begin{equation}
    \begin{aligned}
        \left |\int_{\Omega} 2 \, \eta \, v_t^2 \, u_{N,t} \, dx  \right | &\leq  2\|\eta\|_{L^{\infty}} \|v_t^2\|_{L^2} \|u_{N,t}\|_{L^2} \\
        & \leq  2\|\eta\|_{L^{\infty}} \|v_t\|^2_{L^4} \|u_{N,t}\|_{L^2} \\
        & \leq \epsilon \|u_{N,t}\|^2_{L^2} + \|v_t\|^4_{L^4}C_{\epsilon} 
    \end{aligned}
\end{equation}
for any $\epsilon > 0$. Choosing $\epsilon$ small and using $\|u_t\|^2_{L^2} \leq E_N(t)$ gives
\[
 \left |\int_{\Omega} 2 \, \eta \, v_t^2 \, u_{N,t} \, dx  \right | \leq C_4 {E}_N(t) +F(t)
\]
with $F(t) := C_{\epsilon} \|v_t(t)\|_{L^4}^4 \leq L^1(0,T)$ for fixed $v$. 
Combining all the estimates, we get:
{\begin{equation}
    \begin{aligned}
       \frac{d}{dt} E_N(t) &= \frac{1}{2} \int_{\Omega} (1-2 \eta v)_t |u_{N,t}|^2 -\int_{\Omega} (\nabla c^2 \cdot \nabla u_N ) \,u_{N,t}(t) \,dx-\int_{\Omega} \nabla d(x,t) \cdot \nabla (\mathcal{M} *  u_{N,t}) \,u_{N,t}\, dx\\  
       & -\int_{\Omega}  d(x,t) \cdot \nabla (\mathcal{M} *  u_{N,t}) \cdot \nabla u_{N,t}\, dx +\int_{\Omega} 2 \, \eta \, v_t^2 \, u_{N,t} \, dx\\
       & \leq D_1 E_N(t) + D_2 \|\nabla u_{N,t}\|^2_{L^2} + D_3 \int_0^t \mathcal{M}(t-s) \|\nabla u_{N,t}\|^2_{L^2} + F(t)\\
    \end{aligned}
\end{equation}
Let
\[
H_N=\left[\frac{1}{2} \int_0^t \int_{\Omega} d(x,s) \mathcal{M}(t-s) |\nabla u_{N,t} (x,s)|^2 \right]
\]
Using Leibniz's rule, we obtain
\begin{equation}
  \begin{aligned}
    \frac{dH_N}{dt} &= \frac{1}{2} \mathcal{M}(0) \int_{\Omega} d(x,t) |\nabla u_{N,t} (x,t)|^2 dx + \frac{1}{2}\int_0^t \int_{\Omega}\mathcal{M'}(t-s) d(x,s)|\nabla u_{N,t} (x,s)|^2dxds\\
    & \leq d_1 \left(\int_{\Omega}|\nabla u_{N,t} (x,t)|^2 dx+\int_0^t \int_{\Omega}\mathcal{M'}(t-s) |\nabla u_{N,t} (x,s)|^2dxds \right) \quad \text{as }d \in W^{1,\infty}(0,T; L^{\infty}(\Omega))\\
    &= d_1 \left(||\nabla u_{N,t} (x,t)||_{L^{2}}^2 dx+\int_0^t \int_{\Omega}\mathcal{M'}(t-s) |\nabla u_{N,t} (x,s)|^2dxds\right) 
\end{aligned}  
\end{equation}
%
Finally, we have
\begin{equation}
\begin{aligned}
\frac{d}{dt}\mathcal{E}_N(t)& = \frac{d}{dt}{E}_N(t)+\frac{d}{dt}{H}_N(t)\\
& \leq  D_1 E_N(t) + D_2 \|\nabla u_{N,t}\|^2_{L^2} + D_3 \int_0^t \mathcal{M}(t-s) \|\nabla u_{N,t}\|^2_{L^2} + F(t)+ \int_0^t \int_{\Omega}\mathcal{M'}(t-s) |\nabla u_{N,t} (x,s)|^2dxds\\
&\leq D_1 E_N(t) + D_2 \|\nabla u_{N,t}\|^2_{L^2} + D_3 \int_0^t \mathcal{M}(t-s) \|\nabla u_{N,t}\|^2_{L^2} + F(t) \quad \text{(as } \mathcal{M'} \leq 0)
\end{aligned}
\end{equation}}
Since $\mathcal{M} \in L^1(0,T)$ which gives us
\begin{equation}
\begin{aligned}
\frac{d}{dt}\mathcal{E}_N((t)  
&  \leq C \mathcal{E}_N(t) + F(t)
\end{aligned}
\end{equation}
Using Gr\"{o}nwall inequality, we obtain the uniform bound on $E(t)$.
\begin{equation}\label{eq: uniform bound}
    \mathcal{E}_N(t) \leq \left(\mathcal{E}_N(0) +  \int_0^t F(s) ds\right) e^{Ct}, \quad ~t \in [0,T],
\end{equation}
where the constant $C$ is independent of $N$. Since $\mathcal{E}_N(0)$ depends only on $u_0$, $F(t)= C\|v_t(t)\|^4_{L^4(\Omega)} \in L^1(0,T)$ and constants are independent of $N$ . Therefore, we have
\[\text{sup}_{t\in[0,T]} \mathcal{E}_N(t) \leq C_T
\]
Using the above uniform bound \eqref{eq: uniform bound}, we obtain:
\begin{enumerate}[(a)]
    \item $\{u_N\}$ is bounded in $L^{\infty}(0,T;H^1(\Omega))$,
    \item $\{u_{N,t}\}$ is bounded  in $L^{\infty}(0,T;L^2(\Omega))$,
    \item $\{\nabla u_{N,t}\}$ bounded in $L^2(0,T;L^2(\Omega))$.
\end{enumerate}
 From the Galerkin equation, we have 
 \[
 (1-2\eta v_{tt}) u_{N,tt} = c^2 \Delta u_N + d \mathcal{M}*\Delta u_{N,t} + 2 \eta v_t^2
 \]
Note that each term on the right-hand side is bounded in $L^2(0, T; H^{-1}(\Omega))$ because $\Delta u_N, \mathcal{M}* \Delta u_{N,t}$ and $v^2_t$ belong to $H^{-1}(\Omega)$.
Hence $u_{N,tt}$ is bounded in $L^2(0,T;H^{-1}(\Omega))$
uniformly in $N$. 
Using Aubin-Lions lemma \cite{ferreira2023global}, there exists a strongly convergent subsequence (say $u_N$), i.e.,
 \[
 u_N \rightarrow u ~\text{in}~ C([0,T]; L^2(\Omega))
 \]
 with
 \[
 u \in L^{\infty}(0,T; H^1(\Omega)), u_t \in L^{\infty}(0,T;L^2(\Omega)), u_{tt} \in L^2(0,T;H^{-1}(\Omega)). 
 \]
 Passing the limit in the weak formulation gives $u$ as a weak solution of the linear problem for the fixed $v$. 
 Thus for each fixed $v\in B_R \subset X_T$, there exists a unique $u\in X_T$ that solves the linearized equation, and
 \begin{equation}\label{eq:linestimate}
 \|u\|_{X_T} \leq C_T (\|u_0\|_{\mathcal{H}} + \|v_t\|^2_{L^4(0,T;L^4)}), \quad \text{where}~ {\mathcal{H}} = H^1(\Omega) \times L^2(\Omega)  
 \end{equation}
\end{proof}

\begin{proof}[Proof of Theorem \ref{th:NDVW}]
We define the solution operator 
\[
\Phi : X_T \rightarrow X_T, \quad \Phi(v)=u,
\]
where, $u$ is the unique weak solution of \eqref{eq: linear pde} corresponding to $v$. The estimate \eqref{eq:linestimate} shows that $\Phi$ is well-defined.
Next, we will show that $\Phi$ maps $B_R$ into itself.
Choose $R > 0$ such that 
\[
\|(u_0,u_1)\|_{\mathcal{H}} \leq \frac{R}{2C_T}.
\]
and take $T>0$ small enough that
the \eqref{eq: nondeg} holds for all $v \in B_R$ and 
the source term satisfies the following bound $\|v_t\|^2_{L^4(0,T;L^4)} \leq \frac{R}{2C_T}$. This gives us
\[
\|\Phi(v)\|_{X_T} \leq R ~\text{for all}~ v \in B_R.
\]
Hence, $\Phi(B_R) \subset B_R$.

Next, we will prove the contraction property.
Let $v_1,v_2 \in B_R$, denote $u^i = \Phi (v^i)$ and $w=u^1-u^2$. Then $w$ solves a linear equation with coefficients $(1-2\eta v^1)- (1-2\eta v^2)= 2\eta (v^2-v^1)$ and the source term $(v^1_t)^2-(v^2
_t)^2$, i.e.,
\begin{equation}\label{eq: v1-v2}
\begin{aligned}
(1-2\eta v_1) w_{tt} - c^2 \Delta w - d(\mathcal{M} * \Delta w_t) &= 2 \eta (v_{1,t}^2 - v_{2,t}^2) + 2\eta (v_2-v_1) u_{2,tt}        
\end{aligned}
\end{equation}
with initial data:
\[
w(x,0)=0,~ \text{and}~  w_t(x,0)=0
\]
Now, define the energy
\begin{equation}
\begin{aligned}
E_w(t): = \frac{1}{2} \int_{\Omega} (1-2\eta v_1) |w_t|^2 + \frac{1}{2} \int_{\Omega} c^2 |\nabla w|^2
\end{aligned}
\end{equation}
As in the linear case, a similar result can be obtained for $w_t$:
\[
\frac{d}{dt} E_w(t) \leq C E_w(t) + I_1(t) + I_2(t)        
\]
where,
\[
I_1(t) = \int_{\Omega} 2 \eta (v_{1,t}^2 - v_{2,t}^2 ) w_t
\]
\[
I_2(t) = \int_{\Omega} 2 \eta (v_2 - v_1) u_{2,tt} w_t
\]
The coefficient term $(1-2\eta v_1)$ is handled as before and will be absorbed into $CE_w(t)$.
We now have the following estimate for $I_1(t)$:
\begin{equation*}
    \begin{aligned}
        |I_1(t)| &\leq C\|v_{1,t} + v_{2,t}\|_{L^4} \|v_{1,t} - v_{2,t}\|_{L^4} \|w_t\|_{L^2} \\
        & \leq \epsilon \|w_t\|^2_{L^2} + C_{\epsilon} \|v_{1,t} + v_{2,t}\|^2_{L^4} \|v_{1,t} - v_{2,t}\|^2_{L^4}
    \end{aligned}
    \end{equation*}
Since $v_i \in B_R$, we have
\[
\|v_{i,t} \|_{L^{\infty}(0,T;L^2)} \leq R
\]
So,
\[
|I_1(t)| \leq \epsilon \|w_t\|^2_{H^1} + C_{\epsilon}(\|v_1 - v_{2}\|^2_{H^1})
\]
For $I_2$, 
\begin{equation*}
    \begin{aligned}
         |I_2(t)| &\leq 2 \|\eta\|_{L^{\infty}} \|(v_2 - v_1) w_t \|_{H^1_0(\Omega)} \|u_{2,tt}\|_{H^{-1} (\Omega)}\\
         & \leq C \|v_2-v_1\|_{H^1} \|w_t\|_{H^1}\|u_{2,tt}\|_{H^{-1}} \\
         & \leq C \|v_2-v_1\|_{H^1} \|w_t\|_{H^1}
    \end{aligned}
\end{equation*}
Since $u=\phi(v_1)$, so $\|u_{2,tt}\|_{H^{-1}} \leq C$.
Then, applying Young's inequality, we get
\[
|I_2(t)| \leq \epsilon \|w_t(t)\|^2_{H^1} + C \|v_2 - v_1\| 
\]
Substituting these estimates finally gives,
\[
\frac{d}{dt} E_w(t) \leq C_1 E_w(t) + C_2 \|v_2-v_1\|^2_{H^1}
\]
Then using Growwall, 
the energy estimate provides the following inequality: 
\[
\|w\|_{X_T} \leq CT^{\gamma} \|v^1-v^2\|_{X_T}
\]
for some $\gamma > 0$, where smallness comes from the factors $\|v^i\|_{X_T}$ and small time interval $[0,T]$. For $T>0$ sufficiently small, $CT^{\gamma} < 1$, so $\Phi$ is a contraction map on $B_R$. By Banach's fixed-point theorem, $\Phi$ has a unique fixed point such that:
\[
u \in B_R \subset X_T, \Phi (u) =u 
\]
Thus, u is a unique solution of \eqref{eq: Main PDE}.
\end{proof}

\begin{thm}[Stability]\label{th:stab}
Let $u^{(1)}$ and $u^{(2)}$ be two sufficiently regular solutions of
\begin{equation}
u_{tt}-c^2(x)\Delta u
-d(x,t)\int_0^t \mathcal M(t-s)\Delta u_t(s)\,ds
=
\eta(x,t)(u^2)_{tt}
\label{eq:main}
\end{equation}
in $\Omega\times(0,T)$, with boundary conditions
\[
u^{(i)} = g^{(i)}(x,t)
\quad \text{on } \partial\Omega\times(0,T),
\]
and initial conditions
\[
u^{(i)}(x,0)=u_0^{(i)}(x),
\qquad
u_t^{(i)}(x,0)=u_1^{(i)}(x),
\]
for $i=1,2$.
Assume that
\[
c(x)\ge c_0>0,
\qquad
d(x,t)\ge d_0>0,
\]
and
\[
\mathcal M(t)\ge0,
\qquad
\mathcal M'(t)\le0.
\]
Further assume
\[
u^{(i)}\in L^\infty(0,T;H^2(\Omega)),
\qquad
u_t^{(i)}\in L^\infty(0,T;H^1(\Omega)),
\]
and
\[
\|\eta A\|_{L^\infty(\Omega\times(0,T))}<1,
\]
where
\[
A=u^{(1)}+u^{(2)}.
\]
Then the solutions depend continuously on the initial and boundary data. More precisely,
\begin{equation}\label{eq:stability}
\|u^{(1)}-u^{(2)}\|
\le
Ce^{Ct}
\left(
\|u_0^{(1)}-u_0^{(2)}\|_{H^1(\Omega)}
+
\|u_1^{(1)}-u_1^{(2)}\|_{L^2(\Omega)}
+
\|g^{(1)}-g^{(2)}\|_{\partial \Omega}
\right).
\end{equation}
\end{thm}

\begin{proof}
For a nonhomogeneous boundary, we can introduce lifting functions
$w^{(i)}$ satisfying
\[
w^{(i)} = g^{(i)}
\quad \text{on } \partial\Omega\times(0,T).
\]
Define
\[
v^{(i)} = u^{(i)} - w^{(i)}.
\]
Then
\[
v^{(i)}|_{\partial\Omega}=0.
\]
The functions $v^{(i)}$ satisfy
\begin{align}
v_{tt}^{(i)}
-c^2(x)\Delta v^{(i)}
-d(x,t)\int_0^t
\mathcal M(t-s)\Delta v_t^{(i)}(s)\,ds
=
\eta(x,t)\big((v^{(i)}+w^{(i)})^2\big)_{tt}
+F^{(i)},
\label{eq:v-eqn}
\end{align}
where
\[
F^{(i)}
=
-w_{tt}^{(i)}
+c^2(x)\Delta w^{(i)}
+d(x,t)\int_0^t
\mathcal M(t-s)\Delta w_t^{(i)}(s)\,ds.
\]
Define the difference variable
\[
y = v^{(1)} - v^{(2)}
\]
Subtracting \eqref{eq:v-eqn} satisfied by  $v^{(1)}$ and $v^{(2)}$, we get
\begin{align}
y_{tt}
-c^2(x)\Delta y
-d(x,t)\int_0^t
\mathcal M(t-s)\Delta y_t(s)\,ds
=
N(y)+\big(F^{(1)}-F^{(2)}\big),
\label{eq:y}
\end{align}
where
\[
N(y)
=
\eta(x,t)
\partial_{tt}
\left(
(v^{(1)}+w^{(1)})^2
-
(v^{(2)}+w^{(2)})^2
\right).
\]
$A = u^{(1)} + u^{(2)}$ implies 
\[
A = v^{(1)} + v^{(2)} + w^{(1)} + w^{(2)}
\]
and
\[
B
=
w^{(1)}-w^{(2)}.
\]
Since
\[
y=v^{(1)}-v^{(2)},
\]
we obtain
\[
(v^{(1)}+w^{(1)})^2
-
(v^{(2)}+w^{(2)})^2
=
A(y+B).
\]
Hence
\[
N(y)
=
\eta\partial_{tt}(A(y+B)).
\]
Define the energy
\begin{align}
E_y(t)
&=
\frac12
\int_\Omega
|y_t|^2
+
c^2(x)|\nabla y|^2
\,dx +
\frac12
\int_0^t
\mathcal M(t-s)
\|\nabla y_t(s)\|_{L^2(\Omega)}^2\,ds.
\label{eq:energy}
\end{align}
We also define the modified energy as follows:
\begin{align}
\tilde{E}_y(t)
&=
\frac12
\int_\Omega (1-\eta A)
|y_t|^2
+
c^2(x)|\nabla y|^2
\,dx +
\frac12
\int_0^t
\mathcal M(t-s)
\|\nabla y_t(s)\|_{L^2(\Omega)}^2\,ds.
\label{eq:energy2}
\end{align}
Since 
\[
\|\eta A\|_{L^{\infty}} < 1,
\]
there exist positive constants $c_1,c_2$ such that
\[
c_1 E_y(t) \leq \tilde{E}_y(t) \leq c_2 E_y(t),
\]
where
\[
E_y(t) \leq C(\|y_t\|_{L^2(\Omega)}^2
+
\|\nabla y\|_{L^2(\Omega)}^2)
\]
Differentiating \eqref{eq:energy} with respect to time, we obtain
\begin{align}\label{eq:energy_derivative}
\frac{d}{dt}E_y(t)
&=
\int_\Omega y_{tt}y_t\,dx
+
\int_\Omega c^2(x)\nabla y\cdot \nabla y_t\,dx
+
\frac12 \mathcal M(0)\|\nabla y_t(t)\|^2
+
\frac12
\int_0^t
\mathcal M'(t-s)
\|\nabla y_t(s)\|^2\,ds.
\end{align}
Multiplying \eqref{eq:y} by $y_t$ and integrating over $\Omega$, we obtain
\begin{align}
\int_\Omega y_{tt}y_t\,dx
-
\int_\Omega c^2(x)\Delta y\,y_t\,dx
-
\int_\Omega d(x,t)
\left(
\int_0^t
\mathcal M(t-s)\Delta y_t(s)\,ds
\right)&y_t\,dx
=
\int_\Omega N(y)y_t\,dx
\nonumber\\
&+
\int_\Omega (F^{(1)}-F^{(2)})y_t\,dx.
\label{eq:A1}
\end{align}
%
Now,
\[
\int_\Omega y_{tt}y_t\,dx
=
\frac12\frac{d}{dt}\|y_t\|_{L^2(\Omega)}^2.
\]
Further,
\begin{align}
-\int_\Omega c^2(x)\Delta y\,y_t\,dx
&=
\int_\Omega c^2(x)\nabla y\cdot \nabla y_t\,dx
+
\int_\Omega \nabla(c^2(x))\cdot \nabla y\,y_t\,dx
\nonumber\\
&=
\frac12\frac{d}{dt}
\int_\Omega c^2(x)|\nabla y|^2\,dx
+
\int_\Omega \nabla(c^2(x))\cdot \nabla y\,y_t\,dx.
\label{eq:A2}
\end{align}
Since $d \in W^{1,\infty}(0,T; L^{\infty}(\Omega))$, there exists $d_1 > 0$ such that $d(x,t) \le d_1$ for all $(x,t) \in \Omega \times (0,T)$
\begin{align}
-
d(x,t)\int_\Omega
\left(
\int_0^t
\mathcal M(t-s)\Delta y_t(s)\,ds
\right)y_t\,dx
&=
d_0\mathcal M(0)\|\nabla y_t\|^2
+
d_0\int_0^t
\big(-\mathcal M'(t-s)\big)
\|\nabla y_t(s)\|^2\,ds
\nonumber\\
&\quad
-
\frac12\frac{d}{dt}
\int_0^t
\mathcal M(t-s)
\|\nabla y_t(s)\|_{L^2(\Omega)}^2\,ds.
\label{eq:A3}
\end{align}
\begin{align}
-
\int_\Omega d(x,t)
\left(
\int_0^t
\mathcal M(t-s)\Delta y_t(s)\,ds
\right)y_t\,dx
&\geq -
\int_\Omega d_1
\left(
\int_0^t
\mathcal M(t-s)\Delta y_t(s)\,ds
\right)y_t\,dx\\
&= 
-
d_1\int_0^t
\mathcal M(t-s)\left(\int_\Omega \Delta y_t(s) y_t \,dx\right)\,ds\\
&=
d_1\int_0^t
\big(\mathcal M(t-s)\big)
\|\nabla y_t(s)\|^2\,ds
\nonumber\\
\label{eq:A3}
\end{align}
Substituting $\int_\Omega y_{tt}y_t\,dx$ term from \eqref{eq:energy_derivative} and using \eqref{eq:A2} and \eqref{eq:A3} into \eqref{eq:A1}, we obtain
\begin{align}\label{eq: derivative_identity}
\frac{d}{dt}E_y(t)
+
d_1 \int_0^t\mathcal M(t-s)\|\nabla y_t(s)\|^2
-
&\frac{1}{2}\int_0^t
\mathcal M'(t-s)
\|\nabla y_t(s)\|^2\,ds
\leq
\int_\Omega N(y)y_t\,dx\nonumber\\
&+
\int_\Omega (F^{(1)}-F^{(2)})y_t\,dx
-
\int_\Omega
\nabla(c^2(x))\cdot \nabla y\,y_t\,dx + \frac{1}{2}\mathcal{M}(0)\|\nabla y_t(t)\|^2\,
\end{align}
where
\[
E_y(t)
=
\frac12
\int_\Omega
|y_t|^2
+
c^2(x)|\nabla y|^2\,dx
+
\frac12
\int_0^t
\mathcal M(t-s)
\|\nabla y_t(s)\|_{L^2(\Omega)}^2\,ds.
\]
Using the fact $\mathcal{M} \ge 0$ and $\mathcal{M}' \le 0$, we further obtain
\begin{align}\label{eq: derivative_identity2}
\frac{d}{dt}E_y(t)
& \leq 
\int_\Omega N(y)y_t\,dx
+
\int_\Omega (F^{(1)}-F^{(2)})y_t\,dx
-
\int_\Omega
\nabla(c^2(x))\cdot \nabla y\,y_t\,dx + \frac{1}{2}\mathcal{M}(0)\|\nabla y_t(t)\|^2\,
\end{align}
The last two terms can be  estimated as follows:\[\left|\int_\Omega\nabla(c^2(x))\cdot \nabla y\,y_t\,dx\right|\le\|\nabla(c^2)\|_{L^\infty(\Omega)}\|\nabla y\|_{L^2(\Omega)}\|y_t\|_{L^2(\Omega)}.\]
\[\frac{1}{2}\mathcal{M}(0)\|\nabla y_t(t)\|^2  \leq C \]
Applying Young's inequality, for any $\varepsilon>0$,\[\left|\int_\Omega\nabla(c^2(x))\cdot \nabla y\,y_t\,dx\right|\le\varepsilon \|\nabla y\|_{L^2(\Omega)}^2+C_\varepsilon \|y_t\|_{L^2(\Omega)}^2.\]
Hence, this term is controlled by the energy functional $E_y(t)$.
Next, we estimate the nonlinear term. Expanding
\[
N(y)
=
\eta
\Big(
A(y+B)_{tt}
+
2A_t(y+B)_t
+
A_{tt}(y+B)
\Big),
\]
we obtain
\[
\int_\Omega N(y)y_t\,dx
=
I_1+I_2+I_3.
\]
The principal term is
\[
I_1
=
\int_\Omega
\eta A y_{tt}y_t\,dx
+
\int_\Omega
\eta A B_{tt}y_t\,dx.
\]
Using
\[
y_{tt}y_t
=
\frac12\partial_t(|y_t|^2),
\]
we obtain
\begin{align}
\int_\Omega
\eta A y_{tt}y_t\,dx
&=
\frac12
\frac{d}{dt}
\int_\Omega
\eta A |y_t|^2\,dx
\nonumber\\
&\quad
-
\frac12
\int_\Omega
(\eta A)_t |y_t|^2\,dx.
\end{align}
Since
\[
A=u^{(1)}+u^{(2)},
\]
and
\[
u^{(i)}\in L^\infty(0,T;H^2(\Omega)),
\]
Sobolev embedding implies
\[
A,\ A_t,\ A_{tt}\in L^\infty(\Omega\times(0,T)).
\]
Therefore,
\[
\left|
\int_\Omega
(\eta A)_t |y_t|^2\,dx
\right|
\le
C\|y_t\|^2.
\]
The remaining terms are estimated using Hölder's, Poincaré's,
and Young's inequalities:
\begin{align}
\left|
\int_\Omega N(y)y_t\,dx
\right|
&\le
\frac12
\frac{d}{dt}
\int_\Omega
\eta A |y_t|^2\,dx
\nonumber\\
&\quad
+
C
\left(
\|y_t\|^2
+
\|\nabla y\|^2
\right)
\nonumber\\
&\quad
+
C
\left(
\|B\|^2
+
\|B_t\|^2
+
\|B_{tt}\|^2
\right).
\label{eq:nonlinear}
\end{align}
Substituting \eqref{eq:nonlinear} into \eqref{eq: derivative_identity},
we obtain
\begin{align}
\frac{d}{dt}E_y(t)
&\le
\frac12
\frac{d}{dt}
\int_\Omega
\eta A |y_t|^2\,dx
+
CE_y(t)
\nonumber\\
&\quad
+
C
\left(
\|B\|^2
+
\|B_t\|^2
+
\|B_{tt}\|^2
\right)
+
\left|
\int_\Omega
(F^{(1)}-F^{(2)})y_t\,dx
\right|.
\end{align}
Moving the derivative term to the left-hand side gives
\begin{align}
\frac{d}{dt}\widetilde E_y(t) +
&\le C\widetilde E_y(t)
+
C
\left(
\|B\|^2
+
\|B_t\|^2
+
\|B_{tt}\|^2
\right)
+
\left|
\int_\Omega
(F^{(1)}-F^{(2)})y_t\,dx
\right|.
\label{eq:finalenergy}
\end{align}
Finally, using Young's inequality,
\[
\left|
\int_\Omega
(F^{(1)}-F^{(2)})y_t\,dx
\right|
\le
\varepsilon\|y_t\|^2
+
C_\varepsilon
\|F^{(1)}-F^{(2)}\|^2.
\]
Since the lifting operator is continuous,
\[
\|B\|
+
\|B_t\|
+
\|B_{tt}\|
+
\|F^{(1)}-F^{(2)}\|
\le
C
\|g^{(1)}-g^{(2)}\|_{\partial \Omega},
\]
we obtain
\[
\frac{d}{dt}\widetilde E_y(t)
\le
C\widetilde E_y(t)
+
C
\|g^{(1)}-g^{(2)}\|^2_{\partial\Omega}.
\]
Applying Grönwall's inequality yields
\[
\widetilde E_y(t)
\le
Ce^{Ct}
\left(
\widetilde E_y(0)
+
\|g^{(1)}-g^{(2)}\|^2_{\partial\Omega}
\right).
\]
For the forcing term,
\begin{align}
\left|
\int_\Omega
(F^{(1)}-F^{(2)})y_t\,dx
\right| &\le
\|F^{(1)}-F^{(2)}\|
\|y_t\| \\
&\le
\varepsilon\|y_t\|^2
+
C_\varepsilon
\|F^{(1)}-F^{(2)}\|^2.
\end{align}
Using the fact that the lifting operator is continuous, we get
\[
\|F^{(1)}-F^{(2)}\|
\le
C
\|g^{(1)}-g^{(2)}\|_{\partial\Omega}.
\]
Therefore,
\[
E_y'(t)
\le
CE_y(t)
+
C
\|g^{(1)}-g^{(2)}\|_{\partial\Omega}^2.
\]
Applying Grönwall's inequality yields
\[
E_y(t)
\le
Ce^{Ct}
\left(
E_y(0)
+
\|g^{(1)}-g^{(2)}\|_{\partial\Omega}^2
\right).
\]
Note
\[
E_y(0)
\le C
(\|u_0^{(1)}-u_0^{(2)}\|_{H^1(\Omega)}^2
+
\|u_1^{(1)}-u_1^{(2)}\|_{L^2(\Omega)}^2).
\]
Therefore,
\[
E_y(t)
\le
Ce^{Ct}
\Big(
\|u_0^{(1)}-u_0^{(2)}\|_{H^1(\Omega)}^2
+
\|u_1^{(1)}-u_1^{(2)}\|_{L^2(\Omega)}^2
+
\|g^{(1)}-g^{(2)}\|_{\partial \Omega}^2
\Big).
\]
Since
\[
E_y(t)
\sim
\|y_t(t)\|_{L^2(\Omega)}^2
+
\|y(t)\|_{H^1(\Omega)}^2,
\]
We conclude that
\begin{align}
\|u_t^{(1)}-u_t^{(2)}\|_{L^2(\Omega)}^2
+
\|u^{(1)}-u^{(2)}\|_{H^1(\Omega)}^2 
\le
Ce^{Ct}
\Big(
\|u_0^{(1)}-u_0^{(2)}\|_{H^1(\Omega)}^2
+
\|u_1^{(1)}-u_1^{(2)}\|_{L^2(\Omega)}^2
+
\|g^{(1)}-g^{(2)}\|_{\partial \Omega}^2
\Big).
\end{align}
which is equivalent to 
\begin{align}\label{eq: sqrt}
\left(\|u_t^{(1)}-u_t^{(2)}\|_{L^2(\Omega)}^2
+
\|u^{(1)}-u^{(2)}\|_{H^1(\Omega)}^2\right)^{\frac{1}{2}}
\le
Ce^{Ct}
\Big(
\|u_0^{(1)}-u_0^{(2)}\|_{H^1(\Omega)}^2
+
\|u_1^{(1)}-u_1^{(2)}\|_{L^2(\Omega)}^2
+
\|g^{(1)}-g^{(2)}\|_{\partial \Omega}^2
\Big)^{\frac{1}{2}}
\end{align}
Next, using the inequality, $a+b \le \sqrt{2}(a^2 + b^2)^{1/2} $, we obtain
\begin{align}\label{eq: ab}
\|u_t^{(1)}-u_t^{(2)}\|_{L^2(\Omega)}
+
\|u^{(1)}-u^{(2)}\|_{H^1(\Omega)} \le \sqrt{2}\left(\|u_t^{(1)}-u_t^{(2)}\|_{L^2(\Omega)}^2
+
\|u^{(1)}-u^{(2)}\|_{H^1(\Omega)}^2\right)^{1/2} 
\end{align}
and using  $(x^2 + y^2 +z^2)^{1/2} \leq x+y+z $ for $x,y,z \geq 0$ we get
\begin{align}\label{eq: xyz}
&\left(\|u_0^{(1)}-u_0^{(2)}\|_{H^1(\Omega)}^2
+
\|u_1^{(1)}-u_1^{(2)}\|_{L^2(\Omega)}^2
+
\|g^{(1)}-g^{(2)}\|_{\partial \Omega}^2\right)^{1/2}\\
&\le \left(\|u_0^{(1)}-u_0^{(2)}\|_{H^1(\Omega)}
+
\|u_1^{(1)}-u_1^{(2)}\|_{L^2(\Omega)}
+
\|g^{(1)}-g^{(2)}\|_{\partial \Omega}^2\right) \end{align}
Combining \eqref{eq: sqrt}, \eqref{eq: ab} and \eqref{eq: xyz}
gives us the required result, i.e.,
\begin{equation}
\begin{aligned}
\|u_t^{(1)}-u_t^{(2)}\|_{L^2(\Omega)}
+
\|u^{(1)}-u^{(2)}\|_{H^1(\Omega)} \le & Ce^{Ct}\left(\|u_0^{(1)}-u_0^{(2)}\|_{H^1(\Omega)}
+
\|u_1^{(1)}-u_1^{(2)}\|_{L^2(\Omega)}
+
\|g^{(1)}-g^{(2)}\|_{\partial \Omega}^2\right)         
\end{aligned}
\end{equation}
\end{proof}


\section{Theory of the inverse problem}\label{sec: Inverse problem theory}

Let $\Omega$ be a bounded domain with a smooth boundary $\partial \Omega$ such that at each point $y \in \partial \Omega$, a detector is placed that measures $g(x,t)$ for any $t > 0$. 
We define a forward map 
\begin{equation}
    \begin{aligned}
\mathcal{W}:H^1(\Omega) \rightarrow L^2[(0,T);H^{1/2}(\partial\Omega)],~\mathcal{W} : f \mapsto g = u |_{\partial \Omega \times (0,\infty)}
    \end{aligned}
\end{equation}
where $u$ solves the \eqref{eq: Main PDE}.
\begin{equation}\label{eq: PDE for inverse}
\begin{aligned}
 &u_{tt} - c^2(x) \Delta u - d(x,t) \mathcal{M}* \Delta u_t = \eta(x,t)\big( u^2 \big)_{tt}, \\
& u(x,0)=f(x), \quad u_t(x,0)=0.
\end{aligned}
\end{equation}
We assume that $f$ is compactly supported in $\Omega$. For the forthcoming discussions, we define the following spaces:
\[
X := H^1(\Omega),
\]
\[
Z := C([0,T];H^1(\Omega)) \cap C^1([0,T];L^2(\Omega)) \cap L^2(0,T;H^2(\Omega)),
\]
\[
Y := L^2(0,T;L^2(\Omega)).
\]
We next state the following result on observability estimates of an associated linear equation, whose proof follows from Theorem 1 in \cite{Claret2024}. 
\begin{prop}\label{prop:observability}
For the linear operator
\[
\mathcal L_{u_0}v :=
v_{tt}-c^2\Delta v-d\mathcal M * \Delta v_t - 2\eta (u_0 v)_{tt}
\]
we have the following observability estimate:
\begin{equation}\label{obs2}
\|v(\cdot,0)\|_{X}
\leq
C \|\mathcal L_{u_0}v\|_Y
\end{equation}
for all $v \in Z$ such that
\[
v|_{\partial\Omega\times(0,T)}=0, \quad v_t(\cdot,0)=0.
\]
\end{prop}
We now state and prove a local uniqueness result for the recovery of $f$ from $\mathcal{W}(f)$. 

\begin{thm}
Let $\Omega \subset \mathbb{R}^n$ be a bounded domain with smooth boundary and let $T>0$. Consider
\begin{equation}
\begin{aligned}
u_{tt}-c^2(x)\Delta u-d(x,t)\mathcal M * \Delta u_t
=
\eta(x,t)(u^2)_{tt},\\
u(x,0)=f(x), \quad u_t(x,0)=0.
\end{aligned}
\end{equation}
Fix $f_0 \in X$ and let $u_0 \in Z$ be the corresponding solution. Assume $1 - 2\eta(x,t)u_0(x,t) \geq \kappa > 0$.
Then there exists $\varepsilon>0$ such that if
\[
\|f_1-f_0\|_{X} + \|f_2-f_0\|_{X} < \varepsilon
\]
and
\[
\mathcal F(f_1)=\mathcal F(f_2),
\]
then
\[
f_1=f_2.
\]
\end{thm}

\begin{proof}
Let $u_1,u_2 \in Z$ be the solutions corresponding to $f_1,f_2$. Define
\[
w := u_1 - u_2.
\]
Then
\[
w(x,0)=f_1-f_2, \quad w_t(x,0)=0,
\]
and
\[
w|_{\partial\Omega\times(0,T)}=0.
\]
Subtracting the equations gives
\[
w_{tt}-c^2\Delta w-d\mathcal M * \Delta w_t
=
\eta\big[(u_1^2)_{tt}-(u_2^2)_{tt}\big].
\]
Let
\[
u_j = u_0 + z_j, \quad j=1,2,
\]
such that
\[
w = z_1 - z_2.
\]
Then, we have
\[
u_1^2 - u_2^2 = 2u_0 w + z_1^2 - z_2^2,
\]
and hence
\[
(u_1^2)_{tt}-(u_2^2)_{tt}
=
2(u_0 w)_{tt} + (z_1^2 - z_2^2)_{tt}.
\]
Thus,
\[
\mathcal L_{u_0} w
=
\eta (z_1^2 - z_2^2)_{tt}.
\]
Applying the observability estimate \eqref{obs2} to $w$, we obtain
\[
\|w(\cdot,0)\|_{X}
\leq
C \|\mathcal L_{u_0} w\|_{Y}.
\]
Hence,
\[
\|w(\cdot,0)\|_{X}
\leq
C \|\eta (z_1^2 - z_2^2)_{tt}\|_{Y}.
\]
Since
\[
z_1^2 - z_2^2 = w(z_1 + z_2),
\]
standard product estimates in Sobolev spaces yield
\[
\|\eta (z_1^2 - z_2^2)_{tt}\|_{Y}
\leq
C \|z_1 + z_2\|_{Z} \|w\|_{Z}.
\]
By Theorem \ref{th:stab},
\[
\|z_j\|_{Z} \leq C \|f_j - f_0\|_{X}, \quad j=1,2,
\]
and
\[
\|w\|_{Z} \leq C \|w(\cdot,0)\|_{X}.
\]
Therefore, we have
\[
\|w(\cdot,0)\|_{X}
\leq
C\big(\|f_1-f_0\|_{X} + \|f_2-f_0\|_{X}\big)
\|w(\cdot,0)\|_{X}.
\]
For sufficiently small $\varepsilon$, we have
\[
C\big(\|f_1-f_0\|_{X} + \|f_2-f_0\|_{X}\big) < 1,
\]
which implies
\[
\|w(\cdot,0)\|_{X} = 0.
\]
Thus $f_1 = f_2$, proving local uniqueness.
\end{proof}
We now show the existence of the reconstruction of $f$ from the boundary measurements $g$ of the linear part of the inverse problem having some restrictions on the coefficients.

\begin{thm}
Let $\Omega \subset \mathbb{R}^n$ be a bounded domain with a smooth boundary $\partial \Omega$ such that at each point $y \in \partial \Omega$, a point detector is placed that measures the value of the pressure $g(x,t) = u(x,t)$ for any $t>0$. Let the initial pressure function $f(x)$ to be reconstructed, be compactly supported in $\Omega$.  
Then for the linear equation
\[
\begin{aligned}
 &u_{tt} - c^2 \Delta u - d \mathcal{M}* \Delta u_t = 0, \\
& u(x,0)=f(x), \quad u_t(x,0)=0,
\end{aligned}
\]
with the coefficients $c,d$ as constants, for a given measurement function $g$, there exist an $f(x)$. 
\end{thm}
\begin{proof}
Let $B$ be the interior domain of $\partial \Omega$. Let $E$ be the harmonic extension of function from $\partial \Omega$ to $B$ that is the operator that produces a harmonic function $Eg$ in $B$ for given dirichlet data $g$ on $\partial \Omega$ and substitute  $w= u- Eg$ that satisfies the following problem:
\begin{equation}\label{Main Equation after some extensions}
    \begin{aligned}
        w_{tt} - c^2 \Delta w - d \mathcal{M}* \Delta w_t = -E(g_{tt}) \quad x\in B,t\geq0 \\
    \end{aligned}
\end{equation}
On the boundary $\partial \Omega$, we get $w(y,t)= u(y,t)-g(y,t)=0$, and hence $w$ satisfies homogeneous Dirichlet data. At $t=0$, we have
\[w(x,0)=f(x)- E(g|_{t=0}), w_t(x,0)=-E(g_t|_{t=0}), \quad x\in B.
\]
Let $\Delta_{D}$ be the dirichlet laplacian in $B$ which is unbounded self adjoint operator in $L^2(B)$, $\left\{\Phi_n(x)\right\}_{n=1}^{\infty} $ be orthonormal basis of eigen functions of $-\Delta_{D}$ in $L^2(B)$ and $\left\{\lambda_n\right\}_{n=1}^{\infty} $ be the corresponding spectrum where we assume $\lambda_n>0$ such that
\[-\Delta\Phi_n(x) = \lambda_n^2\Phi_n(x),\quad\forall x \in B,\]
\[\Phi_n(y) = 0, \quad \forall y \in \partial \Omega.\]  
Let
\begin{equation}\label{final solution}
    w(x,t)= \sum\limits_{n=1}^{\infty}w_n(t)\phi_n(x),\\
\end{equation}
be the solution of \eqref{Main Equation after some extensions}.
Multiplying \eqref{Main Equation after some extensions} by $\phi_n(x)$ and integrating with respect to $B$, we obtain
\begin{equation}\label{Converted PDE problem}
    \begin{aligned}
        w_n''(t)+c^2\lambda_n^2w_n(t)+d\lambda_n^2\mathcal{M}*w_n'(t)=Q_n(t)\\
    \end{aligned}
\end{equation}   
using the fact that
\[
\begin{aligned}
&\int\limits_{B}w_{tt}\phi_n(x)dx= w_n''(t), \\
&\int\limits_{B}\Delta w\phi_n(x)dx=-\lambda_n^2w_n(t) \quad \text{(using Green's Identity)}, \\
&\int\limits_{B}\mathcal{M}* \Delta w_t\phi_n(x)dx=(-\lambda_n^2)\mathcal{M}*w_n'(t), \\
& -\int\limits_{B}E(g_{tt})\phi_n(x)dx = Q_n(t).
\end{aligned}
\]
We also have the initial condition for $w_n$ as
\[
\begin{aligned}
&w_n(0)=\int\limits_{B}(f(x)- E(g|_{t=0}))\phi_n(x)dx,\\ 
&w_n'(0)=\int\limits_{B}(-E(g_t|_{t=0}))\phi_n(x)dx.
\end{aligned}
\]
Taking Laplace transform of \eqref{Converted PDE problem}, we obtain
\begin{equation}\label{Laplace transformed equation}
 \begin{aligned}
    W_n(s) &= w_n(0)(s+d\lambda_n^2\textbf{M}(s))H_n(s)+ w'_n(0)H_n(s)+\textbf{Q}_\textbf{n}(s)H_n(s) 
\end{aligned}   
\end{equation}
where \[H_n(s)=\frac{1}{s^2+c^2\lambda_n^2+d\lambda_n^2s\textbf{M}(s)}\]
Another Laplace inverse of \eqref{Laplace transformed equation} gives us
\begin{equation}\label{final}
    \begin{aligned}
        w_n(t)=\mathcal{L}^{-1}((s+d\lambda_n^2\textbf{M}(s))H_n(s)).w_n(0) + \mathcal{L}^{-1}(H_n(s)).w'_n(0) + \mathcal{L}^{-1}(H_n(s))* (Q_n(t))\\
    \end{aligned}
\end{equation}
Substituting $t=0$ in \eqref{final}, we get
\[
w_n(0)=\mathcal{L}^{-1}((s+d\lambda_n^2\textbf{M}(s))H_n(s)).w_n(0) + \mathcal{L}^{-1}(H_n(s)).w'_n(0) + \mathcal{L}^{-1}(H_n(s))* (Q_n(t)),
\]
which implies
\[
w_n(0)=\frac{\mathcal{L}^{-1}(H_n(s)).w_n'(0) + \mathcal{L}^{-1}(H_n(s))* (Q_n(t))}{(1-\mathcal{L}^{-1}((s+d\lambda_n^2\textbf{M}(s))H_n(s)))}.
\]
Further substituting $t=0$ in \eqref{final solution} we obtain
\[
f(x)-E(g|_{t=0})=w(x,0)=\sum\limits_{n=1}^{\infty}w_n(0)\phi_n(x)
\]
which gives us
\[ f(x)=\sum\limits_{n=1}^{\infty}w_n(0)\phi_n(x)+E(g|_{t=0}).
\]
\end{proof}
Using this result, we now state and prove a local existence result for the solvability of the NDVW inverse problem.
\begin{thm}
Let \(g\) be a boundary measurement and assume that the hypotheses of Theorem 5 hold for the
linear damped wave equation, so that the associated linear inverse problem admits a reconstruction
\(f_{\mathrm{l}}\). Assume further that \(T>0\), \(\|\eta\|_{L^\infty(\Omega\times(0,T))}\), and
\(\|f_{\mathrm{l}}\|_{H^1(\Omega)}\) are sufficiently small. Then there exists a radius
\(\rho>0\) and at least one initial condition
\[
f\in B_\rho(f_{\mathrm{l}})\subset H^1(\Omega)
\]
such that 
\[
\mathcal{W}(f)=g.
\]
\end{thm}
\begin{proof}
Let \(\mathcal W_{\mathrm{l}}\) denote the forward map for the linear damped wave equation
\[
v_{tt}-c^2\Delta v-dM*\Delta v_t=0,
\qquad
v(x,0)=f(x),\quad v_t(x,0)=0,
\]
and let \(\mathcal W\) denote the forward map for the nonlinear equation
\[
u_{tt}-c^2\Delta u-dM*\Delta u_t=\eta(x,t)(u^2)_{tt}.
\]
By Theorem 5, for the given boundary data \(g\), there exists
\(f_{\mathrm{l}}\) such that
\[
\mathcal W_{\mathrm{l}}(f_{\mathrm{l}})=g.
\]
We now write the nonlinear solution as
\[
u=v+z,
\]
where \(v\) solves the linear damped wave equation with initial condition \(f\), and \(z\) is the nonlinear correction. Then \(z\) satisfies
\[
z_{tt}-c^2\Delta z-dM*\Delta z_t
=
\eta(x,t)(u^2)_{tt},
\]
with homogeneous initial data
\[
z(x,0)=0,\qquad z_t(x,0)=0.
\]
Taking traces on \(\partial\Omega\times(0,T)\), the nonlinear boundary measurement becomes
\[
\mathcal W(f)
=
\mathcal W_{\mathrm{l}}(f)+\mathcal K(f),
\]
where \(\mathcal K(f):=z|_{\partial\Omega\times(0,T)}\) is the nonlinear boundary correction. Thus, solving the nonlinear inverse problem
\[
\mathcal W(f)=g
\]
is equivalent to solving
\[
\mathcal W_{\mathrm{l}}(f)+\mathcal K(f)=g.
\]
Using Theorem 5, let \(\mathcal R_{\mathrm{l}}\) denote the local linear reconstruction operator satisfying
\[
\mathcal R_{\mathrm{l}}\mathcal W_{\mathrm{l}}(f)=f.
\]
Then the nonlinear inverse problem can be rewritten as the fixed-point problem
\[
f
=
\mathcal R_{\mathrm{l}}\big(g-\mathcal K(f)\big).
\]
Define
\[
\mathcal T(f):=\mathcal R_{\mathrm{l}}\big(g-\mathcal K(f)\big).
\]
Since \(\mathcal R_{\mathrm{l}}g=f_{\mathrm{l}}\), we obtain
\[
\mathcal T(f)
=
f_{\mathrm{l}}-\mathcal R_{\mathrm{l}}\mathcal K(f).
\]
We now show that \(\mathcal T\) is a contraction on a sufficiently small ball
\[
B_\rho(f_{\mathrm{l}})
=
\{f\in H^1(\Omega):\|f-f_{\mathrm{l}}\|_{X}\le \rho\}.
\]
By the local well-posedness theorem for the nonlinear damped wave equation, the solution satisfies an estimate of the form
\[
\|u\|_Z
\le C\|f\|_{X}
\]
for \(f\) sufficiently small. Moreover, since
\[
(u^2)_{tt}=2u_t^2+2uu_{tt},
\]
standard product estimates give
\[
\|\eta(u^2)_{tt}\|_{L^2(0,T;L^2(\Omega))}
\le
C\|\eta\|_{L^\infty}\|u\|_Z^2.
\]
Therefore,
\[
\|\mathcal K(f)\|_{L^2(0,T;H^{1/2}(\partial\Omega))}
\le
C\|\eta\|_{L^\infty}\|f\|_{X}^2.
\]
Applying the boundedness of \(\mathcal R_{\mathrm{l}}\), we get
\[
\|\mathcal T(f)-f_{\mathrm{l}}\|_{X}
\le
C\|\eta\|_{L^\infty}\|f\|_{X}^2.
\]
Hence, if \(\rho\), \(\|f_{\mathrm{l}}\|_{X}\), and \(\|\eta\|_{L^\infty}\) are sufficiently small, then
\[
\mathcal T(B_\rho(f_{\mathrm{l}}))
\subset B_\rho(f_{\mathrm{l}}).
\]
Next, let \(f_1,f_2\in B_\rho(f_{\mathrm{l}})\), and let \(u_1,u_2\) be the corresponding nonlinear
solutions. Then
\[
u_1^2-u_2^2=(u_1-u_2)(u_1+u_2).
\]
Using the stability estimate for the nonlinear forward problem, given in \eqref{eq:stability}, we obtain
\[
\|\mathcal K(f_1)-\mathcal K(f_2)\|
\le
C\|\eta\|_{L^\infty}
\big(\|u_1\|_Z+\|u_2\|_Z\big)
\|u_1-u_2\|_Z.
\]
Furthermore, the same stability estimate gives us
\[
\|u_1-u_2\|_Z
\le
C\|f_1-f_2\|_{X}.
\]
Therefore,
\[
\|\mathcal T(f_1)-\mathcal T(f_2)\|_{X}
\le
C\|\eta\|_{L^\infty}
\big(\|f_1\|_{X}+\|f_2\|_{H^1}\big)
\|f_1-f_2\|_{X}.
\]
Since \(f_1,f_2\in B_\rho(f_{\mathrm{l}})\), we have
\[
\|f_j\|_{X}\le \|f_{\mathrm{l}}\|_{X}+\rho,
\qquad j=1,2.
\]
Thus,
\[
\|\mathcal T(f_1)-\mathcal T(f_2)\|_{X}
\le
C\|\eta\|_{L^\infty}
\big(\|f_{\mathrm{l}}\|_{X}+\rho\big)
\|f_1-f_2\|_{X}.
\]
Choose \(\rho\), \(T\), and \(\|\eta\|_{L^\infty}\) sufficiently small so that
\[
C\|\eta\|_{L^\infty}
\big(\|f_{\mathrm{l}}\|_{X}+\rho\big)<1.
\]
Then \(\mathcal T\) is a contraction on \(B_\rho(f_{\mathrm{l}})\). By Banach's fixed-point theorem,
there exists a unique \(f\in B_\rho(f_{\mathrm{l}})\) such that
\[
f=\mathcal T(f).
\]
Equivalently,
\[
\mathcal W(f)=g,
\]
which proves the result.
\end{proof}

\section{Numerical reconstruction framework}\label{sec:control}
To solve for $u_0(x)$ given observation function $g(y,t)$ on the observation domain $\partial \Omega$, we consider the following optimization problem:
\begin{equation}\label{eq:minproblem}
\begin{aligned}
\min_{u_0\in U_{ad}}& J(u_0,u):= \dfrac{1}{2}\intL_0^T\intL_{\partial\Omega} (u(x,t) - g(y,t))^2~{\rm d}s(x){\rm d}t + \dfrac{\alpha}{2} \intL_\Omega u_0^2(x)~{\rm d}x + \beta \intL_\Omega |u_0(x)|~{\rm d}x\\
\mbox{such that } 
&u_{tt} - c^2(x) \Delta u - d(x,t) \mathcal{M}* \Delta u_t = \eta(x,t)\big( u^2 \big)_{tt},~ (x,t) \in \Omega \times (0,T), \\
& u(x,0)=u_0(x), \quad u_t(x,0)=0,~ x \in \Omega.
\end{aligned}
\end{equation}
where $u_0\in U_{ad} = \lbrace L^2(\Omega): 0 \leq u_l \leq u_0(x) \leq u_r \rbrace.$ Here the first term represents the standard least-squares data fitting term, and the last two terms represent the $L^2-L^1$ regularization term, with $\alpha,\beta>0$, that has the ability to reconstruct sparsity patterns in $u_0(x)$.

\begin{prop}\label{th:coercive}
The objective functional \( J \), defined in \eqref{eq:minproblem}, is sequentially weakly lower semi-continuous (w.l.s.c.), bounded from below, and coercive on the set \( U_{ad} \).
\end{prop}

\begin{proof}
The bounded from below and coercivity property of $J$ follows from the non-negativity and the fact that $J$ contains the $L^2$ regularization of $u_0$. For proving the w.l.s.c property, we first note that $J$ is not convex in $u_0$. However, $U_{ad}$ is convex and closed and also weakly sequentially compact. This is because if we choose any $\{u_0^n\}\in U_{ad}$ such that $||u_0^n||_{L^2}\leq D,$ for all $n$, then the sequence is bounded. Thus, there is a  subsequence $\{u_0^{n_k}\}$ such that $u_0^{n_k}\rightharpoonup y\in U_{ad}$. Also, we have that $||u_0||_{L^2}\leq \lim\inf||u_0^n||_{L^2}\leq D$. Thus, we obtain $u_0\in U_{ad}$. Next, we consider the set 
\[
U_{\eta} = \lbrace u_0 \in U_{ad}: J(u_0) \leq \zeta \rbrace.
\]
Since $J$ is continuous, the set $U_{\eta}$ is closed for all $\zeta \in \mathbb{R}$. Thus, $U_{\eta}$ is a closed subset of a weakly sequentially compact spaces $U_{ad}$ and is weakly sequentially closed for all $\zeta\in\mathbb{R}.$
This implies that $J$ is weakly sequentially lower semi-continuous. 
\end{proof}

We denote the solution of \eqref{eq:main-pde} for a given initial condition $u_0$ as $u = \Lambda (u_0)$. Then the existence of an optimal $u_0^*$ can be established in the following theorem:

\begin{thm}\label{th:existence}
Let $u \in C([0,T];H^1(\Omega))$ satisfy \eqref{eq:waveic}, and let the cost functional $J$ be defined as in \eqref{eq:minproblem}. Then, there exist a pair 
$
(u_0^*,u^*) \in U_{ad} \times C([0,T];H^1(\Omega))
$
such that $(u_0^*,u^*)$ satisfy \eqref{eq:waveic} $u_0^*$ minimize $J$ in $U_{ad}$.
\end{thm}

\begin{proof}

Since $U_{ad}$ is a closed subspace of a Hilbert space and $J$ is coercive on $U_{ad}$, any minimizing sequence $(u_0^m)$ admits a weakly convergent subsequence $(u_0^{m_l})$ such that $u_0^{m_l} \to u_0^*$ in $U_{ad}$. 
By applying the Aubin–Lions compactness lemma \cite{Lions1969}, the corresponding sequence $(u^{m_l} = \Lambda( u_0^{m_l}))$ contains a subsequence $(u^{m_k})$ that converges strongly in $L^2([0,T]; L^2(\Omega))$ to $u^* = \Lambda(u^*)$, thereby ensuring the existence of a minimizer for $J$.
\end{proof}

To characterize the solutions of \eqref{eq:minproblem}, we use the framework of the Pontryagin's maximum principle (PMP). For this purpose, we formulate the following Hamiltonian function:
\begin{equation}\label{eq:Hamiltonian}
H(x,u_0(x),p(x,:)) =  \dfrac{\alpha}{2} u_0^2(x) + \beta  |u_0(x)| + u_0(x) \left[ \partial_tp(x,0)-\int_0^T \mathcal{M}(t) \Delta_x (d(x,t) p(x,t)) \, dt\right],\quad x \in \Omega,
\end{equation}
where $p(x,t)$ solves the following adjoint equation
\begin{equation}\label{eq:adjoint}
\begin{aligned}
& p_{tt}(x,t) - \Delta (c^2(x)p)+ \mathcal{M}(0) \Delta_x \left(d(x,t) p(x,t) \right) + \int_t^T \mathcal{M}'(\tau - t) \Delta_x(d(x,\tau)p(x,\tau))d\tau \\
&- 2 (\eta(x,t) p(x,t))_{tt} u(x,t)
=-\chi_{\partial \Omega} [u(x,t) - \tilde{u}(x,t)], ~(x,t) \in \mathbb{R}^n \times (0,T),  \\
& p(x,T)=0~ \text{and} ~ p_t(x,T)=0,~ x \in \mathbb{R}^n.
\end{aligned}
\end{equation}
Then, we have the following characterization of the optimal control through the PMP:
\begin{thm}\label{th:PMP}
The optimal initial condition and adjoint $(u_0^*,p^*)$ satisfies the following PMP criterion 
\[
H(x, u_0^*(x),p^*(x,:)) = \min_{v\in[u_l,u_r]} H(x, v,p^*(x,:)),~ \mbox{ for a.e.  } x \in \Omega.
\]
\end{thm}
For the proof of this theorem, we employ the classical needle variation argument (see, for example, \cite{Borzi2023}). 
Let \(S_{\varrho}(x_{0})\subset \Omega\) denote an open ball centered at \(x_{0}\in\Omega\) with radius \(\varrho>0\), satisfying
\[
\lim_{\varrho\to 0}|S_{\varrho}(x_{0})|=0.
\]
Given \(u_{0}\in U_{ad}\), we define its needle perturbation at the point \(x_{0}\) by
\begin{equation}
u_0^{\varrho}(x)
:=
\begin{cases}
u_0(x), & x\in \Omega\setminus S_{\varrho}(x_{0}),\\[2mm]
w, & x\in S_{\varrho}(x_{0})\cap\Omega,
\end{cases}
\label{eq:Def of u_k}
\end{equation}
where \(w\in [u_l,u_r]\). By construction, \(u_0^\varrho\in U_{ad}\) for every \(x_{0}\in\Omega\) and every admissible \(u_{0}\) (see \cite{Borzi2023}).
Since \(u_0\in L^\infty(\Omega)\), almost every point in \(\Omega\) is a Lebesgue point of \(u_0\). Hence,
\begin{equation}
\|u_0^{\varrho}-u_0\|_{L^{2}(\Omega)}
=
\left(
\int_{S_{\varrho}(x_{0})}
|w-u_0(x)|^{2}\,{\rm d}x
\right)^{1/2}
\longrightarrow 0,
\qquad \text{as } \varrho\to 0,
\label{limurho}
\end{equation}
for almost every \(x_{0}\in\Omega\).
Combining this convergence with the stability estimates for the attenuated wave equation yields
\begin{equation}
\|u_{\varrho}-u\|_{L^\infty(0,T;L^2(\Omega))}
\to 0,
\qquad
\|p_{\varrho}-p\|_{L^\infty(0,T;L^2(\Omega))}
\to 0,
\qquad
\|\partial_t p_{\varrho}-\partial_t p\|_{L^\infty(0,T;L^2(\Omega))}
\to 0,
\label{limpqrho}
\end{equation}
as \(\varrho\to 0\), where \(u_\varrho\) and \(p_\varrho\) denote the state and adjoint variables associated with the perturbed initial condition \(u_0^\varrho\). In addition,
\begin{equation}
\|\partial_t u_{\varrho}(\cdot,0)-\partial_t u(\cdot,0)\|_{L^2(\Omega)}
\to 0,
\qquad
\|u_{\varrho}(\cdot,0)-u(\cdot,0)\|_{L^2(\Omega)}
\to 0.
\label{limpqrhoBc}
\end{equation}
Assume now that \(u_0\) is an optimal control. We consider the corresponding variation of the objective functional:
\begin{align}
J(u_0^\varrho,u_\varrho)-J(u_0,u)
&=
\int_0^T
\int_{\partial\Omega}
\left(
\frac{u_\varrho(y,t)+u(y,t)}{2}
-g(y,t)
\right)
\big(u_\varrho(y,t)-u(y,t)\big)
\,{\rm d}s(y)\,{\rm d}t
\notag\\
&\quad
+
\int_{\Omega}
\Big(
G(u_0^\varrho(x))-G(u_0(x))
\Big)
\,{\rm d}x,
\label{eDiffJ}
\end{align}
where
\[
G(u_0)
=
\frac{\alpha}{2}u_0^2
+
\beta |u_0|.
\]
To facilitate the analysis, we introduce the following auxiliary adjoint problem:
\begin{equation}\label{eWaveAdjxInter}
\begin{aligned}
\partial_t^2 p_\varrho(x,t)
-
\partial_t\big(\gamma(t)p_\varrho(x,t)\big)
-
\Delta\big(c^2(x)p_\varrho(x,t)\big)
&=
-
\left[
\frac{u_\varrho(x,t)+u(x,t)}{2}
-g(x,t)
\right]
\chi_{\partial\Omega},
\\
&\hspace{-3cm}
(x,t)\in \mathbb{R}^n\times(0,T),
\\[2mm]
p_\varrho(x,T)&=0,
\qquad x\in\mathbb{R}^n,
\\
\partial_t p_\varrho(x,T)&=0,
\qquad x\in\mathbb{R}^n.
\end{aligned}
\end{equation}
Standard regularity theory for \eqref{eWaveAdjxInter} implies
\begin{equation}
p_\varrho
\in
L^\infty(0,T;L^2(\Omega)),
\qquad
\partial_t p_\varrho
\in
L^\infty(0,T;L^2(\Omega))
\cap
L^2(0,T;L^2(\Omega)).
\label{eStimate}
\end{equation}
We now state the following lemma concerning the variation of the functional \(J\) under the needle perturbation.
 {\begin{lem}\label{eDifJDifH}
The following equation holds 
\[
J(u_0^\varrho,u_\varrho) - J(u_0,u)  =  \intL_{\Omega} \left( H\left(x,u_0^{\varrho},p_\varrho\right)-H\left(x,u_0,p_\varrho\right) \right) \, {\rm d}x, 
\]
where $H$ is given as in \eqref{eq:Hamiltonian}.
\end{lem}
\begin{proof}
Using \eqref{eWaveAdjxInter} in \eqref{eDiffJ}, to replace 
$\Big( \frac{u_\varrho(x,t) 
+ u(x,t)}{2} - g(x,t) \Big)$  with the corresponding 
left-hand side of the intermediate adjoint problem, 
integration by parts and the use of the boundary and initial and terminal conditions give the desired result. 
\end{proof}

Now, we can consider the needle variation in the limit $\varrho \to 0$. We have the following lemma:
\begin{lem}\label{eMon}
Let $u_0^*\in U_{ad}$ be a minimizer and $w\in [u_l,u_r]$. 
Furthermore, let $u_0^{\varrho}$ be defined as in
\eqref{eq:Def of u_k}, and $u_{\varrho}$ be the
solution to the NDVW wave equation \eqref{eq: Main PDE} with $u_0= u_0^{\varrho}$. Then, the following holds
$$
0 \le  \lim_{\varrho\rightarrow 0}\frac{1}{|S_{\varrho}\left(x_0\right)|}   \left(J\left({u_0}^{\varrho},u_{\varrho}\right)-J\left(u_0^*, u^*\right)\right)=  -H\left(x_0,u_0^*,p^*\right)  +H\left(x_0,w,p^*\right) , 
   $$
for almost all $x_0 \in \Omega$ and $w \in [u_l,u_r]$. 
\end{lem}
\begin{proof}
For any $k\in\mathbb{N}$, we have
\[
\begin{aligned}
0 \leq &\frac{1}{|S_{\varrho}\left(x_0\right)|}   \left(J\left(u_0^{\varrho},u_{\varrho}\right)-J\left(u_0^*,u^*\right)\right)\\
=& \frac{1}{|S_{\varrho}\left(x_0\right)|}  \intL_\Omega \Big( -H\left(x,u_0^*,p_\varrho\right)+H\left(x,u_0^{\varrho},p_\varrho\right) \, {\rm d}x,\\
= & \frac{1}{|S_{\varrho}\left(x_0\right)|}  \intL_{S_{\varrho}\left(x_0\right)}
\Big( -H\left(x,u_0^*,p^*\right)  +H\left(x,w,p^*\right) \Big) \,{\rm d} x\\
&+\frac{1}{|S_{\varrho}\left(x_0\right)|}  \intL_{S_{\varrho} \left(x_0\right)} \Big( (u_0^*-w)[(\partial_tp_{\varrho}(x,0)-\partial_tp^*(x,0)) -\int_0^T \mathcal{M}(t) \Delta_x (d(x,t) (p_\varrho(x,0)-p^*(x,0))) \, dt] \Big) \, {\rm d}x.
\end{aligned}
\]
Since, $x \in \Omega \mapsto \Big( H\left(x,u_0^*,p^*\right)  -H\left(x,w,p^*\right) \Big) \in L^1(\Omega)$, convergence results for $u_0^\varrho$, $u_\varrho$, and $p_\varrho$ in \eqref{limurho} and \eqref{limpqrho}, we have the following using the mean value theorem: 
\[
0 \le  \lim_{\varrho\rightarrow 0}\frac{1}{|S_{\varrho}\left(x_0\right)|}   \left(J\left(u_0^{\varrho},u_{\varrho}\right)-J\left(u_0^*, u^*\right)\right)
=  -H\left(x_0,u_0^*,p^*\right)  +H\left(x_0,w,p^*\right).
\]
\end{proof}
As a consequence of Lemma \ref{eMon}, we have proved the PMP Theorem \ref{th:PMP}. A major advantage of the PMP framework is that the resulting optimality characterization avoids the need for derivatives of the objective functional \(J\) with respect to the initial condition \(u_0\), in contrast to classical first-order optimality conditions derived through the Euler--Lagrange framework. This observation motivates the development of the sequential quadratic Hamiltonian (SQH) method for the numerical implementation of the PMP condition stated in Theorem~\ref{th:PMP}. 

The SQH method, introduced in its modern form in \cite{Borzi2023}, represents a recent advancement in the class of successive approximation (SA) schemes. The central idea of these methods is the iterative pointwise minimization of the Hamiltonian associated with the underlying optimization problem. The SQH method is based on the augmented Hamiltonian functional defined by
\[
H_\epsilon(x,u_0,\tilde{u}_0,p)
=
H(x,u_0,p)
+
\epsilon (u_0-\tilde{u}_0)^2,
\label{eq:augmented_Hamiltonian}
\]
where \(\epsilon>0\) is a penalization parameter that is adaptively updated during the SQH iterations. In particular, \(\epsilon\) is increased whenever a sufficient decrease in the objective functional \(J\) is not achieved, and decreased otherwise. Here, \(\tilde{u}_0\) denotes the previous iterate of the initial condition \(u_0\). The additional quadratic penalty term
\[
\epsilon (u_0-\tilde{u}_0)^2
\]
acts as a regularization mechanism that restricts successive updates of \(u_0\) from deviating excessively from the previous iterate, especially for large values of \(\epsilon\). It is important to emphasize that, during each optimization sweep over the spatial grid points \(x\), the state and adjoint variables \(u\) and \(p\) are kept fixed and correspond to the solutions obtained from the previous iteration. The resulting SQH procedure is summarized in the following algorithm:
\begin{algorithms}[SQH method]\text{ }
\begin{itemize}
\item Input:  initial approx. $u_0^0$, max. number of iterations $k_{max}$, tolerance $\kappa >0$, 
$\epsilon>0$, $\lambda >1$, $\eta > 0$, and $\zeta\in\left(0,1\right)$; 
set $\tau > \kappa $, $k:=0$.
\item Compute the solution $p^0$ to the damped wave equation given in \eqref{eq:minproblem} with initial condition $u_0=u_0^0$. 
\item While ($k<k_{max} ~\&\& ~ \tau > \kappa$ ) do 
\begin{enumerate}[(a)]
\item Compute the solution $p^k$ to the adjoint problem \eqref{eq:adjoint} with $u=u^k$. 
\item Determine $u_0^{k+1}$ such that the following optimization problem is satisfied 
$$
{H}_\epsilon \left(x,u_0^{k+1},u_0^k, p^k\right)
= \min_{v\in [u_l,u_r]} {H}_\epsilon \left(x,v,u_0^k, p^k\right), 
$$
at almost all $x \in \Omega$.
\item Compute the solution $u^{k+1}$ to the damped wave equation given in \eqref{eq:minproblem} with initial condition $u_0=u_0^{k+1}$
\item Compute $\tau:=\|u_0^{k+1 } -u_0^{k}\|^2_{L^{2}\left(\Omega\right)}$.
\item If $J\left(u_0^{k+1},u^{k+1}\right)-J\left(u_0^{k},u^{k}\right)> -\eta \, \tau$, 
then increase $\epsilon$ with $\epsilon=\lambda \, \epsilon$ and go to Step (b).\\
Else if 
$J\left(u_0^{k+1},u^{k+1}\right)-J\left(u_0^{k},u^{k}\right) \leq -
\eta \, \tau $, then 
decrease $\epsilon$ with $\epsilon=\zeta \, \epsilon$ and continue. 
\item Set $k:=k+1$. 
\end{enumerate}
\item end While
\end{itemize}
\label{algSQHmethod}
\end{algorithms}

In Step {\it (e)} of this algorithm, if the inequality $J\left(u_0^{k+1},u^{k+1}\right)-J\left(u_0^{k},u^{k}\right)> -\eta \, \tau$ holds, it indicates that a sufficient decrease in the objective functional $J$ has not been achieved. In such a case, $\epsilon$ is increased (since $\lambda > 1$), and the optimization in Step {\it (b)} is repeated with the updated augmented Hamiltonian function. In contrast, if the inequality does not hold, it confirms that the required reduction in $J$ has been achieved. The updated initial condition $u_0^{k+1}$ is then adopted, together with the corresponding updates $u^{k+1}$ and $p^{k+1}$ for the damped wave equation and its adjoint. In this situation, $\epsilon$ is reduced by a factor $\zeta < 1$.

\begin{thm} \label{th:SQH} Let $\left(u_0^{k},u^{k}\right)$ and $\left(u_0^{k+1},u^{k+1}\right)$ be generated by the SQH method (Algorithm \ref{algSQHmethod}) applied to \eqref{eq:minproblem}, with $u_0^{k+1}, u^k \in U_{ad}$. Then, there exists a constant $C > 0$, independent of $\epsilon$ and $u_0^k$, such that for the current value of $\epsilon > 0$ chosen by Algorithm \ref{algSQHmethod}, the following inequality holds:
\begin{equation}\label{eq:ineq}
J\left(u_0^{k+1},u^{k+1}\right) - J\left(u_0^{k},u^{k}\right) \leq -\left(\epsilon - C\right) \, \| u_0^{k+1} - u_0^k \|^{2}_{L^2\left(\Omega\right)}.
\end{equation}
In particular, this implies $J\left(u_0^{k+1},u^{k+1}\right) - J\left(u_0^{k},u^{k}\right) \leq - \eta \, \tau$ for $\epsilon \geq C + \eta$ and $\tau = \| u_0^{k+1} - u_0^k \|^{2}_{L^2\left(\Omega\right)}$. \end{thm}
\begin{proof}
 We have the following Hamiltonian function 
$$
H\left(x,u_0(x),p(x)\right) = \dfrac{\alpha}{2} u_0^2(x) + \beta  |u_0(x)| + u_0(x) \left[ \partial_tp(x,0)-\int_0^T \mathcal{M}(t) \Delta_x (d(x,t) p(x,t)) \, dt\right].
$$ 
In Step {\it (b)} of Algorithm \ref{algSQHmethod}, we have that for almost all $x \in \Omega$ it holds:
$$
{H}_\epsilon \left(x,u_0^{k+1},u_0^k, p^k\right)
\leq {H}_\epsilon \left(x,v,u_0^k, p^k\right) , 
$$ 
for all $v\in [u_l,u_r]$. Therefore, we have
\begin{equation*}
{H}_\epsilon \left(x,u_0^{k+1},u_0^k, p^k\right)
\leq {H}_\epsilon \left(x,u_0^k,u_0^k, p^k\right)={H} \left(x,u_0^k, q^k\right) .
\end{equation*}
Hence, we obtain the inequality
\begin{equation}\label{ePartialMax}
{H} \left(x,u_0^{k+1}, p^k\right)+\epsilon \, | u_0^{k+1} - u_0^k |^2 \le {H} \left(x,u_0^k, p^k\right) .
\end{equation}
Define $\delta p = p^{k+1}-p^k$, $l(u_0) = \dfrac{\alpha}{2} u_0^2 + \beta |u_0|$
and  $\delta u_0 = u_0^{k+1}-u_0^k$. We have 
\begin{align*}
&J(u_0^{k+1},p^{k+1}) - J(u_0^{k},p^{k}) \\
&\quad=
 \dfrac{1}{2}\intL_0^T\intL_{\partial\Omega} 2(p^k(y,t) - g(y,t))\delta p(y,t) + (\delta p(y,t))^2 ~{\rm d}s(y){\rm d}t + \intL_\Omega l(u_0^{k+1})(x)-l(u_0^k) (x)~{\rm d}x \\
&\quad=\dfrac{1}{2}\intL_0^T\intL_{\partial\Omega} (\delta p(y,t))^2 ~{\rm d}s(y){\rm d}t + \intL_0^T\intL_{\partial\Omega} (p^k(y,t) - g(y,t))\delta p(y,t) ~{\rm d}s(y){\rm d}t\\
&\quad\qquad\qquad\qquad\qquad\qquad + \intL_\Omega l(u_0^{k+1})(x)-l(u_0^k)(x) ~{\rm d}x\\
&  \quad=\dfrac{1}{2}\intL_0^T\intL_{\partial\Omega} (\delta p(y,t))^2 ~{\rm d}s(y){\rm d}t + \intL_{\Omega} [\partial_tp^k(x,0)-\gamma(0)p^k(x,0)]\delta u_0(x) ~{\rm d}x \\
&\quad\qquad\qquad\qquad\qquad\qquad+ \intL_\Omega l(u_0^{k+1})(x)-l(u_0^k)(x) ~{\rm d}x\\
&\quad=\dfrac{1}{2}\intL_0^T\intL_{\partial\Omega} (\delta p(y,t))^2 ~{\rm d}s(y){\rm d}t + \intL_{\Omega} {H} \left(x,u_0^{k+1}, p^k\right) - {H} \left(x,u_0^{k}, p^k\right) ~{\rm d}x \\
&\quad \le 
C \,  \| u_0^{k+1} - u_0^k \|^2_{L^2(\Omega)}  - \epsilon \,  \| u_0^{k+1} - u_0^k \|^2_{L^2(\Omega)} ,
\end{align*}
where the last step follows from the equation following \eqref{ePartialMax} and the standard stability estimate of solutions of \eqref{eq:main-pde}. 
\end{proof}

This theorem shows that, if $(u_0^k, p^k)$ are not already optimal, 
it is possible to choose $\epsilon > C$ to obtain a successful minimization step. We have the following corollary as a consequence of Theorem \ref{th:SQH}.

\begin{cor} \label{CorollaryEps}
The sequence $\lbrace\epsilon\rbrace$ of the SQH iterates is bounded.
\end{cor}
\begin{proof}
From Algorithm \ref{algSQHmethod}, we note that the successful 
$k$-th minimization step is performed with $\epsilon= \, (C + \eta)$, 
where $C$ is estimated in Theorem \ref{th:SQH} above. 
Then, there is a constant $\bar{\epsilon} $ related to the stability estimate of the attenuated wave equation \eqref{eq:main-pde}, 
the data of the problem, and the regularization and SQH parameters, that 
provides an upper bound $\epsilon \le \bar{\epsilon}$ of the sequence $\lbrace\epsilon\rbrace$ of the SQH iterates.
\end{proof}
This corollary ensures that the SQH algorithm converges, which is what we prove in the next theorem.
\begin{thm}
\label{Tsqh2}
Under the assumptions of Theorem \ref{th:SQH}, if in Algorithm \ref{algSQHmethod}, at every $k$-th iterate, $\epsilon=C+\eta$ 
is chosen (and tolerance  $\kappa=0$, so that the algorithm never stops), then the following holds
\begin{enumerate}[(a)]
\item $\lim\limits_{k \to \infty} | J(u_0^{k+1},p^{k+1}) - J(u_0^{k},p^k)|=0$
	\label{sqhac}
\item $\lim\limits_{k \to \infty} \| u_0^{k+1} - u_0^{k} \|_{ L^2 \left( \Omega \right) }=0$. 
	\label{sqhbc}
\item Along a subsequence, $
u_0^k \rightharpoonup \bar u_0,  
~ \text{in } L^2(\Omega), $
with corresponding state and adjoint limits
\[
u^k \to \bar u, \quad \text{in } L^2(0,T; L^2(\Omega)) 
\qquad
p^k \to \bar p, \quad \text{in } L^2(0,T; L^2(\Omega)), 
\]
and the pair $(\bar u_0,\bar p)$ satisfies the PMP criterion
\[
H(x,\bar u_0(x),\bar p(x,:))
=
\min_{v\in [u_l,u_r]}
H(x,v,\bar p(x,:)),
\qquad \text{for a.e. } x\in \Omega .
\]
\end{enumerate}
\end{thm}
\begin{proof}
The first statement follows from \eqref{eq:ineq}, since $\epsilon$ can be chosen greater than $C$, which would imply that the sequence $\lbrace J(u_0^{k},u^k)\rbrace $ is monotonically decreasing the $\mathbb{R}$ and is a Cauchy sequence, hence convergent. To prove \eqref{sqhbc}, we rewrite 
\eqref{eq:ineq} as follows 
$$
 \| u_0^{k+1} - u_0^{k} \|^{2}_{ L^2 \left( \Omega\right) } \le \frac{1}{\eta}
 \Big[ J\left(u_0^k, p^{k}\right)   - J\left(u_0^{k+1},p^{k+1}\right) \Big] . 
$$
Therefore we have the partial sum
$$
 \sum_{k=0}^K \| u_0^{k+1} - p^{k} \|^{2}_{ L^2 \left( \Omega \right) } \le \frac{1}{\eta}
 \Big[ J\left(u_0^0,p^{0}\right)   - J\left(u_0^{K+1},p^{K+1}\right) \Big].
$$
This shows that in the limit $K \to  \infty$, the series with positive elements 
$\| u_0^{k+1} - u_0^{k} \|^{2}_{ L^2(\Omega)}$ is convergent, which proves the result. 

For the final statement, we note that since $U_{ad}$ is weakly sequentially compact in $L^2(\Omega)$, there exists a subsequence, still denoted by $\{u_0^k\}$, and an element $\bar u_0\in U_{ad}$ such that
\[
u_0^k\rightharpoonup \bar u_0
\quad \text{in }L^2(\Omega).
\]
Let
\[
u^k=\Lambda(u_0^k)
\]
be the corresponding sequence of states. By the compactness argument used in Theorem~\ref{th:existence}, together with the Aubin--Lions lemma, there exists a further subsequence such that
\[
u^k\to \bar u
\quad \text{strongly in }L^2(0,T;L^2(\Omega)).
\]
By the stability of the forward problem and the continuity of the control-to-state map, the limit satisfies
\[
\bar u=\Lambda(\bar u_0).
\]
Therefore, $(\bar u_0,\bar u)$ is an admissible state-initial condition pair. It remains to identify the limit as a stationary point. At the $k$-th SQH iteration, the updated initial condition $u_0^{k+1}$ is obtained by pointwise minimization of the Hamiltonian. Hence,
\[
H(x,u_0^{k+1}(x),p^k(x,:))
=
\min_{v\in[u_l,u_r]}
H(x,v,p^k(x,:)),
\qquad \text{for a.e. }x\in\Omega.
\]
Equivalently, for every $v\in[u_l,u_r]$,
\[
H(x,u_0^{k+1}(x),p^k(x,:))
\le
H(x,v,p^k(x,:)),
\qquad \text{for a.e. }x\in\Omega.
\]
Passing to the limit along the convergent subsequence, using the convergence of $p^k$ to $\bar p$, the weak convergence of $u_0^{k+1}$ to $\bar u_0$, and the continuity of the Hamiltonian with respect to its arguments, we obtain
\[
H(x,\bar u_0(x),\bar p(x,:))
\le
H(x,v,\bar p(x,:)),
\qquad \forall v\in[u_l,u_r],
\]
for a.e. $x\in\Omega$. Therefore,
\[
H(x,\bar u_0(x),\bar p(x,:))
=
\min_{v\in[u_l,u_r]}
H(x,v,\bar p(x,:)),
\qquad \text{for a.e. }x\in\Omega.
\]
which implies that $\bar u_0$ satisfies the PMP optimality condition.
\end{proof}

Theorem \ref{Tsqh2} guarantees that Algorithm \ref{algSQHmethod} is well defined for $\kappa>0$. Hence, there is an iteration number $k_0\in\mathbb{N}$ such that $\Vert u_0^{k_0+1}-u^{k_0}\Vert_{L^2(\Omega)}\leq\kappa$. This implies that the 
SQH algorithm stops in finitely many steps and converges weakly to a local minima along a subsequence.

\subsection{CNN-guided initial guess construction}

To solve for $u_0$ using the SQH scheme, we construct the initial guess, using a CNN algorithm. The architecture of the convolutional neural network is designed to reconstruct the initial condition of a PDE from multiple boundary measurements using a structured multi-channel approach. The input tensor consists of four channels, where each channel corresponds to data from one of the four boundary measurements of the domain. The CNN first processes each channel independently through identical convolutional layers consisting of three convolutional layers with 32 filters, kernel size $3 \times 3$, ReLU activation, and same padding. This allows the network to learn and preserve the distinct physical features of each boundary without prematurely mixing information.
The resulting feature maps from all four channels are then concatenated, forming a combined representation. A shared encoder is then applied, consisting of convolutional layers with 64 filters and kernel size $3 \times 3$, to learn interactions across the different boundary measurements. This is followed by an anisotropic downsampling step implemented via a convolution with stride $(1,2)$, which reduces the spatial resolution in one direction while preserving the other, aligning with the feature map with the reconstruction domain. After downsampling, additional convolutional layers with 64 filters and kernel size $3 \times 3$ further refine the representation and propagate information from the boundaries into the interior. Finally, a $1 \times 1$ convolution with a single filter and linear activation produces a single-channel output corresponding to the reconstructed initial condition.

\section{Numerical results}\label{sec:numericalresult}
In this section, we present the results of our SQH method to solve the inverse problem to obtain the initial damped wave acoustic pressure $p_0$ from observational data. 
We choose our domain $\Omega = (-1,1)\times (-1,1)$, the observation boundary as $\partial\Omega$, and the final time of observation as $T=2$. We choose a non-trapping sound speed $c(x) = 1+w(x)*[0.1*\cos(2\pi x_1) + 0.05*\sin(2\pi x_2)],$ where $w(x)$ is a mollifier centered at the middle of the domain with radius $\sqrt{0.5}$, with a maximum value of 1. The damping coefficient $d(x,t)$ is chosen as $0.1\exp(-t)$, the nonlinearity coefficient is chosen as $0.001$, and the memory kernel $M(t) = 0.01\exp(-2t)$. For the spatial grid, we choose 75 points along each dimension, whereas our time grid comprises of 100 points. To generate the data, we solve the wave equation \eqref{eq: Main PDE} in free space on a spatial grid with 37 points and on the temporal grid at 50 time points, and then interpolate the solution on the original grid to collect the data on the boundary $\partial\Omega$. 

We compare our reconstructions with those obtained via the time-reversal approach, in which the wave equation \eqref{eq: Main PDE} is solved backward in time using the observed boundary data as boundary conditions, as well as with reconstructions produced by our CNN model. The time-reversal solution evaluated at the final time serves as an approximation of the initial condition for the original wave equation \eqref{eq: Main PDE}. For qualitative and quantitative assessment, we employ standard image quality metrics, including the Mean Squared Error (MSE), Peak Signal-to-Noise Ratio (PSNR), and Structural Similarity Index Measure (SSIM), defined as follows:
\[
MSE(I^1,I^2) = \dfrac{1}{n}\sum_{p =1}^n (I_p^1-I^2_p)^2,~ PSNR(I^1,I^2) = 10\log_{10}\left(\dfrac{[\max({I^1,I^2})]^2}{MSE(I^1,I^2)}\right),
\]
where $p$ represents a pixel and $n$ is the total number of pixels, and 
\[
SSIM(I^1, I^2) 
= \frac{(2\mu_{I^1} \mu_{I^2} + C_1)(2\sigma_{I^1I^2} + C_2)}
       {(\mu_{I^1}^2 + \mu_{I^2}^2 + C_1)(\sigma_{I^1}^2 + \sigma_{I^2}^2 + C_2)},
\]
where $\mu_{I^1}$ and $\mu_{I^2}$ denote the average intensities of the 
images $I^1$ and $I^2$, respectively, $\sigma_{I^1}^2$ and $\sigma_{I^2}^2$ represent 
the corresponding intensity variances, and $\sigma_{I^1I^2}$ is the covariance between the two 
images. The MSE measures the average squared pixel-wise discrepancy between two images and is purely error-based, without accounting for perceptual aspects. PSNR, a logarithmic transformation of MSE, quantifies reconstruction fidelity relative to the maximum possible intensity; higher values indicate lower reconstruction error, though it still neglects structural information. In contrast, SSIM evaluates similarity by incorporating luminance, contrast, and structural components, making it more consistent with human visual perception. Consequently, while MSE and PSNR emphasize pixel-level accuracy, SSIM provides a measure of perceptual similarity between images.

For training the CNN, we generated a dataset of \(750\) samples: For the first 150 samples, the output was chosen as Gaussian functions with centers randomly drawn from the interval \( (-0.9, 0.9) \) and widths drawn from \((50, 150)\). For the next 350 samples, the output was chosen as characteristic functions with centers in the interval $(-0.9,0.9)$ and widths in the interval $(0.1,0.5)$. For the last 250 samples, we have a sum of a Gaussian and characteristic function as output, with centers of the Gaussian drawn from the interval $(-0.9,0.9)$ and widths in $(10,60)$ while the centers of the characteristic function are in $(-0.9,0.9)$ and widths in $(0.1,0.5)$. The number of epochs chosen for the training was 500, with batch size 32. The optimizer was chosen as ``Adam" with the mean square loss function.

In our first test case, our true phantom is a combination of a Gaussian phantom with center at $(-0.3,-0.3)$ and standard deviation 0.5, and a disk phantom of center (0.2,-0.2), radius 0.2, and intensity 1. Both the 2D and cross-sectional 1D view  reconstructions are shown in Figure \ref{fig:2G}.
\begin{figure}[H]
\centering
\subfloat[True]{\includegraphics[width=0.24\textwidth, height=0.24\textwidth]{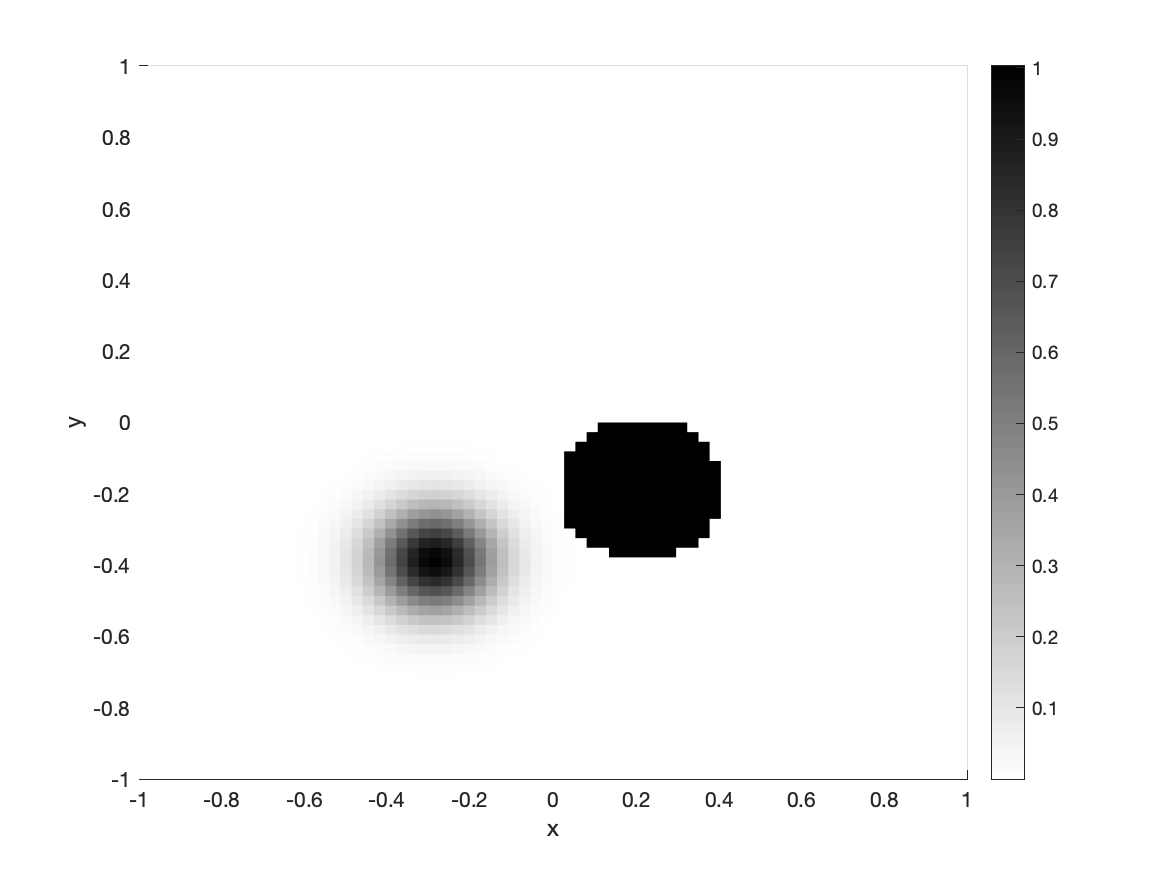} \label{fig:true2G}}
\subfloat[Time-reversal]{\includegraphics[width=0.24\textwidth, height=0.24\textwidth]{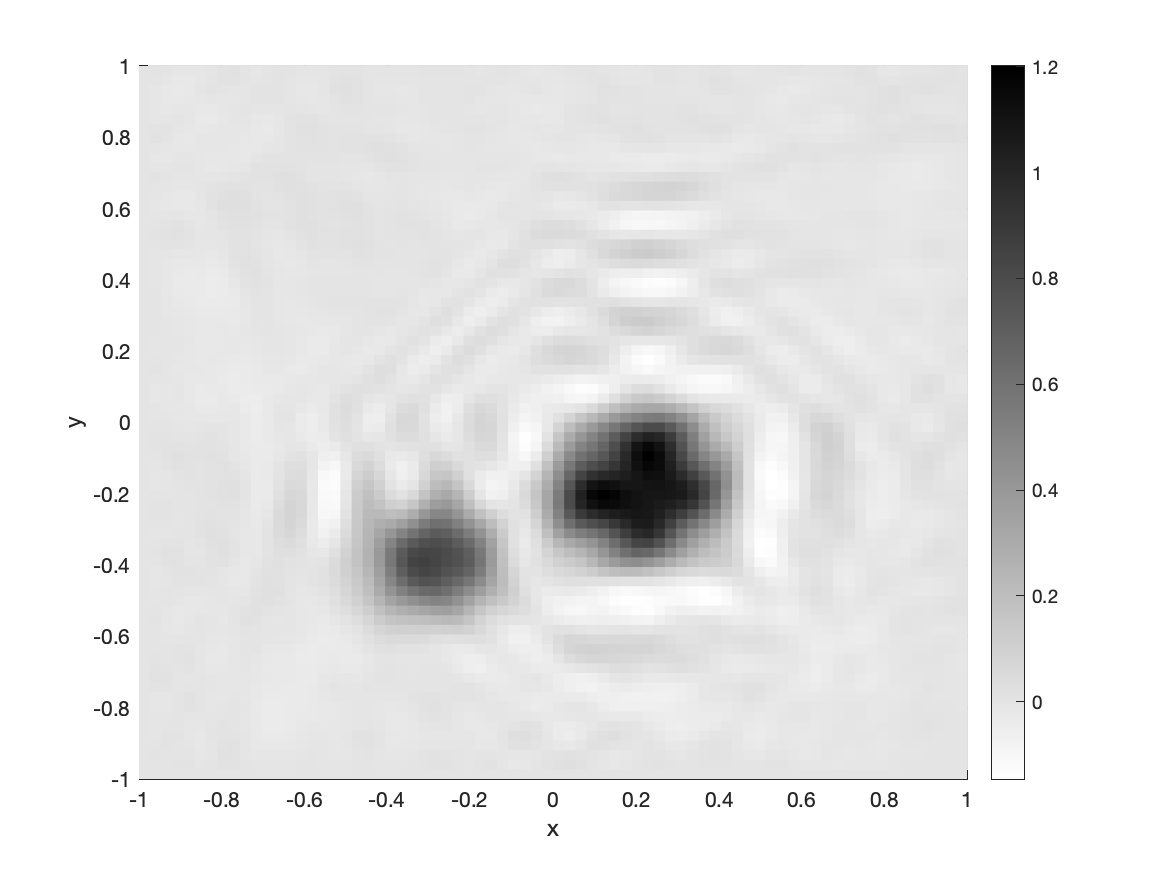}\label{fig:TR2G}}
\subfloat[CNN]{\includegraphics[width=0.24\textwidth, height=0.24\textwidth]{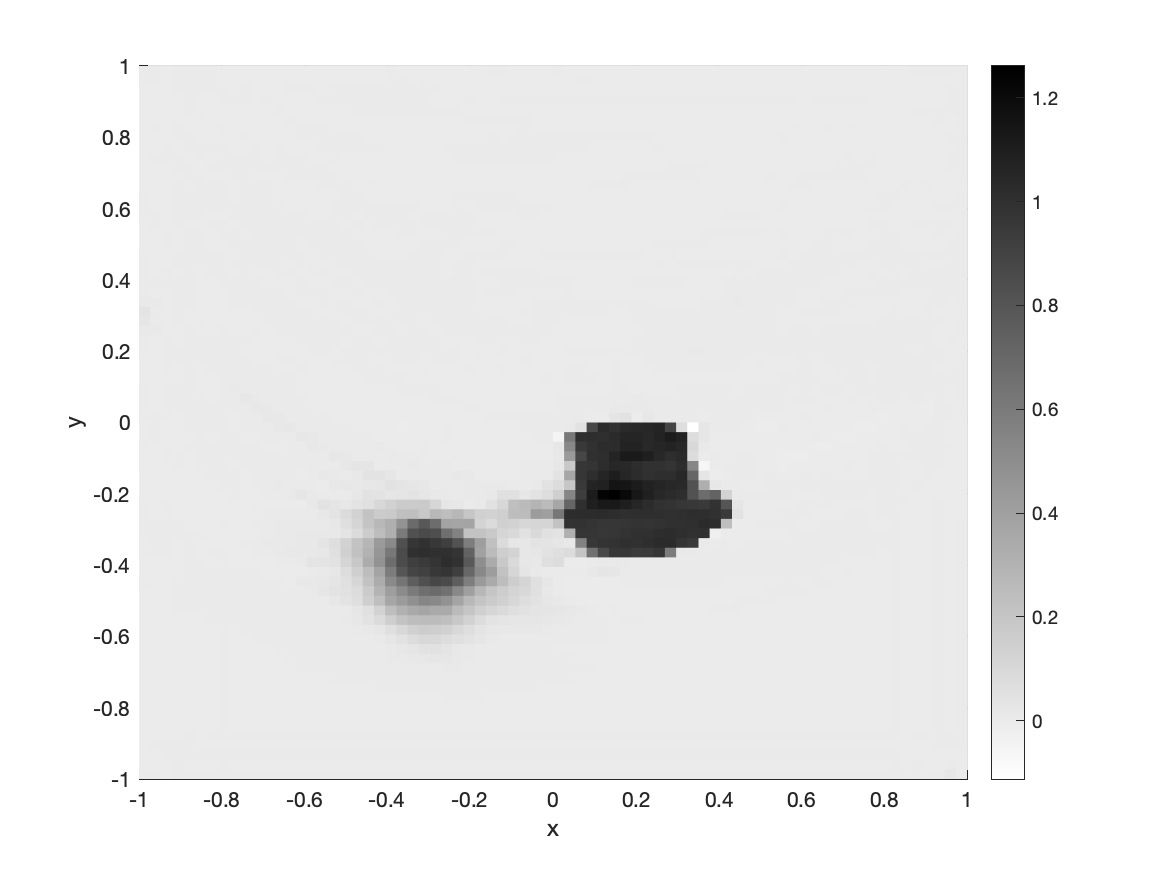}\label{fig:CNN2G}}
\subfloat[SQH]{\includegraphics[width=0.24\textwidth, height=0.24\textwidth]{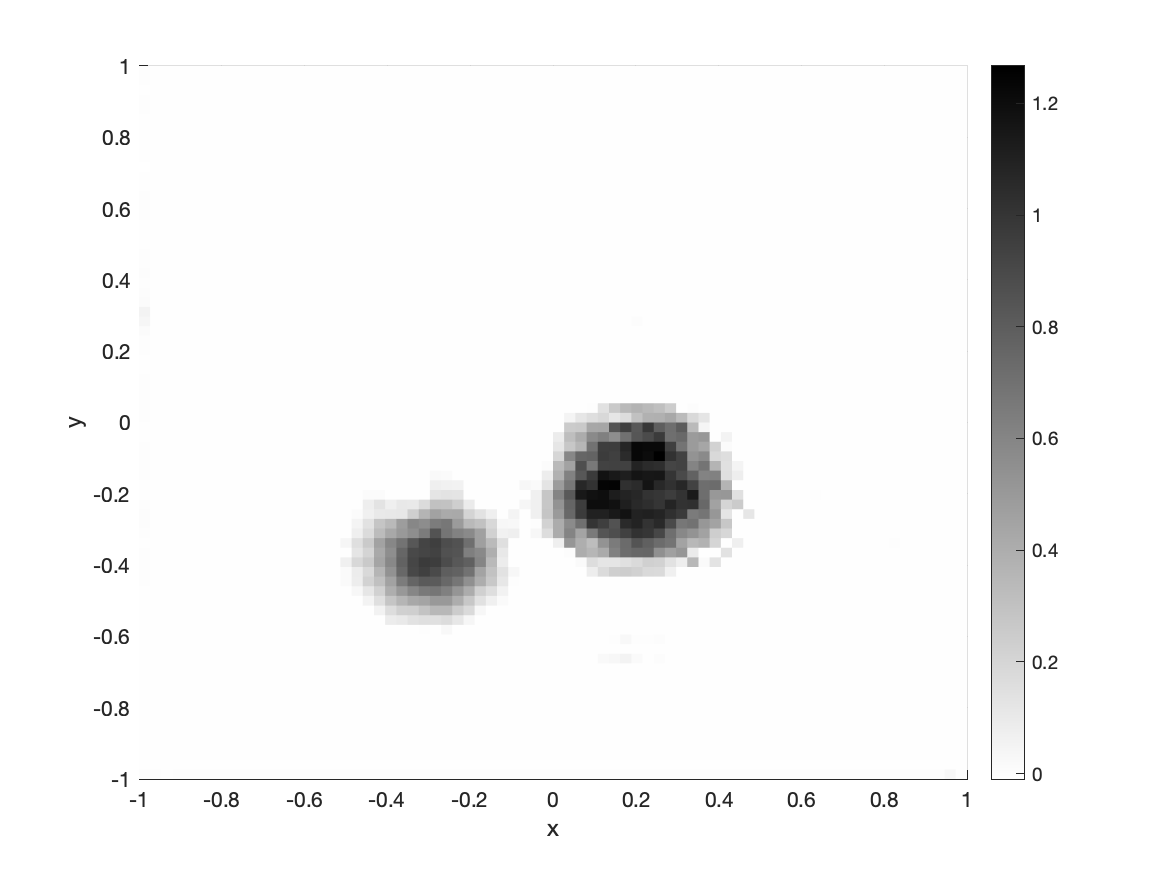}\label{fig:SQH2G}}\\
\subfloat[True]{\includegraphics[width=0.24\textwidth, height=0.24\textwidth]{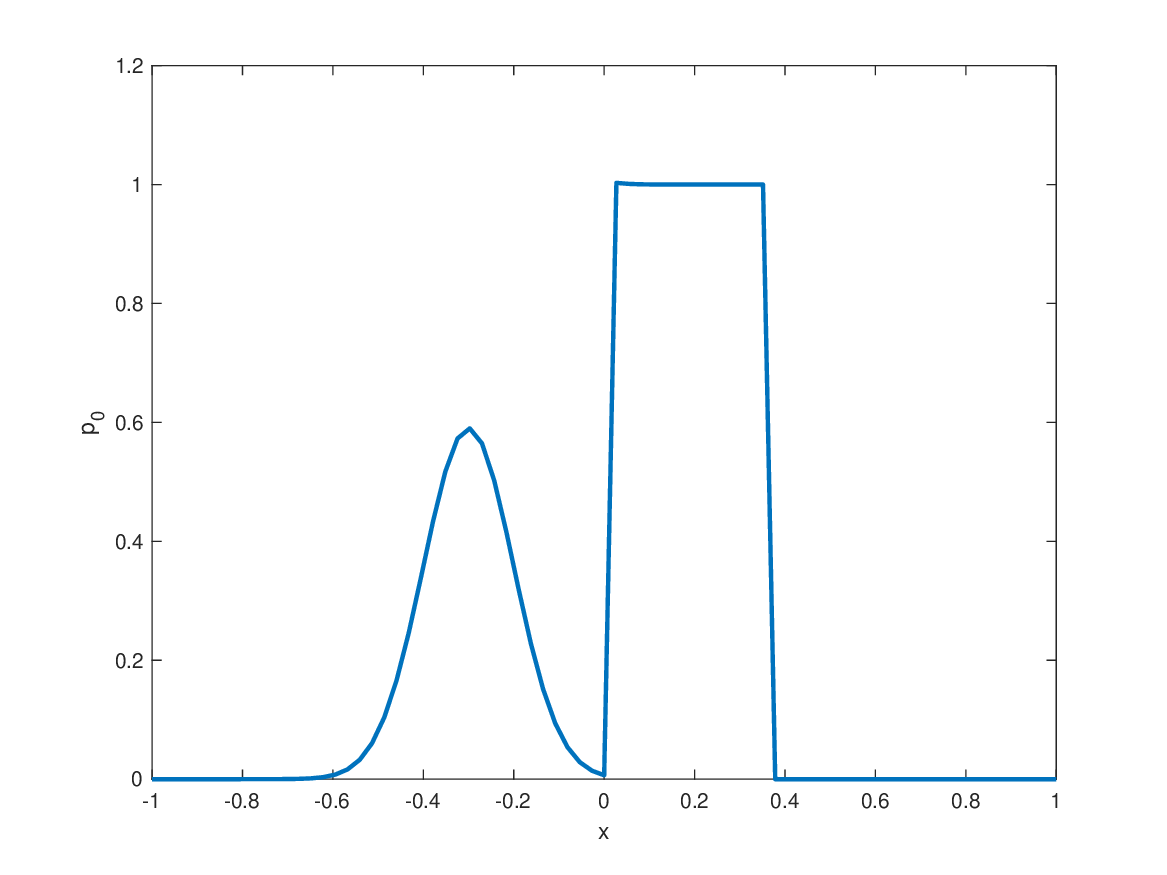} \label{fig:truecross2G}}
\subfloat[Time-reversal]{\includegraphics[width=0.24\textwidth, height=0.24\textwidth]{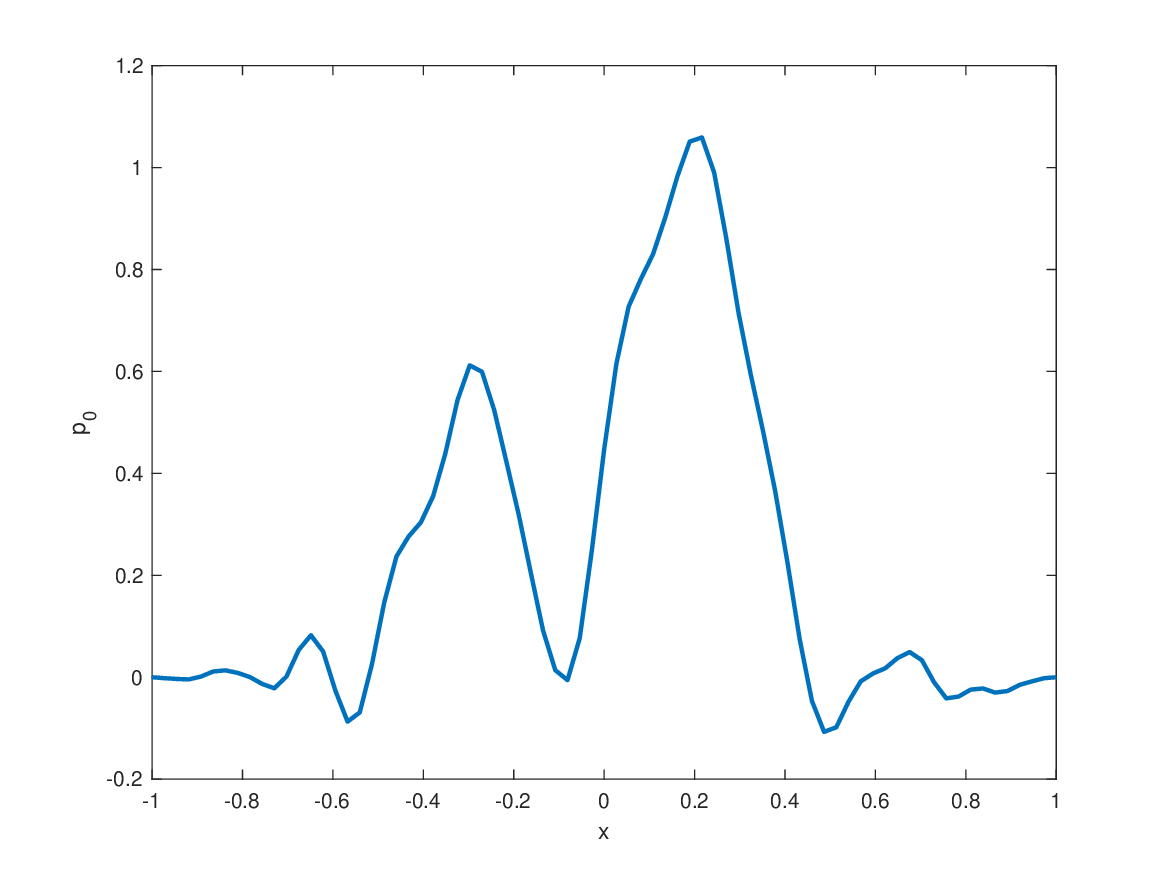}\label{fig:TRcross2G}}
\subfloat[CNN]{\includegraphics[width=0.24\textwidth, height=0.24\textwidth]{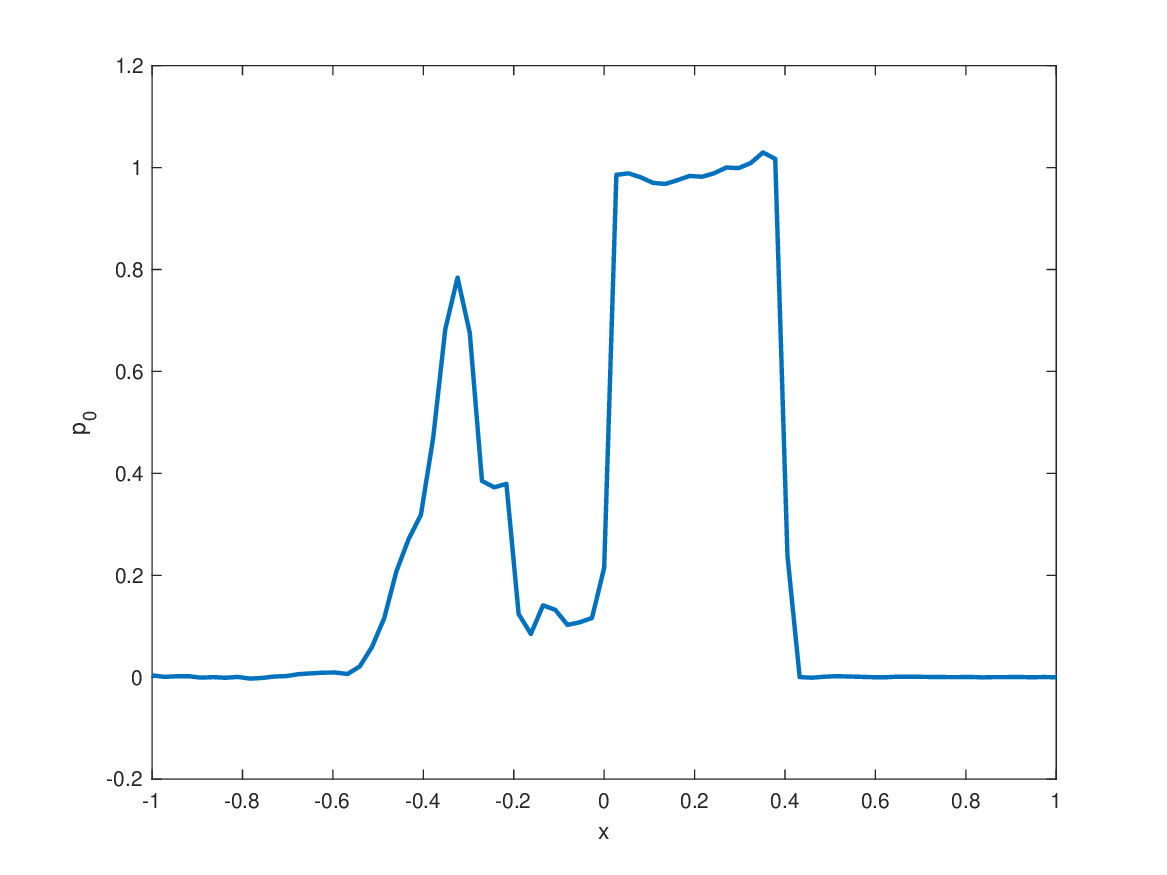}\label{fig:CNNcross2G}}
\subfloat[SQH]{\includegraphics[width=0.24\textwidth, height=0.24\textwidth]{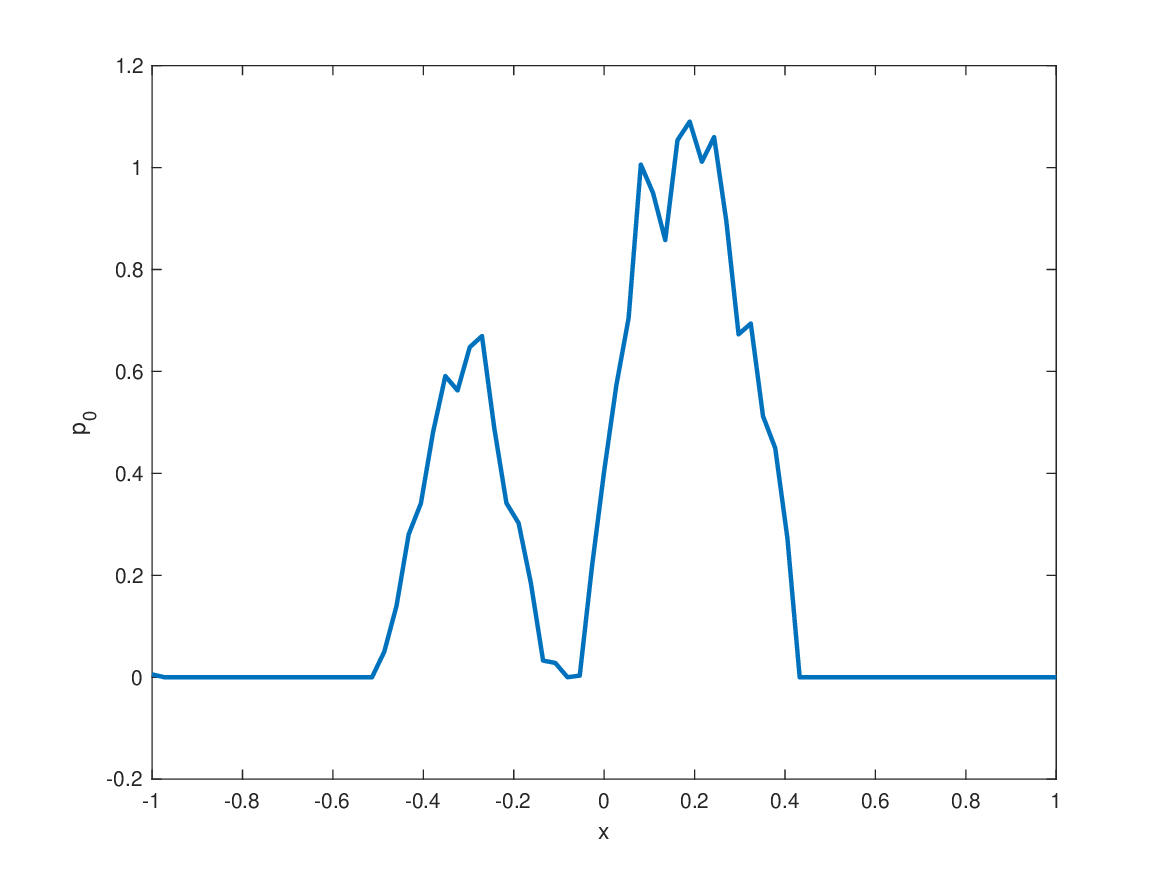}\label{fig:SQHcross2G}}
\caption{Test Case 1: Reconstructions  with two phantoms; Top row corresponds to the 2D profiles; Bottom row corresponds to a 1D cross sectional profile}
\label{fig:2G}
  \end{figure}
We observe that the time-reversal reconstruction provides fair resolution but not the best contrast. The CNN reconstruction is not the best in terms of resolution. On the other hand, the SQH algorithm provides superior resolution and contrast, outperforming both the methods. 

For the next test case, we now consider a combination of 3 phantoms: a Gaussian phantom centered at (-0.5,-0.5) and standard deviation 0.5, 2 disk phantoms centered at $(0.2,0.2)$ and $(0.5,-0.4)$ with radius 0.2 and intensity 1. Both the 2D and cross-sectional 1D view  reconstructions are shown in Figure \ref{fig:3G}.
  \begin{figure}[H]
\centering
\subfloat[True]{\includegraphics[width=0.24\textwidth, height=0.24\textwidth]{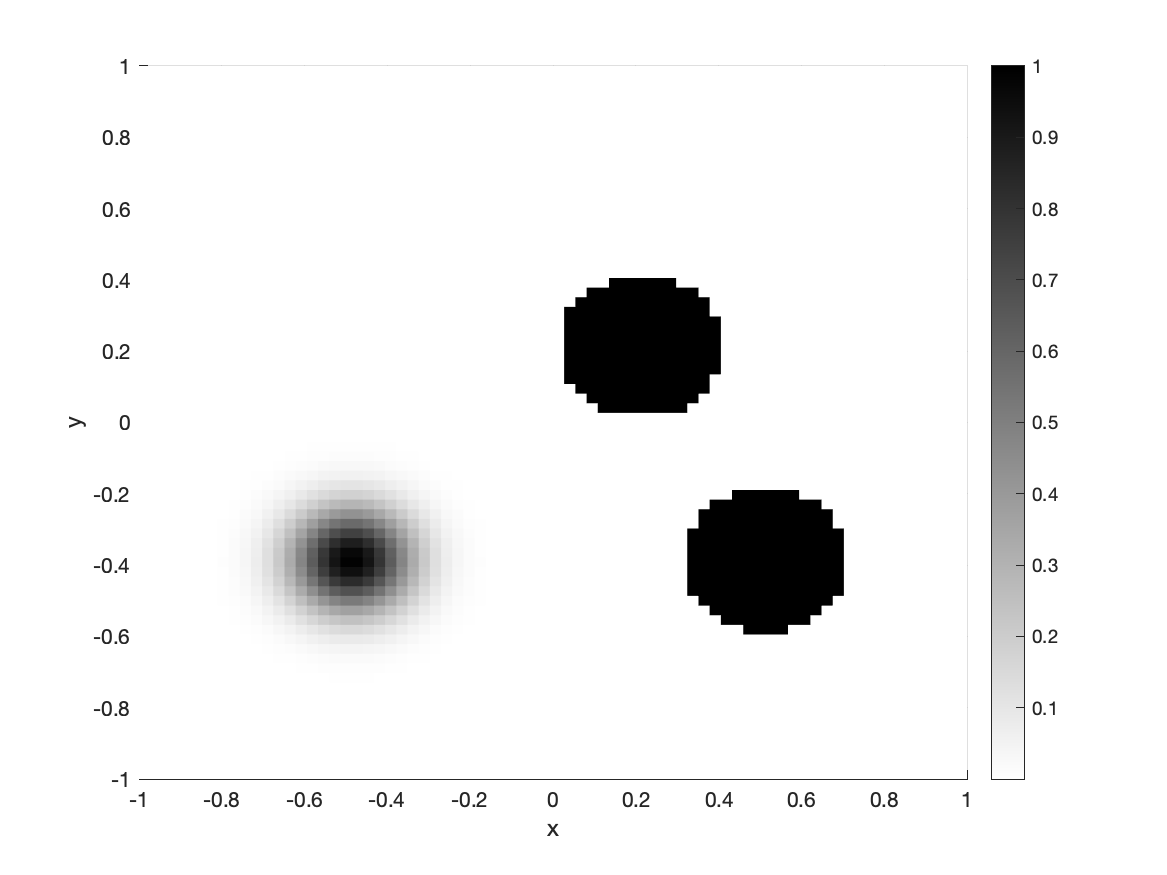} \label{fig:true3G}}
\subfloat[Time-reversal]{\includegraphics[width=0.24\textwidth, height=0.24\textwidth]{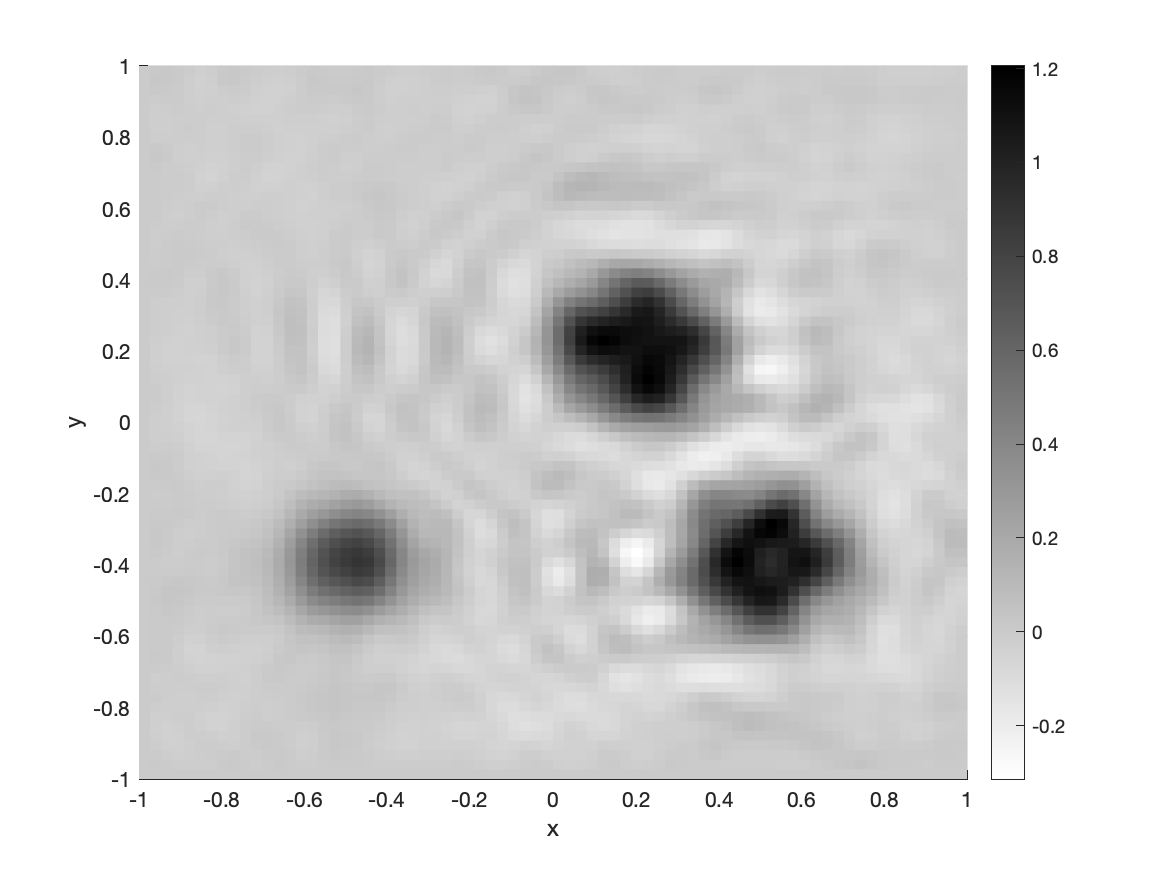}\label{fig:TR3G}}
\subfloat[CNN]{\includegraphics[width=0.24\textwidth, height=0.24\textwidth]{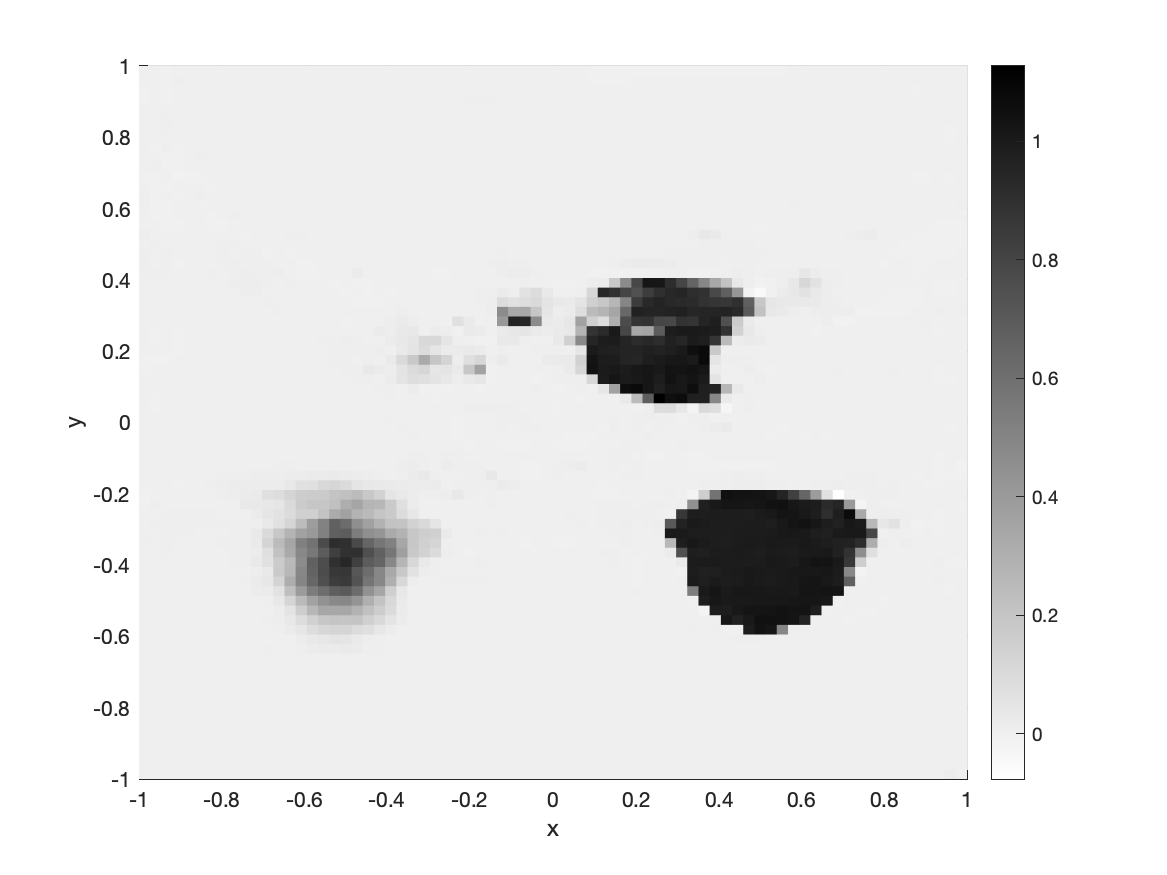}\label{fig:CNN3G}}
\subfloat[SQH]{\includegraphics[width=0.24\textwidth, height=0.24\textwidth]{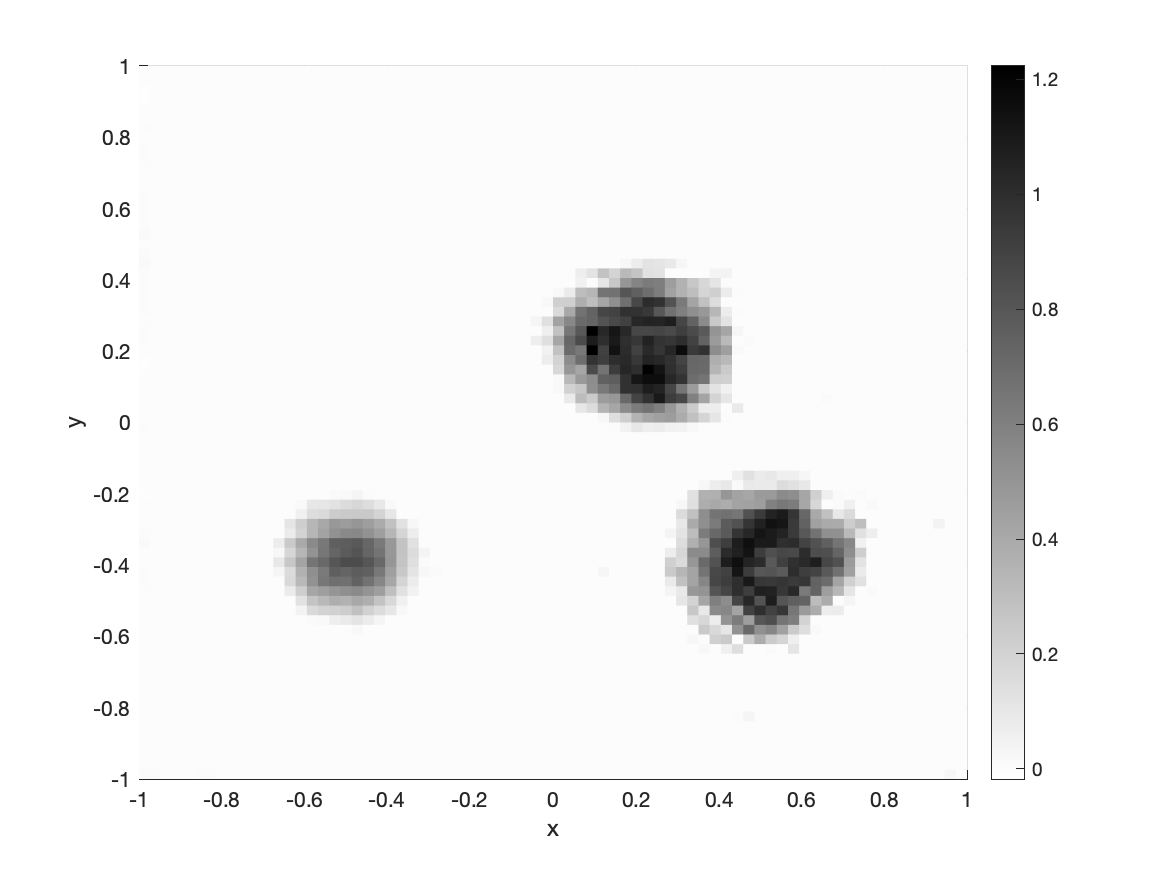}\label{fig:SQH3G}}\\
\subfloat[True]{\includegraphics[width=0.24\textwidth, height=0.24\textwidth]{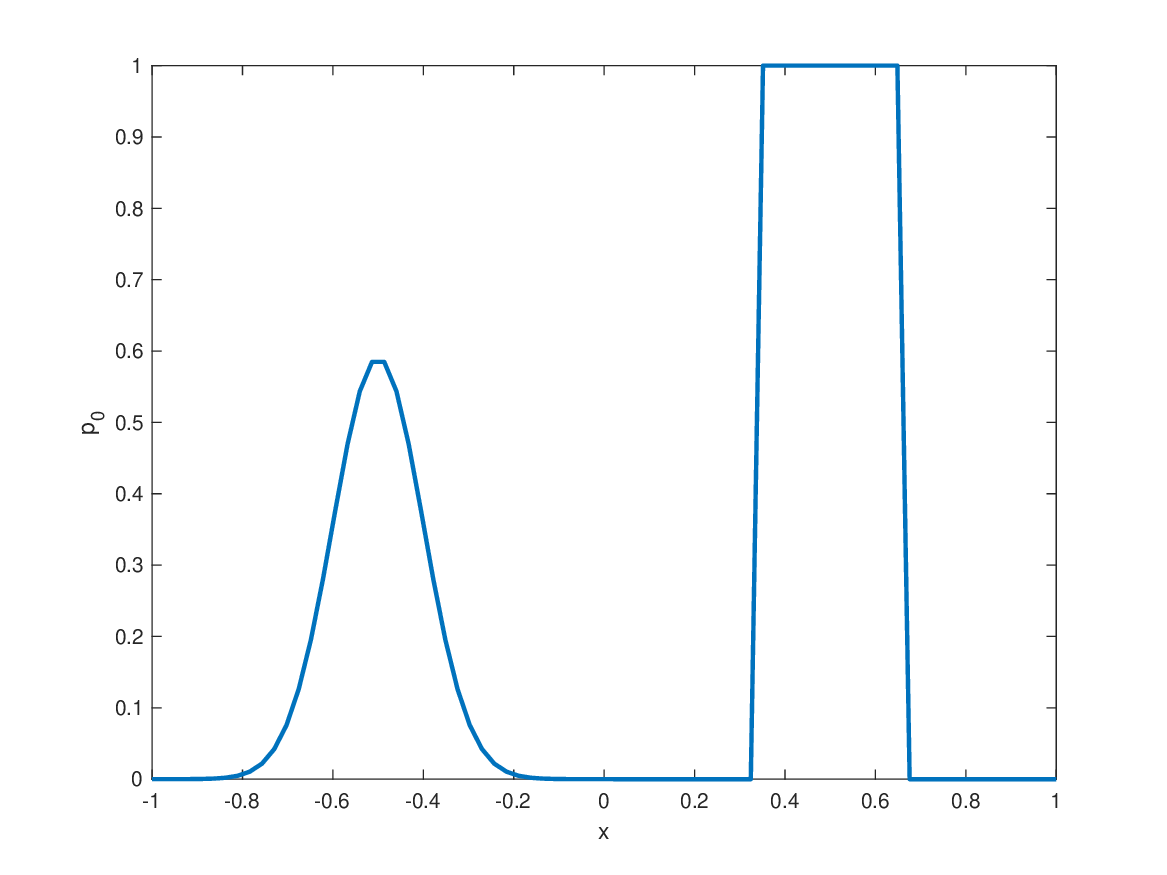} \label{fig:truecross3G}}
\subfloat[Time-reversal]{\includegraphics[width=0.24\textwidth, height=0.24\textwidth]{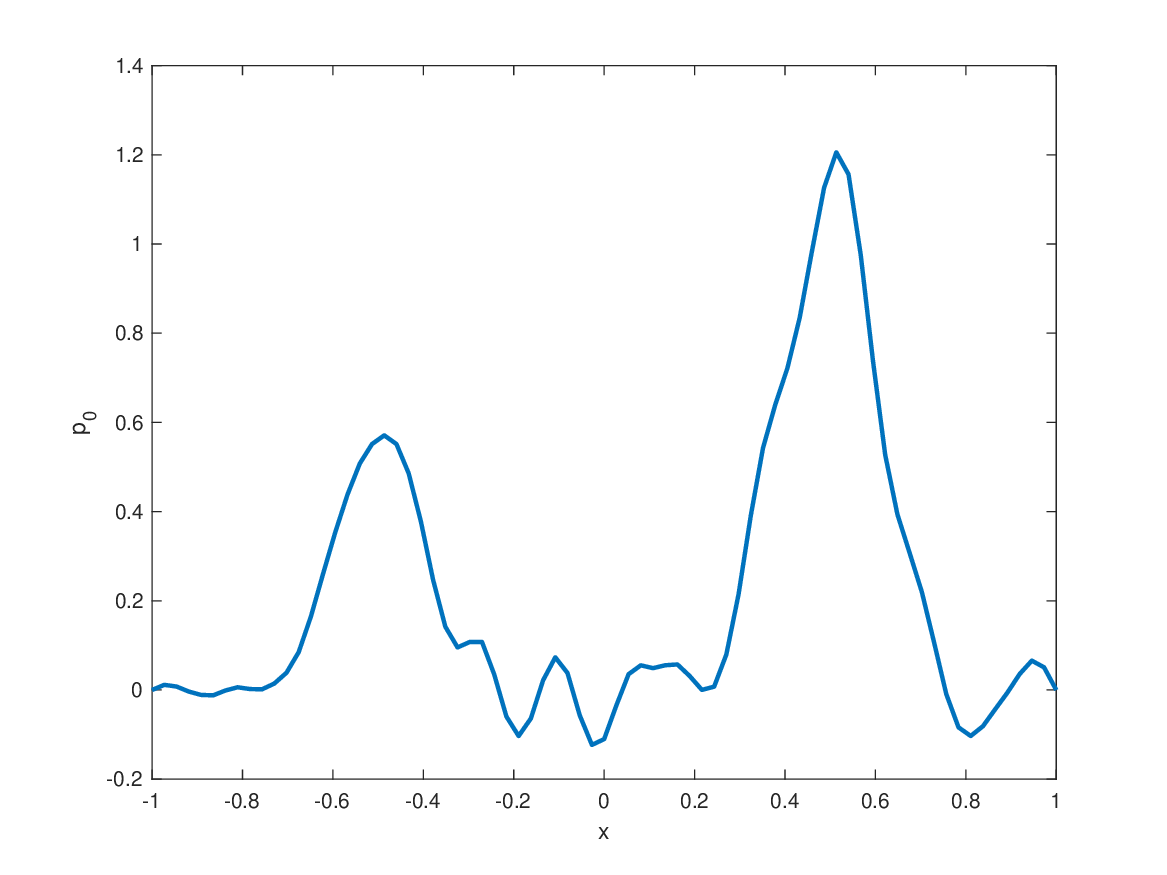}\label{fig:TRcross3G}}
\subfloat[CNN]{\includegraphics[width=0.24\textwidth, height=0.24\textwidth]{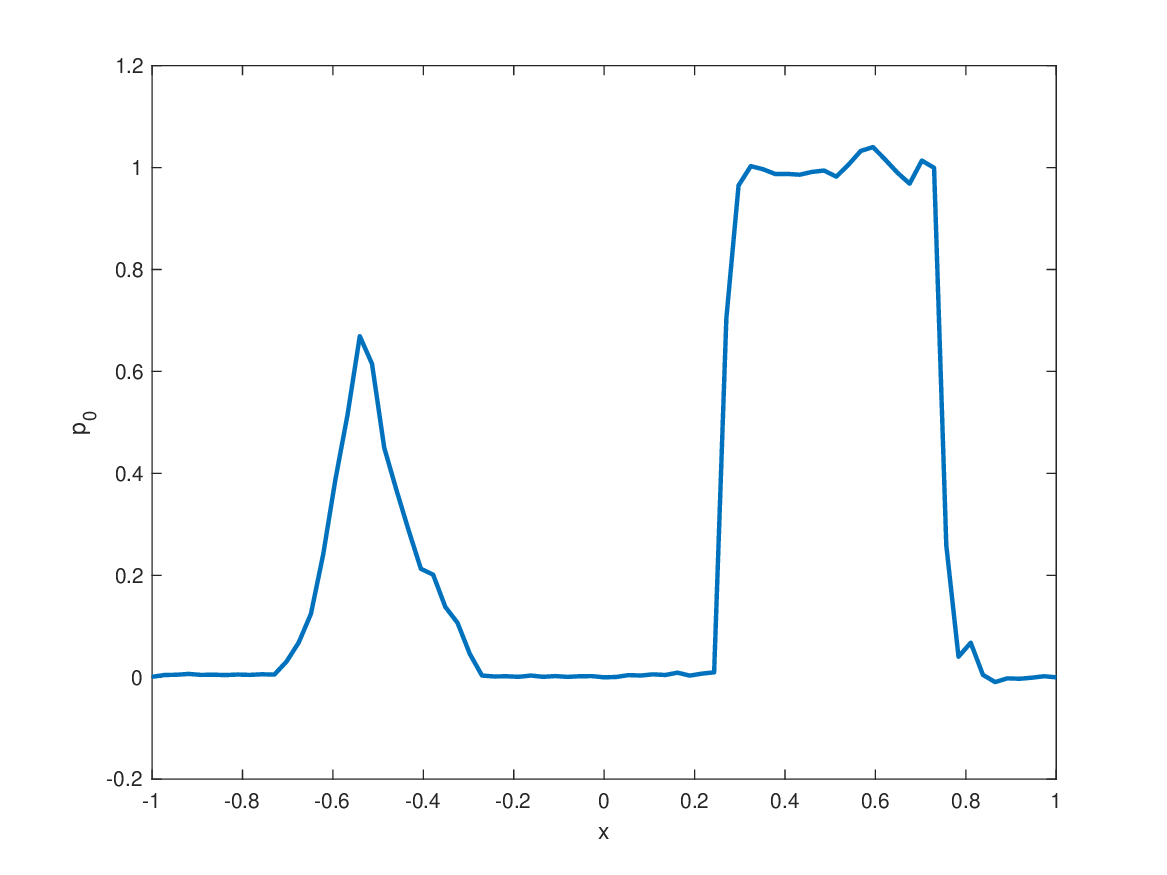}\label{fig:CNNcross3G}}
\subfloat[SQH]{\includegraphics[width=0.24\textwidth, height=0.24\textwidth]{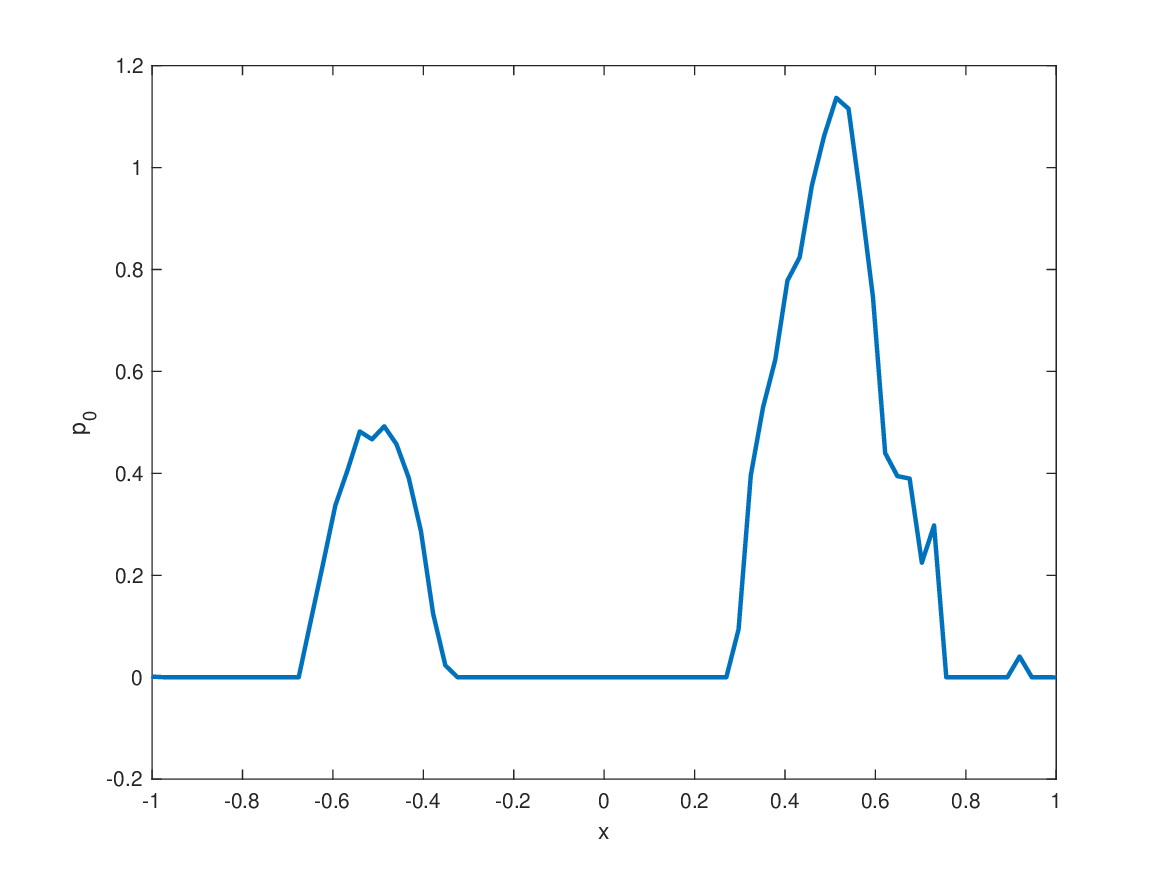}\label{fig:SQHcross3G}}
\caption{Test Case 2: Reconstructions with 3 phantoms; Top row corresponds to the 2D profiles; Bottom row corresponds to a 1D cross sectional profile}
\label{fig:3G}
  \end{figure}
We now notice the presence of severe artifacts in the reconstruction with the time-reversal algorithm, leading to poor resolution and contrast. Even though the CNN was not trained with 3 phantoms, its reconstruction does capture the disk discontinuity with a better contrast, but has other superficial structures present. The SQH algorithm is able to significantly clear off the artifacts due to the sparsity promoting feature and the contrast and resolution are far more better.

For the next test case, we consider a heart and lung phantom, represented by 2 ellipses and a disk. The reconstructions, with both 2D and cross-sectional 1D views, are shown in Figure \ref{fig:HL}.
  \begin{figure}[H]
\centering
\subfloat[True]{\includegraphics[width=0.24\textwidth, height=0.24\textwidth]{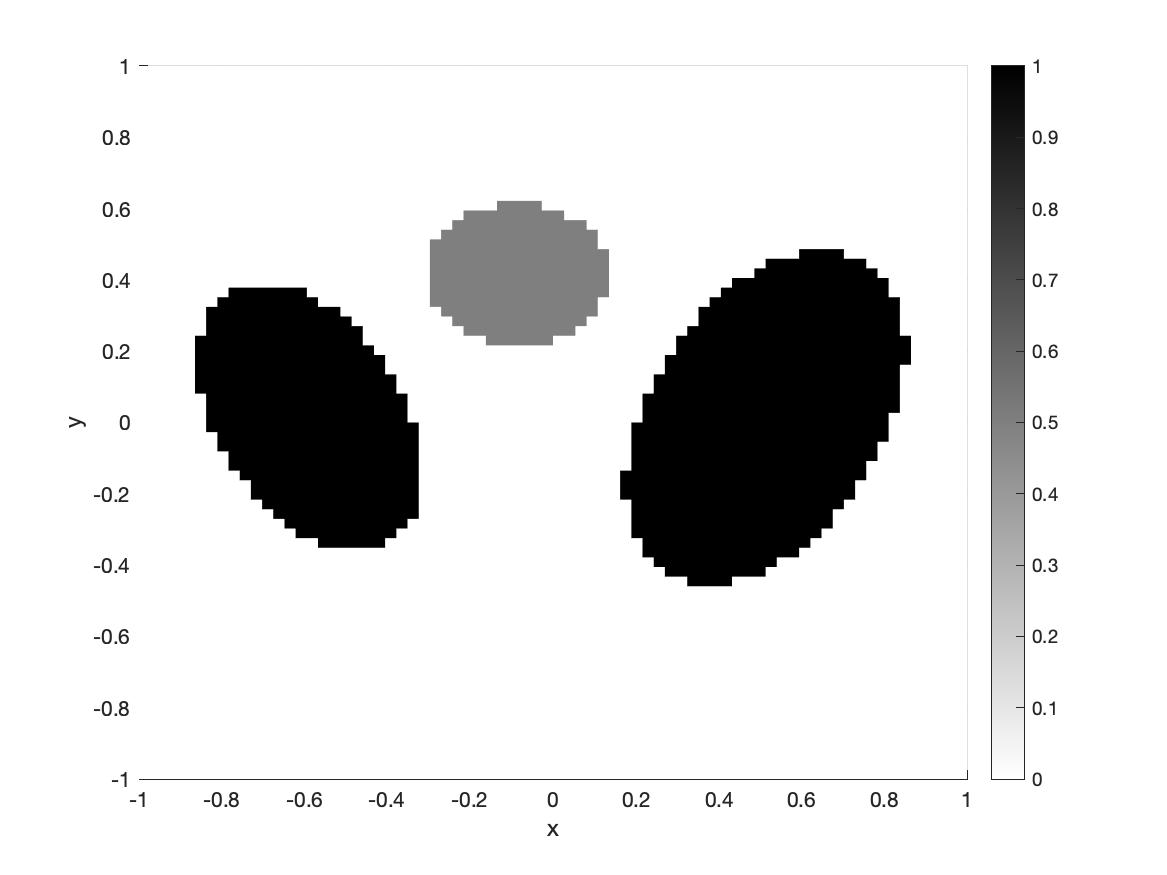} \label{fig:trueHL}}
\subfloat[Time-reversal]{\includegraphics[width=0.24\textwidth, height=0.24\textwidth]{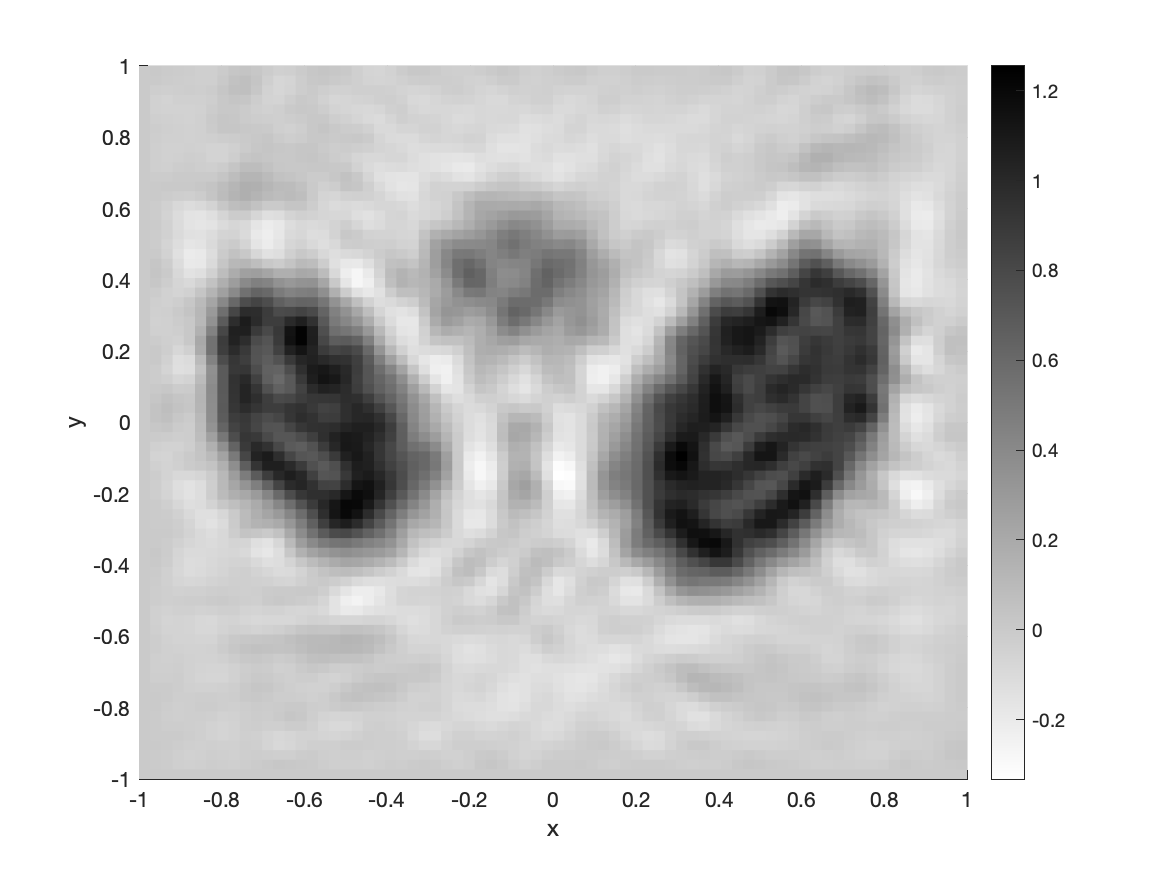}\label{fig:TRHL}}
\subfloat[CNN]{\includegraphics[width=0.24\textwidth, height=0.24\textwidth]{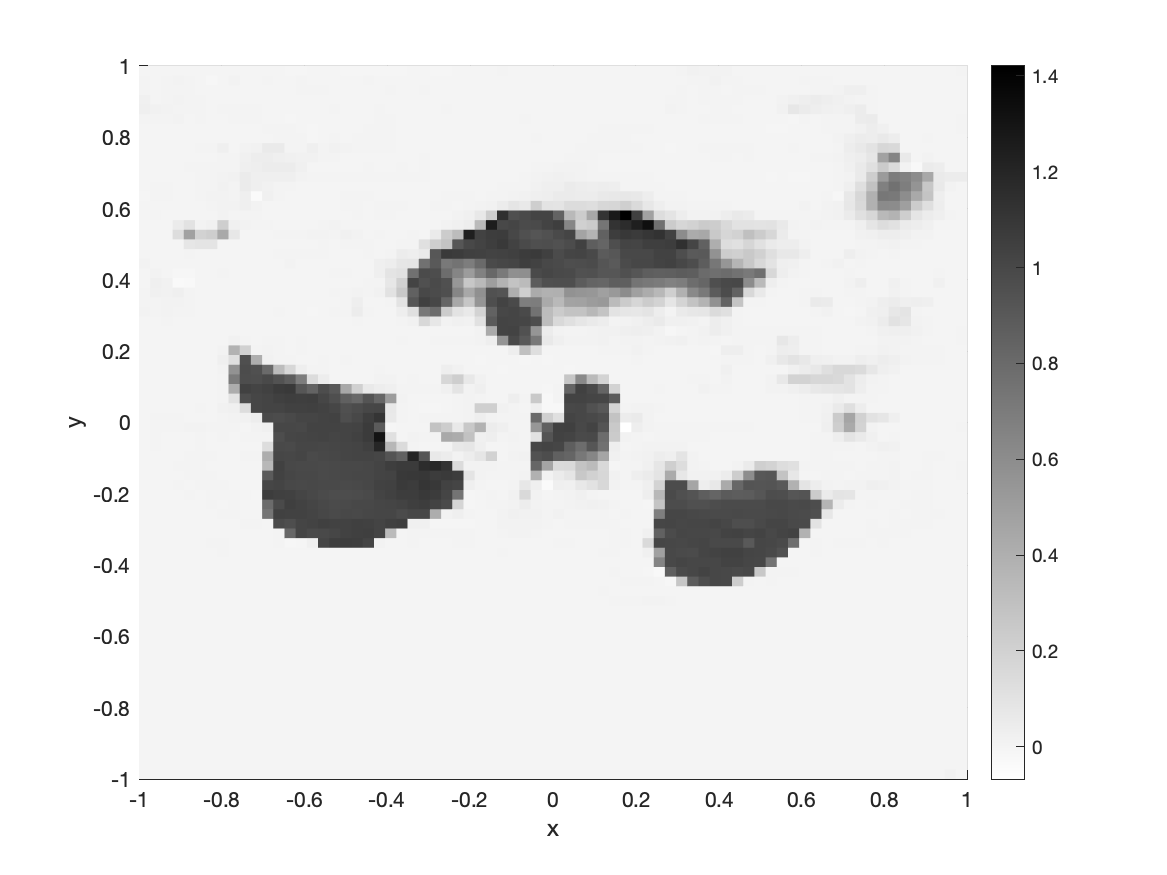}\label{fig:CNNHL}}
\subfloat[SQH]{\includegraphics[width=0.24\textwidth, height=0.24\textwidth]{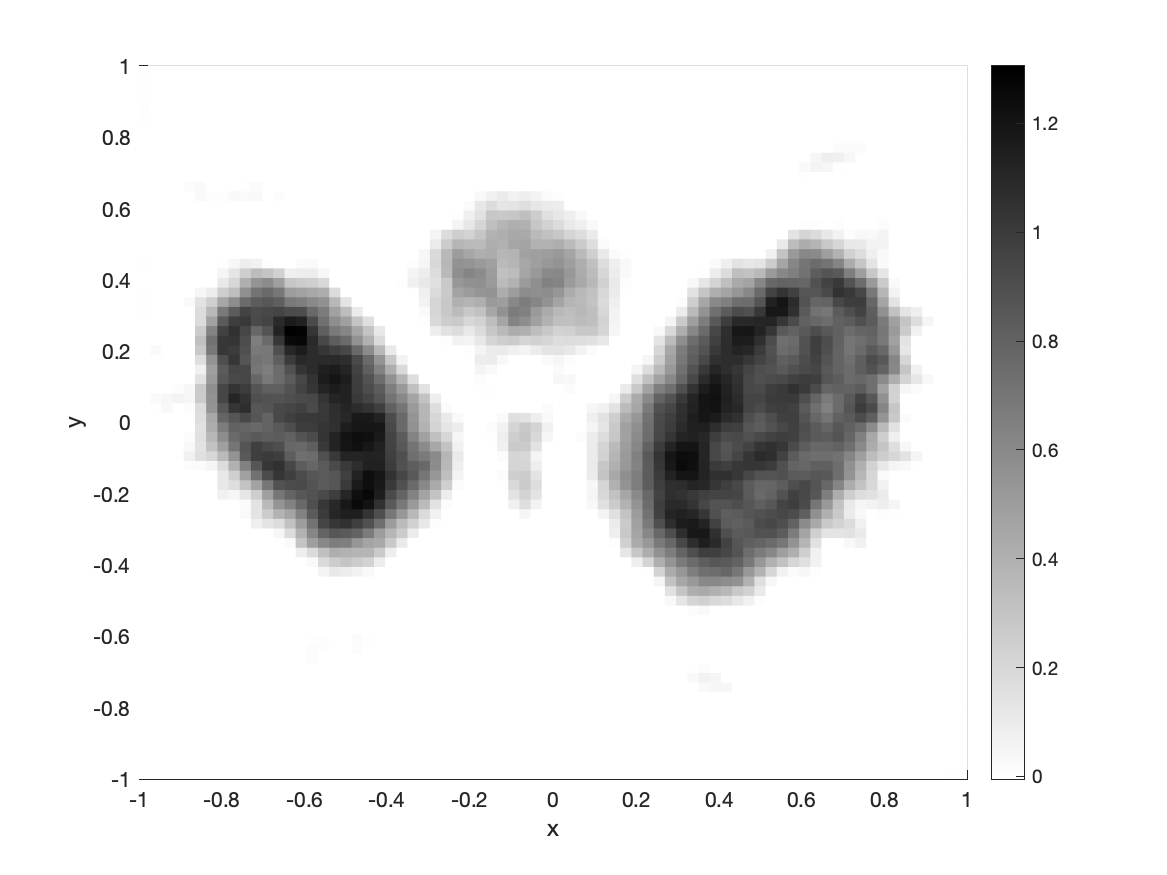}\label{fig:SQHHL}}\\
\subfloat[True]{\includegraphics[width=0.24\textwidth, height=0.24\textwidth]{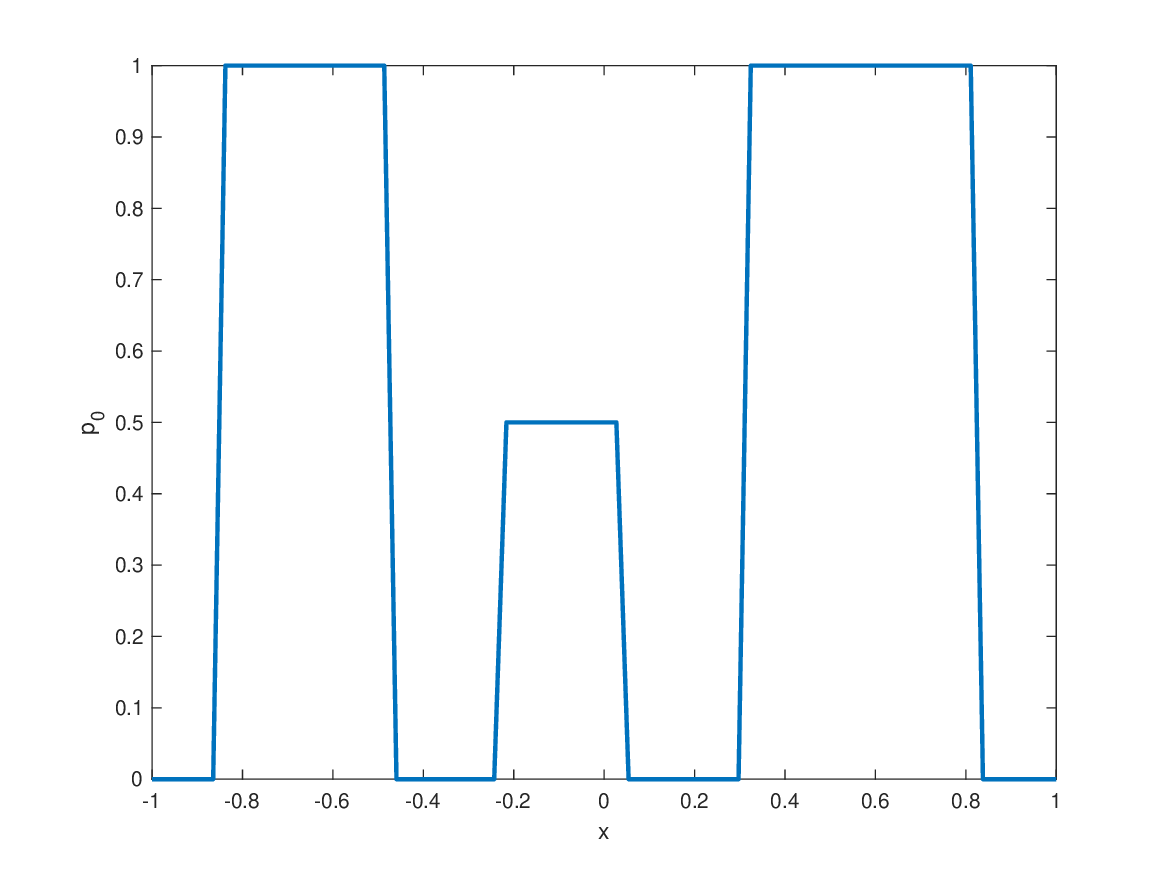} \label{fig:truecrossHL}}
\subfloat[Time-reversal]{\includegraphics[width=0.24\textwidth, height=0.24\textwidth]{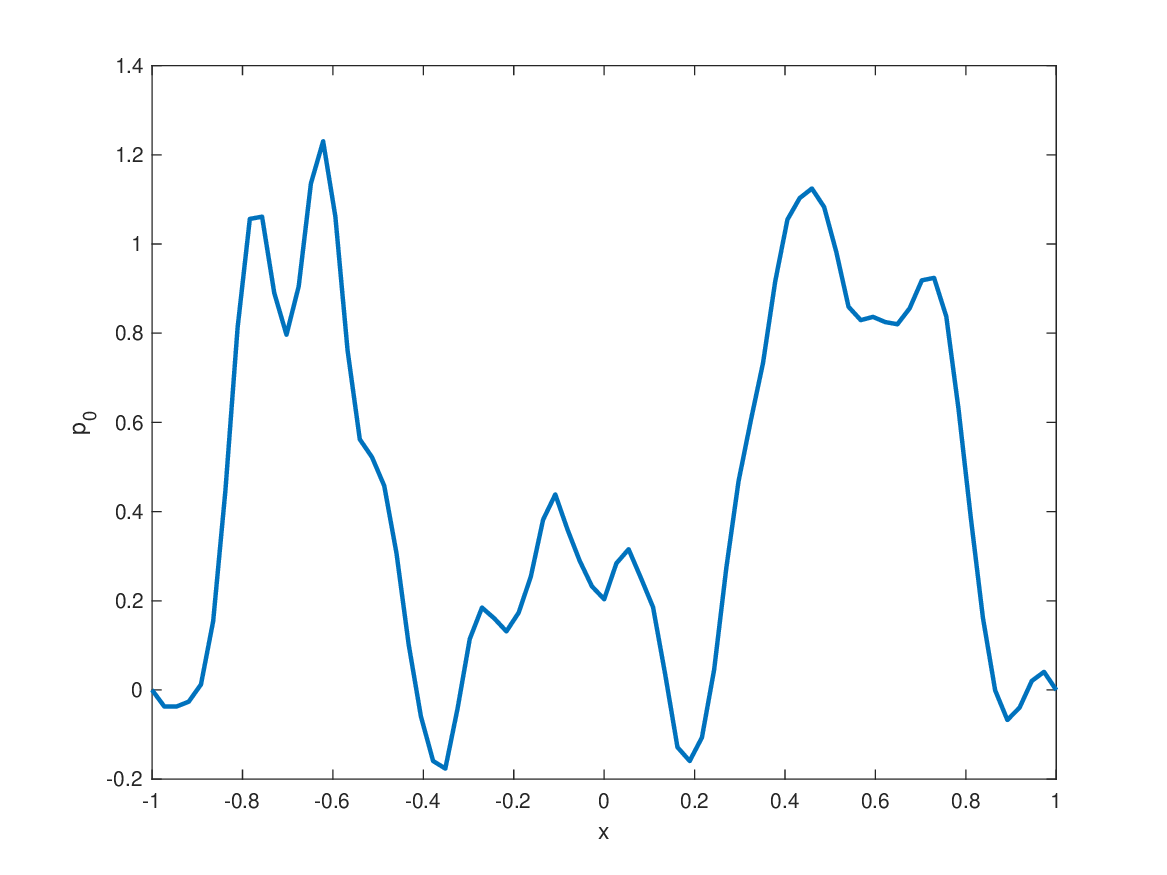}\label{fig:TRcrossHL}}
\subfloat[CNN]{\includegraphics[width=0.24\textwidth, height=0.24\textwidth]{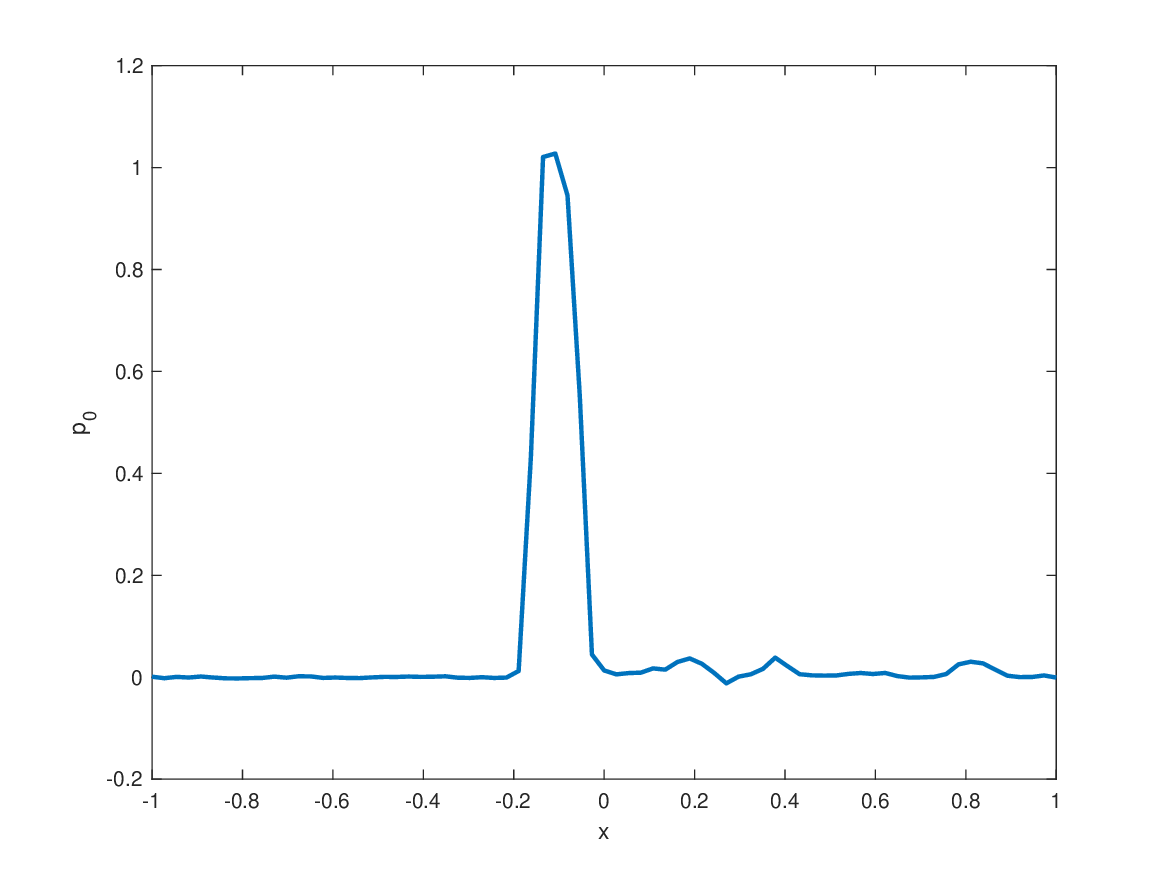}\label{fig:CNNcrossHL}}
\subfloat[SQH]{\includegraphics[width=0.24\textwidth, height=0.24\textwidth]{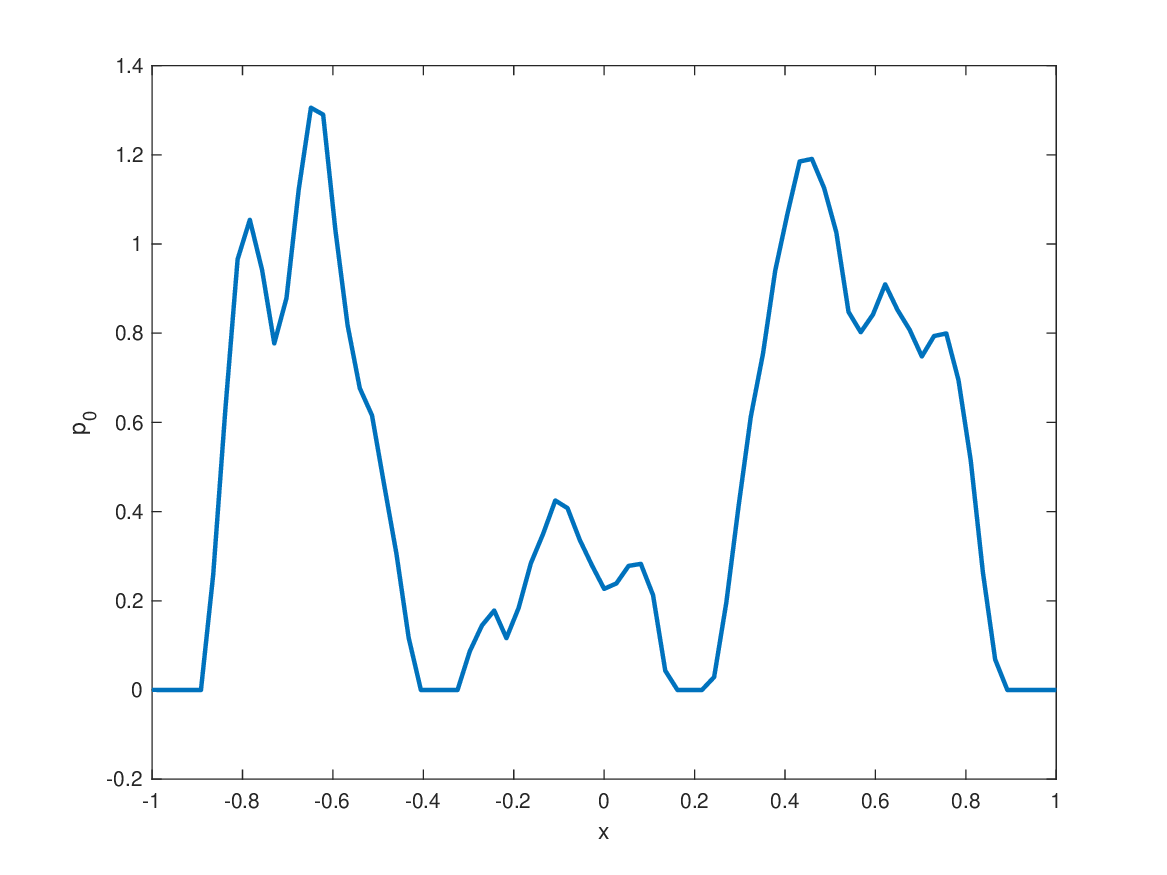}\label{fig:SQHcrossHL}}
\caption{Test Case 3: Reconstructions in 2D with the heart and lung phantom; Top row corresponds to the 2D profiles; Bottom row corresponds to a 1D cross sectional profile}
\label{fig:HL}
  \end{figure}

We again notice the presence of significant artifacts in the time-reversal reconstruction, leading to poor contrast and resolution. The CNN is expected to give a poor reconstruction with lot of superficial structures, since it was only trained on disks and Gaussians and at the most two objects, instead of ellipses and combination of three objects. However, some sharp edges in the CNN reconstruction help with the initial guess of the SQH algorithm, which now significantly improves the contrast and resolution significantly. This is also noted in the cross-sectional views of the reconstruction.

We now list the MSE, PSNR, and SSIM values of the reconstructions in Tables \ref{tab:MSE}, \ref{tab:PSNR}, and \ref{tab:SSIM}. 

\begin{table}[H]
\centering
\begin{tabular}{|c|c|c|c|c|c|}
\hline
Phantom & Test Case& TR & CNN &SQH \\ [0.5ex]
\hline
2 phantom &Test case 1  & 5.2e-3 & 4.3e-3 &3.9e-3 \\
3 phantom &Test case 2   & 1.1e-2 &1.5e-2 &8.8e-3  \\
Heart and lung &Test case 3 & 1.1e-2 &1.3e-2 & 8.9e-3   \\
[1ex]
\hline
\end{tabular}
\caption{MSE values for the various test cases.}
\label{tab:MSE}
\end{table}

\begin{table}[H]
\centering
\begin{tabular}{|c|c|c|c|c|c|}
\hline
Phantom & Test Case& TR & CNN &SQH \\ [0.5ex]
\hline
2 phantom &Test case 1  & 22.85 & 23.67 &24.05 \\
3 phantom &Test case 2   & 19.57 &18.11 &20.55 \\
Heart and lung &Test case 3 & 19.57 &8.86 & 20.51   \\
[1ex]
\hline
\end{tabular}
\caption{PSNR values for the various test cases.}
\label{tab:PSNR}
\end{table}

\begin{table}[H]
\centering
\begin{tabular}{|c|c|c|c|c|c|}
\hline
Phantom & Test Case& TR & CNN &SQH \\ [0.5ex]
\hline
2 phantom &Test case 1  & 0.5 & 0.94 &0.92 \\
3 phantom &Test case 2   & 0.32 &0.84 &0.86  \\
Heart and lung &Test case 3 & 0.33 &0.48 & 0.87   \\
[1ex]
\hline
\end{tabular}
\caption{SSIM values for the various test cases.}
\label{tab:SSIM}
\end{table}
We observe that the TR method gives high values of MSE and low values of PSNR. The MSE decreases and PSNR increases with the CNN, more so in the case of the 2 phantom scenario since it was trained on such datasets, but is significantly outperformed by the SQH algorithm. The striking feature is the SSIM values of the SQH algorithm which is close to 1 for all the test cases compared to the other two methods, further demonstrating the robustness and versatility of our proposed reconstruction framework.

For the next set of test cases, we now consider space-time varying functions for the damping and nonlinearity coefficients $d,\eta$. Specifically, we choose $d(x,t) = 0.1\exp(-t) + 0.001\exp(-x)$ and $\eta = 0.001 + 0.001\exp(-t) + 0.001\exp(-x)$. Our CNN is still trained on the original set of coefficients from the previous test cases. We now consider a combination of 3 phantoms: a Gaussian phantom centered at (-0.5,-0.5) and standard deviation 0.5, 2 disk phantoms centered at $(0.2,0.2)$ and $(0.5,-0.4)$ with radius 0.2 and intensity 1. The reconstructions are shown in Figure \ref{fig:3Gnew}.
  \begin{figure}[H]
\centering
\subfloat[True]{\includegraphics[width=0.24\textwidth, height=0.24\textwidth]{figures/3G_true.eps} \label{fig:true3G2}}
\subfloat[Time-reversal]{\includegraphics[width=0.24\textwidth, height=0.24\textwidth]{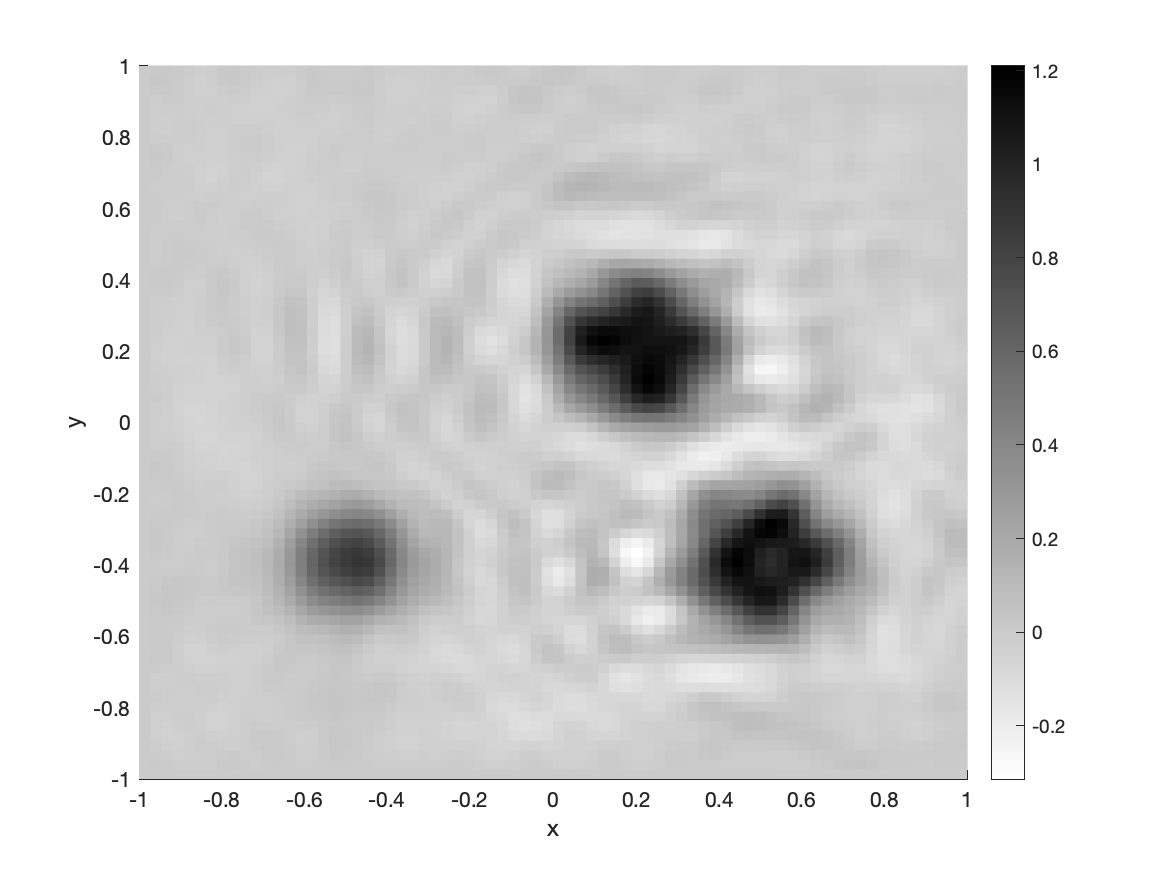}\label{fig:TR3Gn}}
\subfloat[CNN]{\includegraphics[width=0.24\textwidth, height=0.24\textwidth]{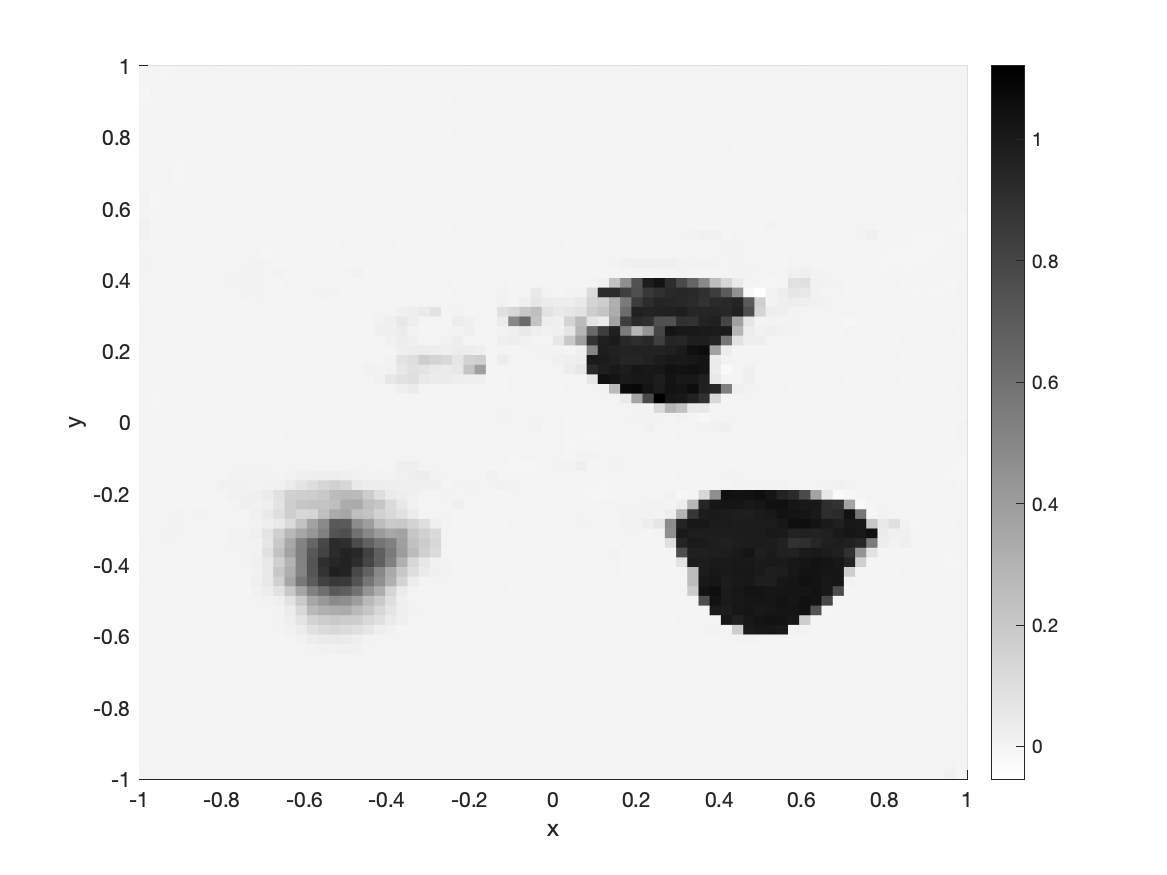}\label{fig:CNN3Gn}}
\subfloat[SQH]{\includegraphics[width=0.24\textwidth, height=0.24\textwidth]{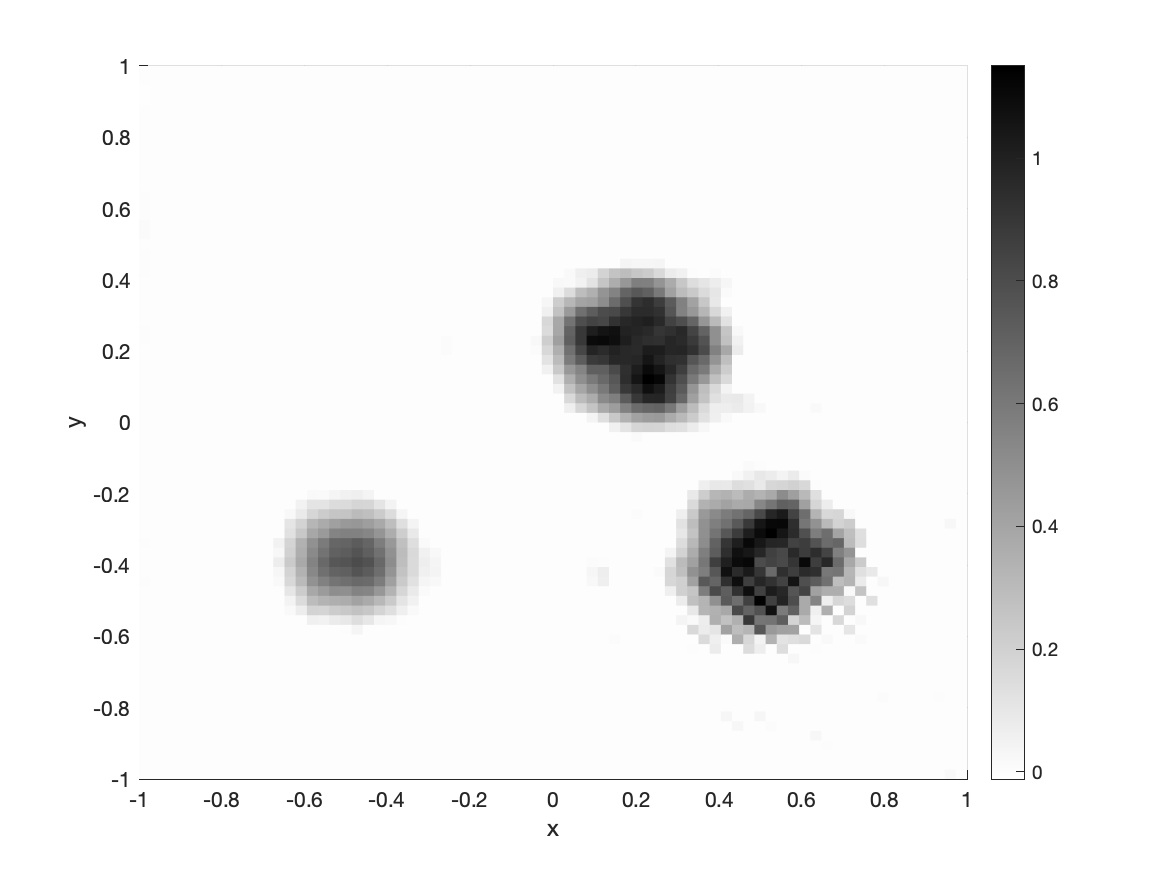}\label{fig:SQH3Gn}}
\caption{Test Case 4: Reconstructions with 3 phantoms with space-time varying coefficients}
\label{fig:3Gnew}
  \end{figure}
We again notice the presence of severe artifacts in the reconstruction with the time-reversal algorithm, leading to poor resolution and contrast. Even an imperfect trained CNN, its reconstruction does capture the disk discontinuity with a better contrast, with other superficial structures present. The SQH algorithm is again able to significantly clear off the artifacts due to the sparsity promoting feature and the contrast and resolution are far more better.

For the final test case, we consider a heart and lung phantom, represented by 2 ellipses and a disk. The reconstructions are shown in Figure \ref{fig:HLnew}.
  \begin{figure}[H]
\centering
\subfloat[True]{\includegraphics[width=0.24\textwidth, height=0.24\textwidth]{figures/HL_true.eps} \label{fig:trueHL2}}
\subfloat[Time-reversal]{\includegraphics[width=0.24\textwidth, height=0.24\textwidth]{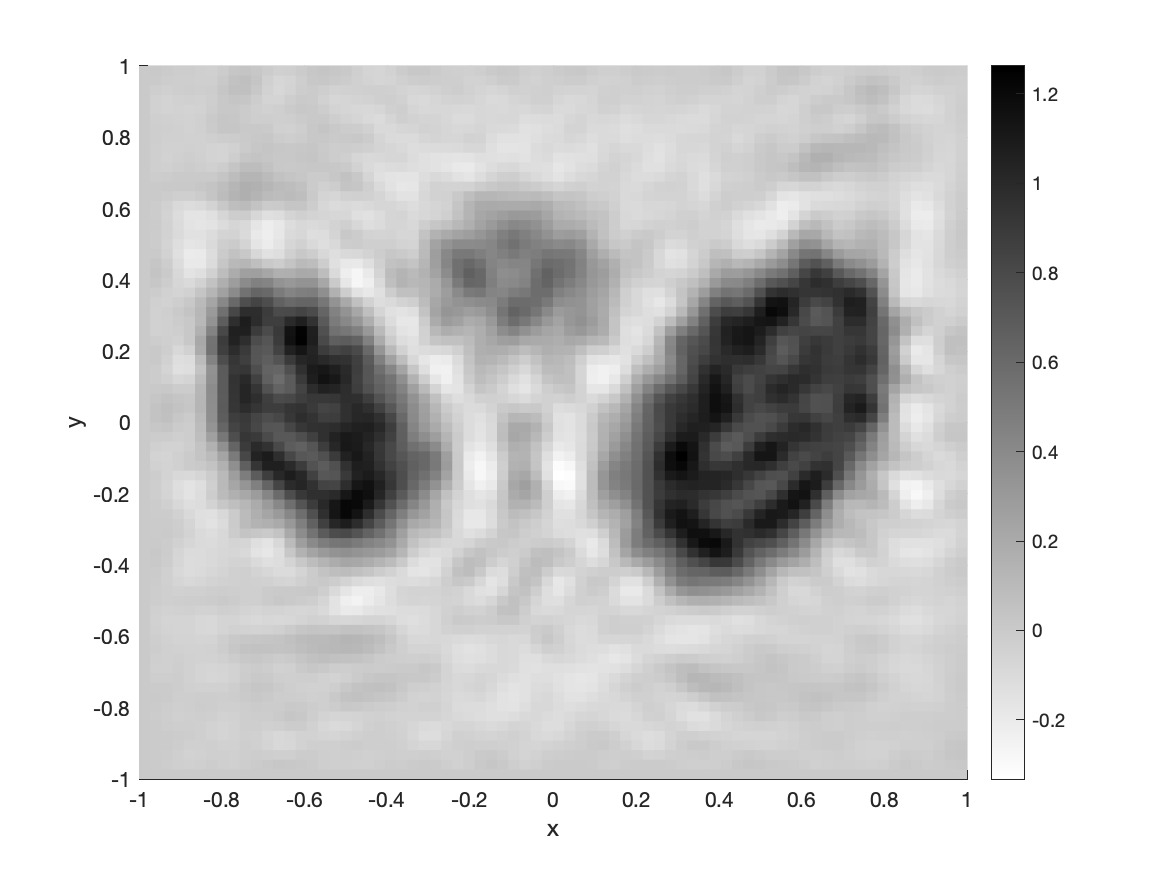}\label{fig:TRHLn}}
\subfloat[CNN]{\includegraphics[width=0.24\textwidth, height=0.24\textwidth]{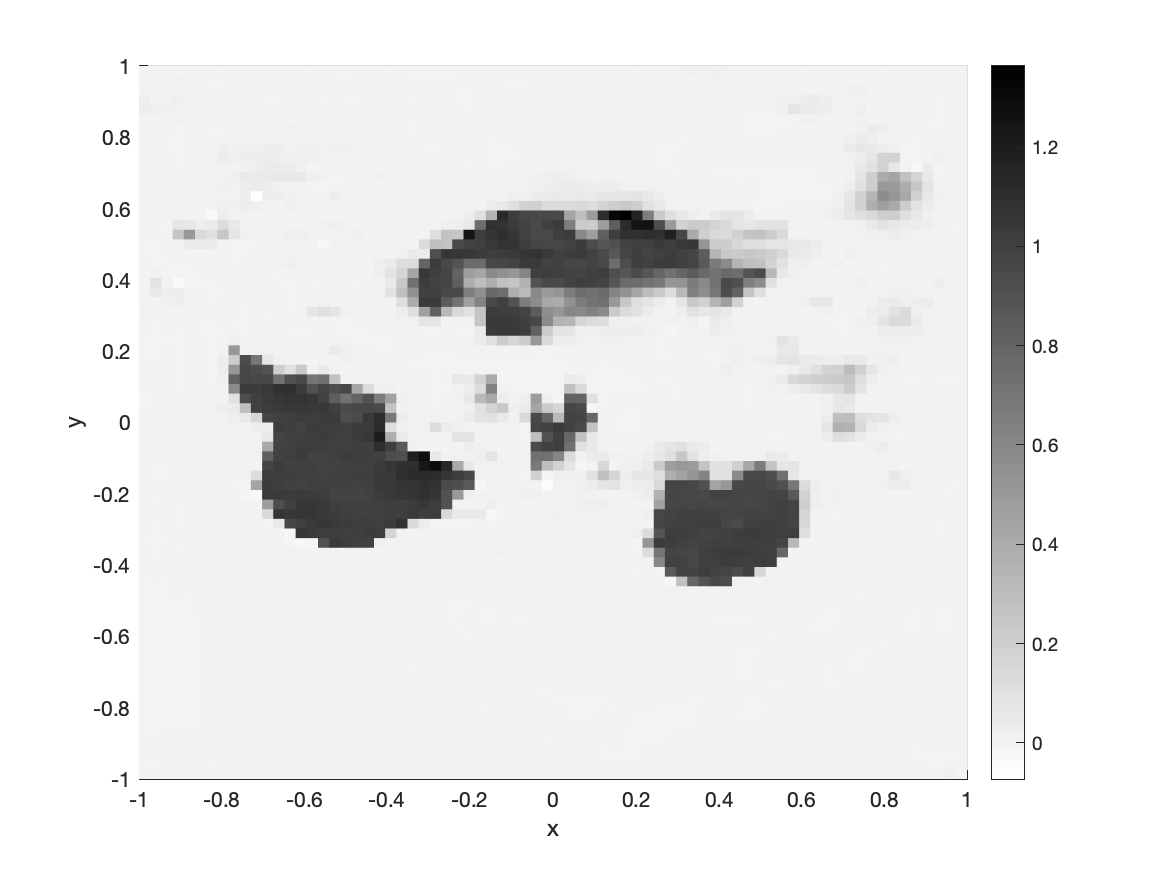}\label{fig:CNNHLn}}
\subfloat[SQH]{\includegraphics[width=0.24\textwidth, height=0.24\textwidth]{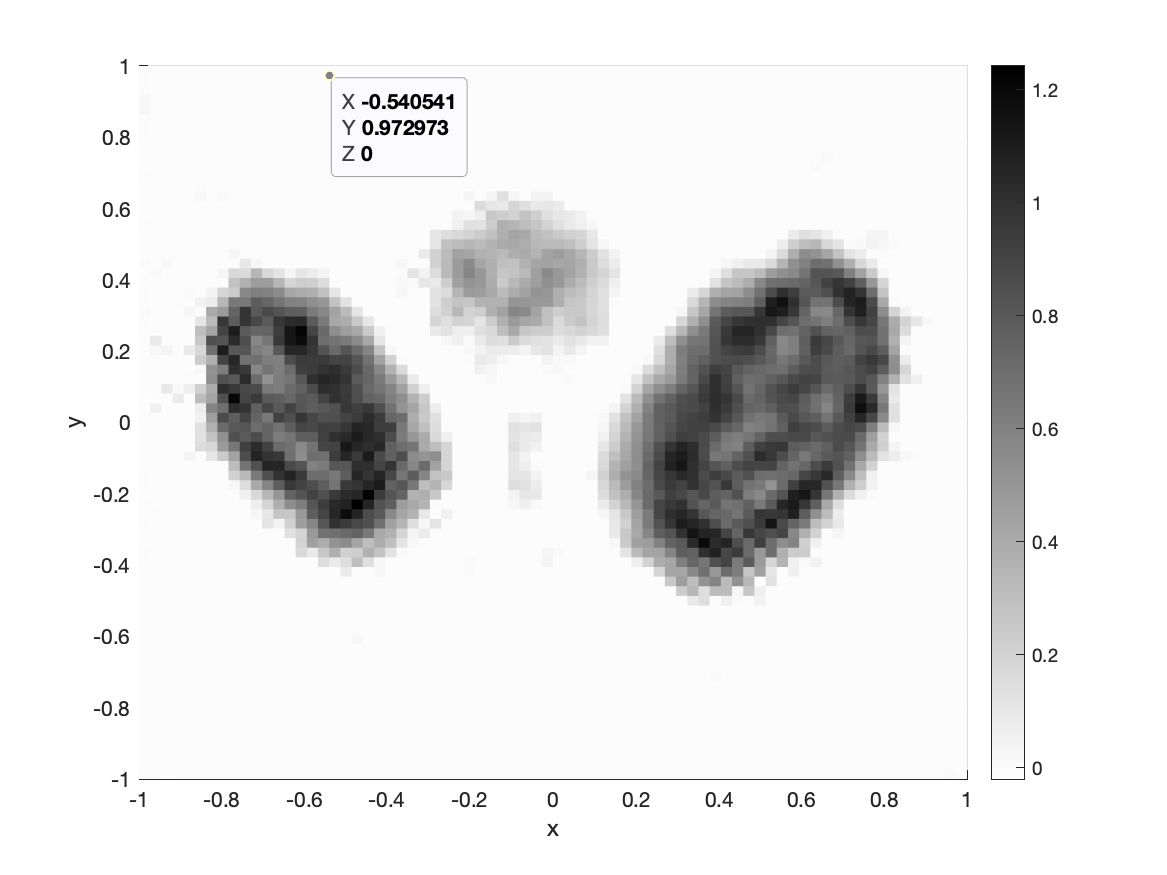}\label{fig:SQHHLn}}
\caption{Test Case 5: Reconstructions in 2D with the heart and lung phantom with space-time varying coefficients}
\label{fig:HLnew}
  \end{figure}
 
We again notice the presence of significant artifacts in the time-reversal reconstruction, leading to poor contrast and resolution. The CNN is expected to give a poor reconstruction with lot of superficial structures, since it was not  only trained on disks and Gaussians and at the most two objects, instead of ellipses and combination of three objects but also with data that was generated with simplified damping and nonlinearity coefficients. However, some sharp edges in the CNN reconstruction help with the initial guess of the SQH algorithm, which now significantly improves the contrast and resolution significantly.

We now list the MSE, PSNR, and SSIM values of the reconstructions with the space-time varying parameters in Tables \ref{tab:MSEnew}, \ref{tab:PSNRnew}, and \ref{tab:SSIMnew}. 

\begin{table}[H]
\centering
\begin{tabular}{|c|c|c|c|c|c|}
\hline
Phantom & Test Case& TR & CNN &SQH \\ [0.5ex]
\hline
3 phantom &Test case 4   & 1.1e-2 &1.5e-2 &1.0e-2  \\
Heart and lung &Test case 5 & 2.5e-2 &1.4e-1 &2.4e-3   \\
[1ex]
\hline
\end{tabular}
\caption{MSE values for the space-time varying coefficients test cases.}
\label{tab:MSEnew}
\end{table}

\begin{table}[H]
\centering
\begin{tabular}{|c|c|c|c|c|c|}
\hline
Phantom & Test Case& TR & CNN &SQH \\ [0.5ex]
\hline
3 phantom &Test case 4   & 19.57 &18.11 &19.86 \\
Heart and lung &Test case 5 & 15.94 &8.54 & 16.12   \\
[1ex]
\hline
\end{tabular}
\caption{PSNR values for the space-time varying coefficients test cases.}
\label{tab:PSNRnew}
\end{table}

\begin{table}[H]
\centering
\begin{tabular}{|c|c|c|c|c|c|}
\hline
Phantom & Test Case& TR & CNN &SQH \\ [0.5ex]
\hline
3 phantom &Test case 4   & 0.32 &0.84 &0.86  \\
Heart and lung &Test case 5 & 0.19 &0.56 & 0.72  \\
[1ex]
\hline
\end{tabular}
\caption{SSIM values for the space-time varying coefficients test cases.}
\label{tab:SSIMnew}
\end{table}
We again observe that the TR method gives high values of MSE and low values of PSNR. The MSE decreases and PSNR increases with the CNN, more so in the case of the 2 phantom scenario since it was trained on such datasets, but is significantly outperformed by the SQH algorithm. This demonstrates the robustness and versatility of our proposed reconstruction framework.

What we observe from these experiments is that a significantly well trained CNN combined with the time-reversal method can be very effective to generate initial guesses for iterative reconstruction algorithms arising in nonlinear PAT mechanisms. This framework can also be further extended to other nonlinear reconstruction frameworks which are governed by PDE dynamics and where traditional methods are either not applicable or yield significantly poor reconstruction results. 

\section{Conclusion}\label{sec: Conclusion}

In this work, we investigated the inverse problem of reconstructing the initial pressure distribution in photoacoustic tomography governed by a nonlinear damped viscoelastic wave equation. Unlike standard acoustic models, the proposed framework incorporates nonlinear propagation effects, viscous damping, temporal attenuation, and spatially varying acoustic parameters, leading to a more realistic mathematical description of wave propagation in biological media. From the theoretical perspective, we established the well-posedness of the nonlinear forward problem using Galerkin approximations, energy estimates, and a fixed-point argument. For the inverse problem, we derived existence, uniqueness, and local uniqueness results under suitable assumptions through harmonic extension methods, spectral Laplace transform techniques, and observability estimates. 

From the computational perspective, we proposed a hybrid reconstruction framework combining a convolutional neural network with a gradient-free optimization strategy based on the sequential quadratic Hamiltonian (SQH) method derived from Pontryagin’s maximum principle. The CNN provides an informative data-driven initialization, while the SQH algorithm enforces the underlying PDE dynamics and improves physical consistency of the reconstruction. Numerical experiments in two spatial dimensions demonstrated that the proposed hybrid framework produces improved contrast, resolution, and reconstruction accuracy compared to standalone time-reversal and CNN-based reconstructions, particularly in the presence of attenuation and nonlinear effects.

The results indicate that combining machine learning with structure-preserving optimization techniques provides an effective framework for solving nonlinear and ill-posed inverse problems arising in PAT. Future work will focus on extending the framework to developing more advanced physics-informed and uncertainty-aware reconstruction strategies for complex heterogeneous media.

\section*{Acknowledgment}
The work of S. Roy was supported by the US National Science Foundation Grant No. DMS-2309491.

\bibliographystyle{unsrt}
\bibliography{Ref}

@article{conti2014well,
author = {Conti, Monica and Marchini, Elsa M. and Pata, Vittorino},
title = {A well-posedness result for nonlinear viscoelastic equations with memory},
journal = {Nonlinear Analysis: Theory, Methods and Applications},
volume = {94},
pages = {206--216},
year = {2014}
}

@article{landes1980galerkin,
  author = {Landes, R.},
  title = {On Galerkin's Method in the Existence Theory of Quasilinear Elliptic Equations},
  journal = {Analysis},
  volume = {39},
  pages = {123--148},
  year = {1980},
}

@article{ferreira2023global,
author = {Ferreira, Jorge and Shahrouzi, Mohammad and Aitzhanov, Serik E. and Rocha, Fabio},
title = {Global existence, uniqueness and asymptotic behavior for a nonlinear viscoelastic problem with internal damping and logarithmic source term},
journal = {Differential Equations \& Applications},
volume = {15},
number = {4},
pages = {395--429},
year = {2023},
doi = {10.7153/dea-2023-15-20},
}

@article{Roy2023,
  title={Radon transform with Gaussian beam: Theoretical and numerical reconstruction scheme},
  author={Roy, S. and Jeon, G. and Moon, S.},
  journal={Applied Mathematics and Computation},
  volume={452},
  pages={128024},
  year={2023},
  publisher={Elsevier}
}

@article{Roy2022,
  title={A New Nonlinear Sparse Optimization Framework in Ultrasound-Modulated Optical Tomography},
  author={Roy, S.},
  journal={IEEE Transactions on Computational Imaging},
  volume={8},
  pages={1--11},
  year={2022},
  publisher={IEEE}
}

@article{Roy2025,
  title={A {PINN}-driven game-theoretic framework in limited data photoacoustic tomography},
  author={Roy, S. and Pal, S.},
  journal={Inverse Problems},
  year={2025}
}

@article{breit19a,
  title={A sequential quadratic Hamiltonian scheme for solving non-smooth quantum control problems with sparsity},
  author={Breitenbach, T. and Borz{\`\i}, A.},
  journal={Journal of Computational and Applied Mathematics},
  volume={369},
  pages={112583},
  year={2020},
  publisher={Elsevier}
}

@article{breit19,
  title={On the {SQH} scheme to solve nonsmooth PDE optimal control problems},
  author={Breitenbach, T. and Borz{\`\i}, A.},
  journal={Numerical Functional Analysis and Optimization},
  volume={40},
  number={13},
  pages={1489--1531},
  year={2019},
  publisher={Taylor \& Francis}
}

@article{breit18,
  title={A sequential quadratic Hamiltonian method for solving parabolic optimal control problems with discontinuous cost functionals},
  author={Breitenbach, T. and Borzi, A.},
  journal={Journal of Dynamical and Control Systems},
  volume={25},
  number={3},
  pages={403--435},
  year={2019},
  publisher={Springer}
}

@book{Borzi2023,
  title={The Sequential Quadratic Hamiltonian Method: Solving Optimal Control Problems},
  author={Borz{\`\i}, A.},
  series={Chapman and Hall/CRC Numerical Analysis and Scientific Computing Series},
  year={2023},
  publisher={Taylor \& Francis}
}

@article{moonip24,
doi = {10.1088/1361-6420/ad42cd},
url = {https://dx.doi.org/10.1088/1361-6420/ad42cd},
year = {2024},
month = {may},
publisher = {IOP Publishing},
volume = {40},
number = {6},
pages = {065009},
author = {Kim, S. and Moon, S. and Seo, I.},
title = {Reconstruction of the initial data from the trace of the solutions on an infinite time cylinder of damped wave equations},
journal = {Inverse Problems}
}

@article{moonzangerlip19,
	doi = {10.1088/1361-6420/ab2a30},
	year = 2019,
	month = {oct},
	publisher = {{IOP} Publishing},
	volume = {35},
	number = {11},
	pages = {114005},
	author = {G. Zangerl and S. Moon and M. Haltmeier},
	title = {Photoacoustic tomography with direction dependent data: an exact series reconstruction approach},
	journal = {Inverse Problems}
}

@article{anastasiozmr07,
  author={M. A Anastasio and J. Zhang and D. Modgil and P. J La Rivi\`{e}re},
  title={Application of inverse source concepts to photoacoustic tomography},
  journal={Inverse Problems},
  volume={23},
  number={6},
  pages={S21},
  url={http://stacks.iop.org/0266-5611/23/i=6/a=S03},
  year={2007}
}

@article{ammaribjk10,
author = {H. Ammari and E. Bossy and V. Jugnon and H. Kang},
title = {Mathematical Modeling in Photoacoustic Imaging of Small Absorbers},
journal = {SIAM Review},
volume = {52},
number = {4},
pages = {677-695},
year = {2010},
doi = {10.1137/090748494},
URL = {         http://dx.doi.org/10.1137/090748494},
eprint = {         http://dx.doi.org/10.1137/090748494}
}

@article{Gao12,
author = {Gao, H. and Osher, S. and Zhao, H.},
journal = {Mathematical Modeling in Biomedical Imaging II, Lecture Notes in Mathematics 2035, Springer, Heidelberg},
pages = {131--158},
title = {{Quantitative Photoacoustic Tomography}},
year = {2012}
}

@misc{dohknn24,
      title={Full Field Inversion of the Attenuated Wave Equation: Theory and Numerical Inversion}, 
      author={N. Do and M. Haltmeier and R. Kowar and L. V. Nguyen and R. Nuster},
      year={2024},
      eprint={2405.18082},
      archivePrefix={arXiv},
      primaryClass={math.NA},
      url={https://arxiv.org/abs/2405.18082}, 
}

@book{scherzer15,
  title      = {Handbook of Mathematical Methods in Imaging},
  address    = {New York},
  doi        = {10.1007/978-1-4939-0790-8},
  edition    = {2},
  editor     = {Scherzer, O.},
  isbn       = {978-1-4939-0789-2},
  publisher  = {Springer},
  year       = {2015},
}

@article{Dey2024,
  title={A high contrast and resolution reconstruction algorithm in quantitative photoacoustic tomography},
  author={Dey, A. and Borz{\`\i}, A. and Roy, S.},
  journal={Journal of Computational and Applied Mathematics},
  pages={116065},
  year={2024},
  publisher={Elsevier}
}

@article{bell80,
  title={On the Production and Reproduction of Sound by Light},
  author={Bell, A.G.},
  url={http://books.google.com/books?id=iq9TGwAACAAJ},
  year={1880},
  pages="305-324",
  month={October},
  Journal="American Journal of Science",
  volume="20",
}

@article {elbauss17,
    AUTHOR = {P. Elbau and O. Scherzer and C. Shi},
     TITLE = {Singular values of the attenuated photoacoustic imaging
              operator},
   JOURNAL = {J. Differential Equations},
  FJOURNAL = {Journal of Differential Equations},
    VOLUME = {263},
      YEAR = {2017},
    NUMBER = {9},
     PAGES = {5330--5376},
      ISSN = {0022-0396,1090-2732},
   MRCLASS = {35L15 (35A17 35B30 35R25 35S05 94A08)},
  MRNUMBER = {3688416},
MRREVIEWER = {Michael\ Ruzhansky},
       DOI = {10.1016/j.jde.2017.06.018},
       URL = {https://doi.org/10.1016/j.jde.2017.06.018},
}

@article{haltmeiern19,
  title={Reconstruction algorithms for photoacoustic tomography in heterogeneous damping media},
  author={Haltmeier, M. and Nguyen, L. V},
  journal={Journal of Mathematical Imaging and Vision},
  volume={61},
  number={7},
  pages={1007--1021},
  year={2019},
  publisher={Springer}
}

@article{homan13,
title = {Multi-wave imaging in attenuating media},
journal = {Inverse Problems and Imaging},
volume = {7},
number = {4},
pages = {1235-1250},
year = {2013},
issn = {1930-8337},
doi = {10.3934/ipi.2013.7.1235},
url = {https://www.aimsciences.org/article/id/bab0adb9-ba16-4214-86da-5e9108cb0102},
author = {A. Homan}
}

@article{nachmansw90,
  title={An equation for acoustic propagation in inhomogeneous media with relaxation losses},
  author={A. I. Nachman and J.F. Smith III and R. C. Waag},
  journal={The Journal of the Acoustical Society of America},
  volume={88},
  number={3},
  pages={1584--1595},
  year={1990},
  publisher={Acoustical Society of America}
}

@incollection {kowars12,
    AUTHOR = {R. Kowar and O. Scherzer},
     TITLE = {Attenuation models in photoacoustics},
 BOOKTITLE = {Mathematical modeling in biomedical imaging. {II}},
    SERIES = {Lecture Notes in Math.},
    VOLUME = {2035},
     PAGES = {85--130},
 PUBLISHER = {Springer, Heidelberg},
      YEAR = {2012},
      ISBN = {978-3-642-22989-3; 978-3-642-22990-9},
   MRCLASS = {92C55},
  MRNUMBER = {3024671},
       DOI = {10.1007/978-3-642-22990-9\_4},
       URL = {https://doi.org/10.1007/978-3-642-22990-9_4},
}

@article{moonjop16,
   author = {{Haltmeier}, M. and {Berer}, T. and {Moon}, S. and {Burgholzer}, P.},
  title={Compressed sensing and sparsity in photoacoustic tomography},
  journal={Journal of Optics},
  volume={18},
  number={11},
  pages={114004},
  url={http://stacks.iop.org/2040-8986/18/i=11/a=114004},
  year={2016}
}

@book{kuchment14book,
  title={The Radon Transform and Medical Imaging},
  author={Kuchment, P.},
  isbn={9781611973280},
  lccn={2013038338},
  series={CBMS-NSF Regional Conference Series in Applied Mathematics},
  url={http://books.google.co.kr/books?id=AbByAgAAQBAJ},
  year={2014},
  publisher={Society for Industrial and Applied Mathematics}
}

@article{palacios16,
doi = {10.1088/0266-5611/32/12/125008},
url = {https://doi.org/10.1088/0266-5611/32/12/125008},
year = {2016},
month = {nov},
publisher = {IOP Publishing},
volume = {32},
number = {12},
pages = {125008},
author = {Palacios, B.},
title = {Reconstruction for multi-wave imaging in attenuating media with large damping coefficient},
journal = {Inverse Problems}
}

@article{liuu15,
doi = {10.1088/0266-5611/31/10/105005},
url = {https://dx.doi.org/10.1088/0266-5611/31/10/105005},
year = {2015},
month = {sep},
publisher = {IOP Publishing},
volume = {31},
number = {10},
pages = {105005},
author = {Liu, H. and Uhlmann, G.},
title = {Determining both sound speed and internal source in thermo- and photo-acoustic tomography},
journal = {Inverse Problems}
}

@article{knoxm20,
doi = {10.1088/1361-6420/ab53fc},
url = {https://dx.doi.org/10.1088/1361-6420/ab53fc},
year = {2020},
month = {jan},
publisher = {IOP Publishing},
volume = {36},
number = {2},
pages = {025002},
author = {Knox, C. and Moradifam, A.},
title = {Determining both the source of a wave and its speed in a medium from boundary measurements},
journal = {Inverse Problems}
}

@article{moonsiims23,
  title={Singular value decomposition of the wave forward operator with radial variable coefficients},
  author={Moon, M. and Hur, I. and Moon, S.},
  journal={SIAM Journal on Imaging Sciences},
  volume={16},
  number={3},
  pages={1520--1534},
  year={2023},
  publisher={SIAM}
}

@article{nguyenhkd22,
author = {Nguyen, L. and Haltmeier, M. and Kowar, R. and Do, N.},
title = {Analysis for Full-Field Photoacoustic Tomography with Variable Sound Speed},
journal = {SIAM Journal on Imaging Sciences},
volume = {15},
number = {3},
pages = {1213-1228},
year = {2022},
doi = {10.1137/21M1463409},

URL = { 
    
        https://doi.org/10.1137/21M1463409
    
    

},
eprint = { 
    
        https://doi.org/10.1137/21M1463409
    
    

}
}

@article{hwangjmp24,
  title={Implicit learning to determine variable sound speed and the reconstruction operator in photoacoustic tomography},
  author={Hwang, G. and Jeon, G. and Moon, S. and Park, D.},
  journal={arXiv preprint arXiv:2407.09749},
  year={2024}
}

@article{hristovakn08,
doi = {10.1088/0266-5611/24/5/055006},
url = {https://doi.org/10.1088/0266-5611/24/5/055006},
year = {2008},
month = {aug},
publisher = {},
volume = {24},
number = {5},
pages = {055006},
author = {Hristova, Y. and Kuchment, P. and Nguyen, L.},
title = {Reconstruction and time reversal in thermoacoustic tomography in acoustically homogeneous and inhomogeneous media},
journal = {Inverse Problems}
}

@article{oksanenu14,
  title={Photoacoustic and thermoacoustic tomography with an uncertain wave speed},
  author={Oksanen, L. and Uhlmann, G.},
  journal={Mathematical Research Letters},
  volume={21},
  number={5},
  pages={1199--1214},
  year={2014},
  publisher={International Press}
}

@article{stefanovu09,
  author={P. Stefanov and G. Uhlmann},
  title={Thermoacoustic tomography with variable sound speed},
  journal={Inverse Problems},
  volume={25},
  number={7},
  pages={075011},
  url={http://stacks.iop.org/0266-5611/25/i=7/a=075011},
  year={2009}
}

@book{wang09,
  title={Photoacoustic Imaging and Spectroscopy},
  author={Wang, L.V.},
  isbn={9781420059922},
  lccn={2009002051},
  series={Optical Science and Engineering},
  url={http://books.google.co.kr/books?id=QzdHEmUFcXgC},
  year={2009},
  publisher={Taylor \& Francis}
}

@book{Enflo2002,
  title={Theory of nonlinear acoustics in fluids},
  author={Enflo, Bengt O and Hedberg, Claes M},
  volume={67},
  year={2002},
  publisher={Springer Science \& Business Media}
}

@article{Lions1969,
  title={Quelque methodes de r{\'e}solution des problemes aux limites non lin{\'e}aires, Paris, Dunod-Gauth},
  author={Lions, JL},
  journal={Vill. MR},
  volume={259693},
  year={1969}
}

@article{Claret2024,
  title={An observability estimate for the wave equation and applications to the Neumann boundary controllability for semi-linear wave equations},
  author={Claret, Sue},
  journal={arXiv preprint arXiv:2409.07214},
  year={2024}
}

@incollection{Ammari2012,
  title={Photoacoustic imaging for attenuating acoustic media},
  author={Ammari, Habib and Bretin, Elie and Jugnon, Vincent and Wahab, Abdul},
  booktitle={Mathematical modeling in biomedical imaging II: optical, ultrasound, and opto-acoustic tomographies},
  pages={57--84},
  year={2012},
  publisher={Springer}
}

@article{gupta2020sparse,
  title   = {Sparse reconstruction of log-conductivity in current density impedance tomography},
  author  = {Gupta, M. and Mishra, R. K. and Roy, S.},
  journal = {Journal of Mathematical Imaging and Vision},
  volume  = {62},
  number  = {2},
  pages   = {189--205},
  year    = {2020},
  publisher = {Springer US},
  doi     = {10.1007/s10851-019-00939-7}
}

@article{gupta2021sparsity,
  title   = {Sparsity-based nonlinear reconstruction of optical parameters in two-photon photoacoustic computed tomography},
  author  = {Gupta, M. and Mishra, R. K. and Roy, S.},
  journal = {Inverse Problems},
  volume  = {37},
  number  = {4},
  pages   = {044001},
  year    = {2021},
  publisher = {IOP Publishing},
  doi     = {10.1088/1361-6420/abe46f}
}

@article{adesokan2018nonlinear,
  title   = {A fully non-linear optimization approach to acousto-electric tomography},
  author  = {Adesokan, B. J. and Knudsen, K. and Krishnan, V. P. and Roy, S.},
  journal = {Inverse Problems},
  volume  = {34},
  number  = {10},
  pages   = {104004},
  year    = {2018},
  publisher = {IOP Publishing},
  doi     = {10.1088/1361-6420/aadf2f}
}

@article{roy2018optimization,
  title   = {A new optimization approach to sparse reconstruction of log-conductivity in acousto-electric tomography},
  author  = {Roy, S. and Borz{\`\i}, A.},
  journal = {SIAM Journal on Imaging Sciences},
  volume  = {11},
  number  = {2},
  pages   = {1759--1784},
  year    = {2018},
  publisher = {Society for Industrial and Applied Mathematics},
  doi     = {10.1137/17M1133954}
}

@article{duck2002nonlinear,
  title   = {Nonlinear acoustics in diagnostic ultrasound},
  author  = {Duck, Francis A.},
  journal = {Ultrasonics},
  volume  = {40},
  number  = {1--8},
  pages   = {9--16},
  year    = {2002},
  publisher = {Elsevier},
  doi     = {10.1016/S0041-624X(02)00174-5}
}

\end{document}